\documentclass[11pt,twoside]{article}
\usepackage{amssymb,amsmath,amsthm,amsfonts,mathrsfs,hyperref}
\usepackage{times}
\usepackage{enumerate}
\usepackage{cite,titletoc}
\usepackage{color}
\usepackage[toc,page,title,titletoc,header]{appendix}
\usepackage{multirow}

\allowdisplaybreaks

\def\cA{\mathcal A}
\def\cB{\mathcal B}

\def\cK{\mathcal K}

\def\cO{\mathcal O}

\def\cS{\mathcal S}

\def\cW{\mathcal W}

\def\na{\nabla}

\def\N{\mathop{\mathbb N\kern 0pt}\nolimits}
\def\Z{\mathop{\mathbb Z\kern 0pt}\nolimits}
\def\Q{\mathop{\mathbb Q\kern 0pt}\nolimits}
\def\R{\mathop{\mathbb R\kern 0pt}\nolimits}
\def\T{\mathop{\mathbb T\kern 0pt}\nolimits}
\def\SS{\mathop{\mathbb S\kern 0pt}\nolimits}

\def\ds{\displaystyle}
\def\f{\frac}
\def\al{\alpha}
\def\supp{\mathop{\rm supp}\nolimits}

\def\p{\partial}

\def\ve{\varepsilon}

\def\dive{\operatorname{div}}
\def\curl{\operatorname{curl}}
\def\supp{\operatorname{supp}}
\def\ls{\lesssim}
\def\gt{\gtrsim}
\def\bn{\mathbf{n}}

\newcommand{\w}[1]{\langle {#1} \rangle}

\hypersetup{colorlinks=true,linkcolor=blue,citecolor=red,urlcolor=cyan}

\theoremstyle{plain}
\newtheorem{theorem}{Theorem}[section]

\newtheorem{lemma}[theorem]{Lemma}

\newtheorem{remark}{Remark}[section]

\theoremstyle{definition}

\numberwithin{equation}{section}

\title{Long time smooth solutions of 2-D quadratic quasilinear wave equations in exterior domains with Neumann boundary conditions}

\author{Fei Hou$^{1,*}$ \quad Huicheng Yin$^{2,*}$ \quad Meng Yuan$^{3,}$
    \footnote{Fei Hou (\texttt{fhou$@$nju.edu.cn}) and Huicheng Yin (\texttt{huicheng$@$nju.edu.cn}, \texttt{05407$@$njnu.edu.cn})
    are supported by the National key research and development program of China (No.2024YFA1013301).
    In addition, Fei Hou, Huicheng Yin and Meng Yuan (\texttt{ym$@$cczu.edu.cn}) are supported by the NSFC (No.~12571237, No.~12331007).}\\
    [12pt]{\small 1. School of Mathematics, Nanjing University, Nanjing, 210093, China}\\
    {\small 2. School of Mathematical Sciences, Nanjing Normal University, Nanjing, 210023, China}\\
    {\small 3. School of Computer Science and Artificial Intelligence, Aliyun School of Big Data,}\\
    {\small School of Software, Changzhou University, Changzhou, 213164, China}}

\begin{document}

\date{}
\maketitle
\thispagestyle{empty}

\begin{abstract}
For the 3-D quadratic quasilinear wave equations in exterior domains with Dirichlet or Neumann boundary conditions,
the global existence or the maximal existence time of small data smooth solutions have been established in the past.
However, so far it is still open for the corresponding 2-D Neumann boundary value problem.
In this paper, we investigate the long time existence of small data solutions to 2-D quadratic
quasilinear wave equations with homogeneous Neumann boundary values.
Our main ingredients include: establishing some new pointwise spacetime decay estimates for the 2-D initial boundary value
problem of the divergence form wave equations, and introducing a series of good unknowns to derive the required energy estimates.
The obtained results can be directly applied to the initial boundary value problem of 2-D isentropic and
irrotational compressible Euler equations
for both the polytropic gases and the Chaplygin gases in exterior domains with impermeable conditions,
the 2-D relativistic membrane equations and  2-D membrane equations with  homogeneous Neumann boundary values.

\vskip 0.2 true cm


\qquad\quad   weighted energy estimate, compressible Euler equations, relativistic membrane equation

\vskip 0.2 true cm
\noindent
\end{abstract}

\vskip 0.2 true cm

\addtocontents{toc}{\protect\thispagestyle{empty}}
\tableofcontents

\section{Introduction}

\subsection{Main results and remarks}

In this paper, we are concerned with the following IBVP (initial boundary value problem) of 2-D quadratic
quasilinear wave equations in exterior domains with homogeneous Neumann boundary conditions
\begin{equation}\label{QWE}
\left\{
\begin{aligned}
&\Box u=Q(\p u,\p^2u),\hspace{2.6cm} (t,x)\in[0,+\infty)\times\cK,\\
&\frac{\p}{\p\bn}u(t,x)=0,\hspace{3.2cm} (t,x)\in[0,+\infty)\times\p\cK,\\
&(u,\p_tu)(0,x)=(\ve u_0,\ve u_1)(x),\qquad x\in\cK,
\end{aligned}
\right.
\end{equation}
where $\p=(\p_0,\p_1,\p_2)=(\p_t,\p_{x_1},\p_{x_2})$, $\ds\Box=\p_t^2-\Delta$ with $\Delta=\p_1^2+\p_2^2$, $\cK=\R^2\setminus\cO$,
the obstacle $\cO\subset\R^2$ is compact and contains the origin, the boundary $\p\cK=\p\cO$ is smooth, $\bn=(0, n_1,n_2)=(0, n_1(x),n_2(x))$ represents the unit outer normal direction of the boundary $[0,+\infty)\times\p\cK$,
$(u_0,u_1)\in C^\infty(\cK)$ and $\supp(u_0,u_1)\subset\{x\in\cK:|x|\le M_0\}$ with some fixed constant $M_0>1$.
In addition,
\begin{equation}\label{YHCC-1}
\begin{split}
Q(\p u,\p^2u)&=\sum_{\alpha,\beta,\mu=0}^2Q^{\alpha\beta\mu}\p^2_{\alpha\beta}u\p_{\mu}u
+\sum_{\alpha,\beta,\mu,\nu=0}^2Q^{\alpha\beta\mu\nu}\p^2_{\alpha\beta}u\p_{\mu}u\p_{\nu}u,
\end{split}
\end{equation}
where the constants $Q^{\alpha\beta\mu}=Q^{\beta\alpha\mu}$ and $Q^{\alpha\beta\mu\nu}=Q^{\beta\alpha\mu\nu}=Q^{\alpha\beta\nu\mu}$.
On the other hand, as in condition (1.8) of \cite{MSS}, the following admissible condition on $Q(\p u,\p^2u)$ for Neumann boundary
value problem \eqref{QWE} is naturally imposed:
\begin{equation}\label{YHCC-2}
\begin{split}
&\text{For any $C^1(\Bbb R^2)$ smooth functions $w$ and $\psi$ verifying $\frac{\p}{\p\bn}w
=\frac{\p}{\p\bn}\psi=0$}\\
&\text{on $[0,+\infty)\times\p\cK$, it holds that on $[0,+\infty)\times\p\cK$},\\
&{\sum_{\alpha,\beta,\mu=0}^2Q^{\alpha\beta\mu}n_{\alpha}\p_{\beta}w\p_{\mu}\psi=\sum_{\alpha,\beta,\mu,\nu=0}^2
Q^{\alpha\beta\mu\nu}n_{\alpha}\p_{\beta}w\p_{\mu}\psi\p_{\nu}\psi\equiv 0.}
\end{split}
\end{equation}

As usual, in order to find the smooth solutions of problem \eqref{QWE}, the compatibility conditions of $(u_0,u_1)$
with $\frac{\p}{\p\bn}u(t,x)=0$ for $(t,x)\in[0,+\infty)\times\p\cK$ are necessary. That is,
let $J_ku=\{\p_x^{\alpha}u:0\le|\alpha|\le k\}$
and $\p_t^ku(0,x)=G_k(J_ku_0,J_{k-1}u_1)$ ($0\le k\le 2N$), where $G_k$ depends on $Q(\p u, \p^2u)$, $J_ku_0$ and $J_{k-1}u_1$
in terms of the equation and the initial data in problem \eqref{QWE},
then the compatibility conditions up to $(2N+1)-$order mean that
\begin{equation}\label{YHCC-3}
\begin{split}
&\text{$\frac{\p}{\p\bn}G_k$ vanish on $\p\cK$ for all $0\le k\le 2N$.}
\end{split}
\end{equation}

Our main results can be stated as
\begin{theorem}\label{thm1}
Suppose that the obstacle $\cO$ is convex and \eqref{YHCC-1}-\eqref{YHCC-3} hold.
Then for any fixed constant $\delta\in(0,1)$, there is a constant $\ve_0>0$ such that when $\ve\le\ve_0$ and
\begin{equation}\label{initial:data-thm1}
\|u_0\|_{H^{2N+1}(\cK)}+\|u_1\|_{H^{2N}(\cK)}\le1 \quad\text{with $N\ge19$},
\end{equation}
problem \eqref{QWE} has a smooth solution $u\in\bigcap\limits_{j=0}^{2N+1}C^{j}([0,\ve^{\delta-2}], H^{2N+1-j}(\cK))$.
\end{theorem}

\begin{theorem}\label{thm2}
Suppose that the obstacle $\cO$ is convex and  \eqref{YHCC-1}-\eqref{YHCC-3} hold. Especially,
$Q(\p u,\p^2u)$ admits the following form
\begin{equation}\label{nonlinear-Chaplygin}
\begin{split}
Q(\p u,\p^2u)&=2\sum_{\alpha=0}^2C^{\alpha}Q_0(\p_{\alpha}u,u)+\sum_{\alpha,\beta,\mu,\nu=0}^2
Q^{\alpha\beta\mu\nu}\p^2_{\alpha\beta}u\p_{\mu}u\p_{\nu}u,\\
\end{split}
\end{equation}
where $C^{\alpha}$, $Q^{\alpha\beta\mu\nu}$ are constants and $\ds Q_0(f,h)=\p_tf\p_th-\sum_{i=1}^2\p_if\p_ih$.
Meanwhile, the second null condition holds:
\begin{equation}\label{null:condition}
\sum_{\alpha,\beta,\mu,\nu=0}^2Q^{\alpha\beta\mu\nu}\xi_{\alpha}\xi_{\beta}\xi_{\mu}\xi_{\nu}\equiv0\quad\text{for any  $(\xi_0,\xi_1,\xi_2)\in\{\pm1\}\times\SS$}.
\end{equation}
Then there is a constant $\ve_0>0$ such that when $\ve\le\ve_0$ and
\begin{equation}\label{initial:data}
\|u_0\|_{H^{2N+1}(\cK)}+\|u_1\|_{H^{2N}(\cK)}\le1 \quad\text{with $N\ge40$},
\end{equation}
problem \eqref{QWE} has a global solution $u\in\bigcap\limits_{j=0}^{2N+1}C^{j}([0,\infty), H^{2N+1-j}(\cK))$.
Moreover,
\addtocounter{equation}{1}
\begin{align}
\sum_{|a|\le N+1}|\p Z^au|&\le C\ve\w{x}^{-1/2}\w{t-|x|}^{-1}\ln^2(2+t),\tag{\theequation a}\label{thm2:decay:a}\\
\sum_{|a|\le N+1}\sum_{i=1}^2|\bar\p_iZ^au|&\le C\ve\w{x}^{-1/2}\w{t+|x|}^{0.001-1},\tag{\theequation b}\label{thm2:decay:b}\\
\sum_{|a|\le N+1}|Z^au|&\le C\ve\w{t+|x|}^{-1/2}\w{t-|x|}^{0.001-1/2},\tag{\theequation c}\label{thm2:decay:c}\\
|\p\p_tu|&\le C\ve\w{x}^{-1/2}\w{t-|x|}^{-1},\tag{\theequation d}\label{thm2:decay:d}
\end{align}
where $Z=\{\p,\Omega\}$ with $\Omega=x_1\p_2-x_2\p_1$, $\w{x}=\sqrt{1+|x|^2}$,  $\bar\p_i=\frac{x_i}{|x|}\p_t+\p_i$ ($i=1,2$)
are the good derivatives (tangent to the outgoing light cone $|x|=t$).
Furthermore, the following time decay estimates of local energy hold
\addtocounter{equation}{1}
\begin{align}
\sum_{|a|\le2N-29}\|\p^au\|_{L^2(\cK_R)}\le C_R\ve(1+t)^{-1}\ln(2+t),\tag{\theequation a}\label{thm2:decay:LE}\\
\sum_{|a|\le2N-30}\|\p^a\p_tu\|_{L^2(\cK_R)}\le C_R\ve(1+t)^{-1},\tag{\theequation b}\label{thm2:decay:LEdt}
\end{align}
where $R>1$ is a fixed constant, $\cK_R=\cK\cap\{x: |x|\le R\}$ and  $C_R>0$ is a constant
depending on $R$.
\end{theorem}

\begin{remark}\label{YinA-1}
For the Cauchy problem of the 2-D quasilinear wave equation in \eqref{QWE} without the first null condition
(i.e., $\ds\sum_{\alpha,\beta,\mu=0}^2Q^{\alpha\beta\mu}\xi_{\alpha}\xi_{\beta}\xi_{\mu}\not\equiv0$ for
$(\xi_0,\xi_1,\xi_2)\in\{\pm1\}\times\SS$), the lifespan $T_\ve$ (maximal existence time) of the smooth solution fulfills $T_\ve\ge C/\ve^2$ for
some suitable constant $C>0$ (see Chapter 6 of \cite{Hormander97book}).
Note that the lifespan $T_\ve$ of the 2-D Cauchy-Neumann problem in Theorem \ref{thm1} satisfies $T_\ve\ge\frac{1}{\ve^{2-}}$,
which is only arbitrarily close to that of the Cauchy problem.
In our forthcoming paper, we intend to show that the lifespan $T_\ve$ of the 2-D quadratic quasilinear wave equation in exterior domain
with both the Dirichlet and Neumann boundary conditions is at least $C/\ve^2$ by constructing delicate approximate solutions and establishing
some more precise energy estimates.
\end{remark}

\begin{remark}\label{YinA-2}
In \cite{HouYinYuan24}, for the following Cauchy-Dirichlet problem of 2-D
quadratic quasilinear wave equation in exterior domain with the homogeneous Dirichlet boundary condition
\begin{equation}\label{QWE-D}
\left\{
\begin{aligned}
&\Box u=\sum_{\alpha,\beta,\mu=0}^2Q^{\alpha\beta\mu}\p^2_{\alpha\beta}u\p_{\mu}u,\hspace{1cm} (t,x)\in[0,+\infty)\times\cK,\\
&u(t,x)=0,\hspace{3.8cm} (t,x)\in[0,+\infty)\times\p\cK,\\
&(u,\p_tu)(0,x)=(\ve u_0,\ve u_1)(x),\qquad x\in\cK,
\end{aligned}
\right.
\end{equation}
when the first null condition is satisfied (i.e., $\ds\sum_{\alpha,\beta,\mu=0}^2Q^{\alpha\beta\mu}\xi_{\alpha}\xi_{\beta}\xi_{\mu}\equiv0$
for any $(\xi_0,\xi_1,\xi_2)\in\{\pm1\}\times\SS$), we have established the global existence of smooth solution $u$.
In this case, it is well known from \cite{Klainerman} (also see \cite[Page 79]{MetcalfeSogge05} or \cite[Remark 1.1]{HouYinYuan24}) that
the nonlinearity in \eqref{QWE-D} admits such a form
\begin{equation}\label{QWE-D-NL}
\sum_{\alpha,\beta,\mu=0}^2Q^{\alpha\beta\mu}\p^2_{\alpha\beta}u\p_{\mu}u
=\sum_{\alpha=0}^2C^{0,\alpha}Q_0(\p_{\alpha}u,u)
+\sum_{\alpha,\mu,\nu=0}^2C^{\mu\nu,\alpha}Q_{\mu\nu}(\p_{\alpha}u,u),
\end{equation}
where $C^{0,\alpha}$ and $C^{\mu\nu,\alpha}$ are constants,  the null forms $Q_{\mu\nu}(\p_{\alpha}u,u)$ are given by
\begin{equation}\label{Intro:nullform}
\begin{split}
Q_{\mu\nu}(f,g)&=\p_{\mu}f\p_{\nu}g-\p_{\nu}f\p_{\mu}g,\quad \mu,\nu=0,1,2.
\end{split}
\end{equation}
In Theorem \ref{thm2}, we have shown the global existence of small data smooth solution $u$ to
the Cauchy-Neumann problem of 2-D quadratic wave equation with $Q_0$-type null forms in exterior domain.
Due to the complexity of $Q_{\mu\nu}(\p_{\alpha}u,u)$ null forms in Neumann boundary condition
(a series of good unknowns are only introduced for the $Q_0$-type form in this paper),
the global existence on the Cauchy-Neumann problem of 2-D $Q_{\mu\nu}$-type null form wave equation in exterior domain
is not solved until now.
\end{remark}

\begin{remark}\label{YinA-3}
In \cite{Secchi-Shibata03}, the authors have obtained  the following time decay estimate
\begin{equation}\label{SecShi-decay}
\|\p u_{lin}\|_{L^\infty(\cK)}\le C(1+t)^{-1/2}\ln^2(2+t)(\|\tilde u_0\|_{W^{6,1}}+\|\tilde u_1\|_{W^{5,1}}),
\end{equation}
where $C>0$, $u_{lin}$ is the solution to the 2-D linear problem $\Box u_{lin}=0$ with
initial data $(u_{lin},\p_tu_{lin})|_{t=0}=(\tilde u_0,\tilde u_1)$ and the Neumann
boundary condition $\frac{\p}{\p\bn}u_{lin}|_{[0,\infty)\times\p\cK}=0$.
Obviously, because of $\w{x}^{-1/2}\w{t-|x|}^{-1}$ $\le (1+t)^{-\f12}\w{t-|x|}^{-\f12}$,
the spacetime pointwise estimate \eqref{thm2:decay:a} of the nonlinear problem \eqref{QWE}
is more precise than \eqref{SecShi-decay} for the corresponding linear problem.
\end{remark}

\begin{remark}\label{rmk-decay}
For the 2-D quadratic
quasilinear wave equations
in exterior domains with homogeneous Dirichlet boundary conditions
\begin{equation}\label{QWE-00}
\left\{
\begin{aligned}
&\Box w=Q(\p w,\p^2w),\hspace{2.6cm} (t,x)\in[0,+\infty)\times\cK,\\
&w(t,x)=0,\hspace{3.9cm} (t,x)\in[0,+\infty)\times\p\cK,\\
&(w,\p_tw)(0,x)=(\ve w_0,\ve w_1)(x),\qquad x\in\cK,
\end{aligned}
\right.
\end{equation}
we have established the following crucial pointwise spacetime decay estimate and local energy decay estimate (see (1.5a) and (1.6) of \cite{HouYinYuan24}, respectively)
\begin{equation}\label{YHCC-20-6}
\sum_{|a|\le N+1}|\p Z^aw|\le C\ve\w{x}^{-1/2}\w{t-|x|}^{-1}
\end{equation}
and
\begin{equation}\label{Dirichlet-decay}
\|\p^aw\|_{L^2(\cK_R)}\le C_R\ve(1+t)^{-1},
\end{equation}
where $R>1$ is a fixed constant, $\cK_R=\cK\cap\{x: |x|\le R\}$ and  $C_R>0$ is a constant
depending on $R$. However, for the Cauchy-Neumann problem \eqref{QWE}, the corresponding pointwise spacetime decay
estimate in \eqref{thm2:decay:a}
and local energy decay estimate in \eqref{thm2:decay:LE} are
\begin{equation}\label{YHCC-20-7}
\sum_{|a|\le N+1}|\p Z^au|\le C\ve\w{x}^{-1/2}\w{t-|x|}^{-1}\ln^2(2+t)
\end{equation}
and
\begin{equation}\label{YinA-4}
\|\p^au\|_{L^2(\cK_R)}\le C_R\ve(1+t)^{-1}\ln(2+t),
\end{equation}
respectively. Due to the appearances of the large factors $\ln^2(2+t)$ in \eqref{YHCC-20-7}
and $\ln(2+t)$ in \eqref{YinA-4},
this leads to many crucial difficulties in the proof of Theorem \ref{thm2} so that a series of new good unknowns are
introduced to obtain the required energy estimates, which is essentially different from
that in \cite{HouYinYuan24} for the 2-D quasilinear Cauchy-Dirichlet problem \eqref{QWE-00}. One can see
more detailed explanations in Subsection \ref{YinA-5} below.

On the other hand, we next give a brief illustration on the sharpness of \eqref{YinA-4}.
In fact, for the free linear wave equations in the exterior domains with homogenous Neumann boundary value,
the decay rate of the local energy is $(1+t)^{-1}$ (see \cite{Kubo13,Morawetz75,Vainberg75}).
When handling the nonlinear problem \eqref{QWE}, the cutoff method and the Duhamel's principle will be applied as usual.
As a result, one can meet such a nonlinear Klein-Gordon equation $\square u+b(x)u={\tilde Q}(\p u, \p^2 u)$
including the variable coefficient $b(x)$ and needs to control the solution $u$ itself near the boundary with local energy decay rate being $(1+t)^{-1}$.
In this case, as illustrated in page 184 of Chapter 7 \cite{Hormander97book}, by the standard energy method for the linear Klein-Gordon
equations and the Gronwall's lemma,
the decay rate of the local energy to nonlinear problem \eqref{QWE} is
\begin{align*}
&\int_0^t(1+t-s)^{-1}(1+s)^{-1}ds=2(2+t)^{-1}\ln(1+t).
\end{align*}
\end{remark}

\subsection{Applications of main results}

\subsubsection{Application to 2-D isentropic
and irrotational compressible Euler equations}\label{CYJ-1}

As the first application of Theorems \ref{thm1} and \ref{thm2}, we now derive the lifespan
of smooth solutions to 2-D isentropic
and irrotational compressible Euler equations for both the polytropic gases and the Chaplygin
gases in exterior domains with impermeable boundary conditions and small perturbed initial data.

The 2-D compressible isentropic Euler equations are
\begin{equation}\label{Euler}
\left\{
\begin{aligned}
&\p_t\rho+\dive (\rho v)=0\hspace{3cm}\text{(Conservation of mass)},\\
&\p_t(\rho v)+\dive (\rho v \otimes v)+\nabla p=0
\hspace{0.9cm}\text{(Conservation of momentum)},\\
\end{aligned}
\right.
\end{equation}
where $(t,x)=(t, x_1, x_2)$, $\nabla=(\p_1, \p_2)=(\p_{x_1}, \p_{x_2})$,
and $v=(v_1,v_2),~\rho,~p$ stand for the velocity, density, pressure, respectively.
In addition, the pressure $p=p(\rho)$ is a smooth
function of $\rho$ for $\rho>0$.
Moreover, $p'(\rho)>0$ as $\rho>0$.

For the polytropic gases (see \cite{CF}),
\begin{equation}\label{polytropicGas}
    p(\rho)=A\rho^\gamma,
\end{equation}
where $A$ and $\gamma$ ($1<\gamma<3$) are positive constants.

For the Chaplygin gases (see \cite{CF} or \cite{Godin07}),
\begin{equation}\label{ChaplyginGas}
    p(\rho)=B_1-\frac{B_2}{\rho},
\end{equation}
where $B_1>0$ and $B_2>0$ are constants.

If $(\rho, v)\in C^1$ is a solution of \eqref{Euler} with $\rho>0$,
then \eqref{Euler} is equivalent to the following form
\begin{equation}\label{EulerC1form}
\left\{
\begin{aligned}
&\p_t\rho+\dive (\rho v)=0,\\
&\p_tv+v\cdot\nabla v+\ds\frac{c^2(\rho)}{\rho}\nabla \rho=0,\\
\end{aligned}
\right.
\end{equation}
where the sound speed $c(\rho)=\sqrt{p'(\rho)}$.

We impose the perturbed initial boundary data of \eqref{EulerC1form} in exterior domain $[0,\infty)\times \cK$ as follows:
\begin{equation}\label{initial}
(\rho(0,x), v(0,x))=(\bar\rho+\ve\rho_0(x), \ve v_0(x)),\quad x\in\cK,
\end{equation}
and
\begin{equation}\label{YinA-110}
(0,v)\cdot\bn=0, \quad (t,x)\in[0,+\infty)\times\p\cK,
\end{equation}
where $\bar\rho>0$ is a constant, $\ve>0$ is small such that $\bar\rho+\ve\rho_0(x)>0$,
$\rho_0(x), v_0(x)=(v_{1,0}(x), v_{2,0}(x))\in C_0^{\infty}(\cK)$,
$\supp \rho_0$ and $\supp v_0$ are contained in the ball $B(0, M_0)$ with $M_0>1$,
$\cK=\R^2\setminus\cO$ with
the bounded obstacle $\cO$ being convex and smooth, $\bn$ is the unit outer normal direction of $[0,\infty)\times\p\cK$,
and $(0,v_0)\cdot\bn=0$ for $x\in\p\cK$.  When $\curl v_0(x):=\p_1 v_{2,0}-\p_2 v_{1,0}\equiv 0$,
as long as $(\rho, v)\in C^1$ for $0\le t<T$, then $\curl v\equiv 0$
always holds for $0\le t<T$, where $T>0$ is any fixed constant or $T=+\infty$. In this case, one can introduce
the potential function $\phi$ such that $v=\na \phi$,
then the Bernoulli's law implies 
\begin{equation}\label{Feng-1}
\p_t\phi+\f12|\na\phi|^2+h(\rho)=0
\end{equation}
with $h'(\rho)=\f{c^2(\rho)}{\rho}$ and $h(\bar\rho)=0$.
By the implicit function theorem due to
$h'(\rho)>0$ for $\rho>0$, the density
function $\rho$ can be expressed as
\begin{equation}\label{density:potential}
\rho=h^{-1}\big(-\p_t\phi-\frac{1}{2}|\na\phi|^2\big)
=:H(\p\phi),
\end{equation}
where $\p=(\p_t,\nabla).$ Substituting \eqref{density:potential} into the mass conservation equation in \eqref{Euler}
yields
\begin{equation}\label{mass:potential}
\p_t(H(\p\phi))+\ds\sum_{i=1}^2\p_i\bigl(H(\p\phi)\p_i\phi\bigr)=0.
\end{equation}
Without loss of generality, $\bar\rho=c(\bar\rho)=1$ is assumed,  then it holds that $h(\rho)=\f{\rho^{\gamma-1}-1}{\gamma-1}$ and
\begin{equation}\label{YHH-12}
\rho=H(\p\phi)=[1-(\gamma-1)(\p_t\phi+\f12|\na\phi|^2)]^{\f{1}{\gamma-1}}.
\end{equation}
It follows from \eqref{mass:potential}-\eqref{YHH-12}  and  direct computation that for $(t,x)\in[0,+\infty)\times\cK$,
\begin{equation}\label{eqn:potential}
\begin{split}
\Box\phi&=-2\sum_{j=1}^2\p_j\phi\p^2_{0j}\phi-\sum_{i,j=1}^2\p_i\phi\p_j\phi\p^2_{ij}\phi
-(\gamma-1)(\p_t\phi+\frac12|\nabla\phi|^2)\Delta\phi\\
&=\sum_{\alpha,\beta,\mu=0}^2Q^{\alpha\beta\mu}\p^2_{\alpha\beta}\phi\p_{\mu}\phi
+\sum_{\alpha,\beta,\mu,\nu=0}^2Q^{\alpha\beta\mu\nu}\p^2_{\alpha\beta}\phi\p_{\mu}\phi\p_{\nu}\phi,
\end{split}
\end{equation}
which corresponds to the irrotational Euler equations of polytropic gases for $1<\gamma<3$ and Chaplygin gases for $\gamma=-1$, respectively.
Meanwhile, the boundary condition \eqref{YinA-110} becomes
\begin{equation}\label{YinA-1100}
\f{\p}{\p\bn}\phi=0, \quad (t,x)\in[0,+\infty)\times\p\cK.
\end{equation}
We next determine the initial data $\phi(0,x)$ and $\p_t\phi(0,x)$.
Let $r=\sqrt{x_1^2+x_2^2}$ and $(x_1,x_2)=(r cos\theta, r sin\theta)$ with $\theta\in [0, 2 \pi]$.
Due to $r\p_r\phi=x_1\p_1\phi+x_2\p_2\phi=x_1v_1(r cos\theta, r sin\theta)+x_2v_2(r cos\theta, r sin\theta)$
and the convex property of  $\cO$, then one has that for $x\in \cK$,
\begin{equation}\label{initial-00}
\begin{split}
\phi(0,x)=\ve\bigg[ \f{x_1}{r}\int_{M_0}^rv_{1,0}(\f{sx_1}{r}, \f{sx_2}{r})ds+ \f{x_2}{r}\int_{M_0}^rv_{2,0}(\f{sx_1}{r}, \f{sx_2}{r})ds\bigg]
\in C_0^{\infty}(\cK).
\end{split}
\end{equation}
In addition, it follows from \eqref{initial} and \eqref{Feng-1} that
\begin{equation}\label{initial-00-1}
\p_t\phi(0,x)=-\ve\rho_0(x)+\ve^2g(x,\ve),\quad x\in\cK,
\end{equation}
where $g(x,\ve)$ is smooth in its arguments and has a compact support in the variable $x$.

In addition, the admissible condition \eqref{YHCC-2} holds. In fact, for any smooth functions
$w$ and $\psi$ satisfying $\f{\p}{\p\bn}w|_{[0,\infty)\times\p\cK}=\f{\p}{\p\bn}\psi|_{[0,\infty)\times\p\cK}=0$,
it follows from \eqref{eqn:potential} and direct computation that for $(t,x)\in[0,+\infty)\times\p\cK$,
\begin{equation}\label{YHHC-111}
\begin{split}
&\sum_{\alpha,\beta,\mu=0}^2Q^{\alpha\beta\mu}n_{\alpha}\p_{\beta}w\p_{\mu}\psi=-\p_tw\f{\p}{\p\bn}\psi-(\gamma-1)\p_t\psi\f{\p}{\p\bn}w\equiv0,\\
&\sum_{\alpha,\beta,\mu,\nu=0}^2Q^{\alpha\beta\mu\nu}n_{\alpha}\p_{\beta}w\p_{\mu}\psi\p_{\nu}\psi=-\f{\gamma-1}{2}
\f{\p}{\p\bn}w\sum_{i=1}^2|\p_i\psi|^2-\f{\p}{\p\bn}\psi\sum_{j=1}^2\p_j\psi\p_jw\equiv0.
\end{split}
\end{equation}
On the other hand, the equation \eqref{eqn:potential} with $\gamma=-1$ can be reformulated as follows
\begin{equation}\label{eqn:potential1}
\begin{split}
\Box\phi&=2Q_0(\p_t\phi,\phi)-\sum_{i,j=1}^2\p_i\phi\p_j\phi\p^2_{ij}\phi+|\nabla\phi|^2\Delta\phi-2\p_t\phi\big(2\p_t\phi\Delta\phi
-2\sum_{j=1}^2\p_j\phi\p^2_{0j}\phi\big)\\
&\quad -2\p_t\phi\big(|\nabla\phi|^2\Delta\phi-\sum_{i,j=1}^2\p_i\phi\p_j\phi\p^2_{ij}\phi\big).
\end{split}
\end{equation}
This means that the quadratic term in the right hand side of \eqref{eqn:potential1} admits $Q_0$-type null form
and certainly satisfy the first null condition. Furthermore, it is easy to verify that the equation \eqref{eqn:potential1}
also fulfills the second null condition.

By analysis in the above, in terms of Theorems \ref{thm1} and \ref{thm2}, we have

\vskip 0.1 true cm

{\bf Corollary 1.1.} For the potential equation \eqref{eqn:potential} with \eqref{YinA-1100}-\eqref{initial-00-1},
under the compatibility condition \eqref{YHCC-3} for the initial boundary values \eqref{YinA-1100} and \eqref{initial-00}-\eqref{initial-00-1},
the lifespan $T_{\ve}$ of smooth solution $\phi$ satisfies
\begin{equation}\label{YHCC0-2}
\begin{split}
&\text{$T_\ve\ge\frac{1}{\ve^{2-}}$ for the polytropic gases,}\\
&\text{$T_\ve=\infty$ for the Chaplygin gases.}
\end{split}
\end{equation}

\subsubsection{Application to 3-D isentropic
and irrotational supersonic Euler equations}\label{CYJ-2}

The 3-D isentropic and steady compressible Euler equations are
\begin{equation}\label{EulerSystem-G}
\left\{
\begin{aligned}
&\displaystyle\sum_{j = 1}^{3} \partial_{j} \left( \rho v_{j} \right) = 0, \\
&\displaystyle\sum_{j = 1}^{3} \partial_{j} \left( \rho v_{i} v_{j} \right) + \partial_{i} p = 0 ,\quad i = 1, 2, 3,
\end{aligned}
\right.
\end{equation}
where $x = (x_1, x_2, x_3)$, $\partial_i= \partial_{x_i}$ for $1\le i\le 3$, $\rho>0$ is the density, $p=A\rho^{\gamma}$ is the pressure
of polytropic gases ($A>0$ and $\gamma>1$ are constants), $c(\rho)=\sqrt{p'(\rho)}$ stands for the sonic speed,
$v(x) = \left( v_1(x), v_2(x), v_3(x) \right)$ represents the velocity of gases,
$\omega=(\p_2v_3-\p_3v_2, \p_3v_1-\p_1v_3, \p_1v_2-\p_2v_1)\equiv 0$ and $v_3>c(\rho)$ hold
(this means that $x_3$ is the supersonic direction).

By the irrotational assumption for \eqref{EulerSystem-G}, there exists a potential function $\Phi(x)$
such that $v_i =\partial_i\Phi$ ($1\le i\le 3$). This, together with the equations in the second
line of \eqref{EulerSystem-G}, yields the following Bernoulli's law
\begin{equation}\label{YHCC-1-G}
\frac{1}{2} |\p\Phi|^2 + h(\rho) \equiv {\bar C}
\end{equation}
where $\p=(\p_1,\p_2, \p_3)$, $|\p\Phi|^2=(\p_1\Phi)^2+(\p_2\Phi)^2+(\p_3\Phi)^2$,
${\bar C}=\f12{\bar q}^2$ is the Bernoulli's constant of the uniform supersonic incoming flow
with the constant density $\bar\rho>0$ and the constant velocity $(0,0,\bar q)$ with $\bar q>c(\bar \rho)$,
and $h(\rho)$ represents the enthalpy satisfying $h'(\rho) = \frac{c^2(\rho)}{\rho}>0$ for $\rho>0$ and $h(\bar \rho)=0$.

Substituting \eqref{YHCC-1-G} into the first equation in \eqref{EulerSystem-G} and taking direct computation yield
\begin{equation}\label{PotentialEquation-G}
\sum_{i = 1}^{3}[\left(\partial_i\Phi\right)^2-c^2(\rho)]\partial_i^2\Phi + 2 \sum_{1 \leq i < j \leq 3} \partial_i \Phi \, \partial_j \Phi \, \partial_{ij}^2 \Phi = 0,
\end{equation}
where $\rho=h^{-1}(\bar C-\f12|\p\Phi|^2)$ and $h^{-1}$ is the inverse function of $h(\rho)$.
It is easy to verify that \eqref{PotentialEquation-G} is strictly hyperbolic with respect to the $x_3-$direction due to $\p_3\Phi>c(\rho)$.
As in Subsection \ref{CYJ-1}, without loss of generality and for the convenience of writing, $\bar\rho=c(\bar\rho)=1$ and $\bar q=\sqrt{2}$ are assumed,  then it holds that $h(\rho)=\f{\rho^{\gamma-1}-1}{\gamma-1}$ and
\begin{equation}\label{YHH-12-G}
\rho=H(\p\Phi)=\big[1+(\gamma-1)(1-\f12|\p\Phi|^2)\big]^{\f{1}{\gamma-1}}.
\end{equation}

In addition, the perturbed initial boundary data of \eqref{EulerSystem-G}
are given in the exterior domain $\cK\times [0,\infty)$:
\begin{equation}\label{initial-G}
(\rho, v)(x_1,x_2,0)=(1, 0, 0, \sqrt{2})+\ve (\rho_0(x_1,x_2), v_{1,0}(x_1,x_2), v_{2,0}(x_1,x_2),
v_{3,0}(x_1,x_2)),\quad (x_1,x_2)\in\cK,
\end{equation}
and
\begin{equation}\label{YinA-110-G}
v\cdot\bn=0, \quad x\in\p\cK\times [0,+\infty),
\end{equation}
where $\ve>0$ is small such that $1+\ve\rho_0(x_1,x_2)>0$,
$\rho_0(x_1,x_2), v_0(x_1,x_2)=(v_{1,0}, v_{2,0}, v_{3,0})$ $(x_1,x_2)\in C_0^{\infty}(\cK)$,
$\supp \rho_0$ and $\supp v_0$ are contained in the ball $B(0, M_0)$ with $M_0>1$,
$\cK=\R^2\setminus\cO$ with
the bounded obstacle $\cO$ being convex and smooth, $\bn$ is the unit outer normal direction of $\p\cK\times [0,+\infty)$,
$v_0\cdot\bn=0$ for $x\in\p\cK\times [0,+\infty)$, and $\omega|_{x_3=0}\equiv 0$. Set $\Phi=\sqrt{2}x_3+\phi$,
which implies $c^2(\rho)=1-\sqrt{2}(\gamma-1)\p_3\phi-\f{\gamma-1}{2}(\p\phi)^2$.
Then it follows from \eqref{PotentialEquation-G}-\eqref{YinA-110-G} and direct computation that
\begin{equation}\label{equationofphi-G}
\begin{cases}
&\p_3^2\phi-(\p_1^2\phi+\p_2^2\phi)=\sqrt{2}(\gamma-1)\p_3\phi(\p_1^2\phi+\p_2^2\phi)
-\sqrt{2}(\gamma+1)\p_3\phi\p_3^2\phi+2\sqrt{2}\ds\sum_{1 \leq i \leq 2}
\partial_i \phi\partial_{i3}^2 \phi\\
&\quad -\big(\f{\gamma-1}{2}((\p_1\phi)^2+(\p_2\phi)^2)+
\f{\gamma+1}{2}(\p_3\phi)^2\big)\p_3^2\phi+\big(\f{\gamma-3}{2}(\p_1\phi)^2+\f{\gamma-1}{2}((\p_2\phi)^2)+
(\p_3\phi)^2)\big)\p_1^2\phi\\
&\quad +\big(\f{\gamma-3}{2}(\p_2\phi)^2+\f{\gamma-1}{2}((\p_1\phi)^2)+
(\p_3\phi)^2)\big)\p_2^2\phi-2\partial_1\phi\partial_2\phi \partial_{12}^2 \phi-2\p_3\phi
\ds\sum_{1 \leq i \leq 2}
\partial_i\phi \partial_{i3}^2 \phi\\
&\quad \equiv\ds\sum_{\alpha,\beta,\mu=1}^2Q^{\alpha\beta\mu}\p^2_{\alpha\beta}\phi\p_{\mu}\phi
+\ds\sum_{\alpha,\beta,\mu,\nu=1}^3Q^{\alpha\beta\mu\nu}\p^2_{\alpha\beta}\phi\p_{\mu}\phi\p_{\nu}\phi,\\
&\f{\p}{\p\bn}\phi=0, \qquad \qquad \qquad \qquad \qquad \qquad \qquad \qquad \qquad \qquad \qquad \qquad x\in\p\cK\times[0,+\infty),\\
&\phi|_{x_3 = 0} =\tilde{\phi}_0(x_1,x_2), \quad
\partial_3 \phi|_{x_3 = 0} = v_{3,0}(x_1,x_2),\qquad \qquad \qquad \qquad (x_1,x_2)\in\cK,
\end{cases}
\end{equation}
where $\tilde{\phi}_0(x_1,x_2)=\ve\ds\sum_{i=1}^2\f{x_i}{r}\int_{M_0}^{r}v_{i,0}(\f{sx_1}{r}, \f{sx_2}{r})ds$ with $r=\sqrt{x_1^2+x_2^2}$.
Meanwhile, the admissible condition \eqref{YHCC-2} holds for the nonlinear wave equation in \eqref{equationofphi-G}.
Indeed, for any smooth functions
$w$ and $\psi$ satisfying $\f{\p}{\p\bn}w|_{\p\cK\times[0,\infty)}=\f{\p}{\p\bn}\psi|_{\p\cK\times[0,\infty)}=0$,
it follows from \eqref{equationofphi-G} and direct computation that for $x\in\p\cK\times[0,\infty)$,
\begin{equation}\label{YHHC-111-G}
\begin{split}
&\sum_{\alpha,\beta,\mu=1}^3Q^{\alpha\beta\mu}n_{\alpha}\p_{\beta}w\p_{\mu}\psi=\sqrt{2}(\gamma-1)\p_3\psi\f{\p}{\p\bn}w
+\sqrt{2}\p_3w\f{\p}{\p\bn}\psi\equiv0,\\
&\sum_{\alpha,\beta,\mu,\nu=1}^3Q^{\alpha\beta\mu\nu}n_{\alpha}\p_{\beta}w\p_{\mu}\psi\p_{\nu}\psi=\f{\gamma-1}{2}
\f{\p}{\p\bn}w\sum_{i=1}^3|\p_i\psi|^2-\f{\p}{\p\bn}\psi\sum_{j=1}^3\p_j\psi\p_jw
\equiv0.
\end{split}
\end{equation}
From the analysis above, we have from Theorem \ref{thm1} that

\vskip 0.1 true cm

{\bf Corollary 1.2.} For problem \eqref{equationofphi-G},
under the compatibility condition \eqref{YHCC-3} with respect to the initial-boundary values of \eqref{equationofphi-G},
the lifespan $X_{\ve}$ of smooth solution $\phi$ in the $x_3-$direction satisfies
\begin{equation}\label{YHCC0-2-G}
\begin{split}
&X_\ve\ge\frac{1}{\ve^{2-}}.\\
\end{split}
\end{equation}

\subsubsection{Application to 2-D relativistic membrane equation}

2-D relativistic membrane equation has the following form
\begin{equation}\label{HCC-X}
\p_t\bigg(\ds\f{\p_t\phi}{\sqrt{1-(\p_t\phi)^2+|\nabla\phi|^2}}\bigg)
-div\bigg(\ds\f{\nabla\phi}{\sqrt{1-(\p_t\phi)^2+|\nabla\phi|^2}}\bigg)=0,
\end{equation}
which corresponds to the Euler-Lagrange equation of the area functional
$\int_{\Bbb R}\int_{\Bbb R^2}\sqrt{1+|\nabla \phi|^2-(\p_t\phi)^2}\\dtdx_1dx_2$
for the embedding of $(t,x)\to (t,x,\phi(t,x))$ in the Minkowski spacetime.
Here $x=(x_1,x_2)$ and $\nabla=(\p_{x_1}, \p_{x_2})$.

\eqref{HCC-X} is actually equivalent to the following nonlinear wave equation for the $C^2$ solution $\phi$
\begin{equation}\label{HCC-Y}
\begin{split}
\Box\phi&=(|\p_t\phi|^2-|\nabla\phi|^2)(\p_t^2\phi-\Delta\phi)
-(\p_t\phi)^2\p_t^2\phi+2\p_t\phi\nabla\phi\cdot\p_t\nabla\phi-\f12\nabla\phi\cdot\nabla(|\nabla\phi|^2)\\
&=\sum_{\alpha,\beta,\mu,\nu=0}^2Q^{\alpha\beta\mu\nu}\p^2_{\alpha\beta}\phi\p_{\mu}\phi\p_{\nu}\phi.
\end{split}
\end{equation}
As in \eqref{QWE}, we consider the following homogenous Neumann boundary value problem of \eqref{HCC-Y} in the exterior domain
$[0,\infty)\times\cK$ with $\cK=\R^2\setminus\cO$
\begin{equation}\label{QWE-Y}
\left\{
\begin{aligned}
&\frac{\p}{\p\bn}\phi(t,x)=0,\hspace{3.2cm} (t,x)\in[0,+\infty)\times\p\cK,\\
&(\phi,\p_t\phi)(0,x)=(\ve \phi_0,\ve \phi_1)(x),\qquad x\in\cK,
\end{aligned}
\right.
\end{equation}
where the bounded obstacle $\cO$ is convex and smooth, $\bn$ is the unit outer normal direction of $[0,+\infty)\times \p\cK$.
It is easy to verify that the cubic terms in \eqref{HCC-Y} satisfy the second null condition. Moreover,
for any smooth functions
$w$ and $\psi$ satisfying $\f{\p}{\p\bn}w|_{[0,\infty)\times\p\cK}=\f{\p}{\p\bn}\psi|_{[0,\infty)\times\p\cK}
=0$, it follows from \eqref{HCC-Y} and direct computation that for $(t,x)\in[0,+\infty)\times\p\cK$,
\begin{equation}\label{YHHC-0111}
\begin{split}
&\sum_{\alpha,\beta,\mu,\nu=0}^2Q^{\alpha\beta\mu\nu}n_{\alpha}\p_{\beta}w\p_{\mu}\psi\p_{\nu}\psi
=-\f{\p}{\p\bn}w(|\p_t\psi|^2-|\nabla\psi|^2)+\f{\p}{\p\bn}\psi(\p_t\psi\p_tw-\sum_{j=1}^2\p_j\psi\p_jw)\equiv0,
\end{split}
\end{equation}
which implies \eqref{YHCC-2} holds for the equation \eqref{HCC-Y}.
Therefore, in terms of Theorem \ref{thm2}, we have

\vskip 0.1 true cm

{\bf Corollary 1.3.} For the equation \eqref{HCC-Y} with \eqref{QWE-Y}, under the compatible condition \eqref{YHCC-3}
for the initial boundary values \eqref{QWE-Y},
then the smooth small data solution $\phi$ exists globally.

\subsubsection{Application to 2-D membrane equation}

2-D membrane equation is
\begin{equation}\label{HCC-X-1}
\p_t^2\phi-div\big(\ds\f{\nabla\phi}{\sqrt{1+|\nabla\phi|^2}}\big)=0,
\end{equation}
where $x=(x_1,x_2)$, $\nabla=(\p_{x_1}, \p_{x_2})$, $\phi=\phi(t,x)$ stands for the position of the membrane at $(t,x)$.
By direct computation, for $C^2$ solution $\phi$, \eqref{HCC-X-1}
can be written as
\begin{equation}\label{HCC-Y-1}
\begin{split}
\Box\phi&=\big(1-(1+|\nabla\phi|^2)^{\f32}\big)\p_t^2\phi+|\nabla\phi|^2\Delta\phi
-\f12\nabla\phi\cdot\nabla(|\nabla\phi|^2)\\
&=\sum_{\alpha,\beta,\mu,\nu=0}^2Q^{\alpha\beta\mu\nu}\p^2_{\alpha\beta}\phi\p_{\mu}\phi\p_{\nu}\phi+O(|\nabla\phi|^4)\p_t^2\phi.
\end{split}
\end{equation}
As in \eqref{QWE}, we consider the initial boundary value of \eqref{HCC-Y-1} in the exterior domain $[0,\infty)\times\cK$ with
$\cK=\R^2\setminus\cO$ as follows
\begin{equation}\label{QWE-Y-1}
\left\{
\begin{aligned}
&\frac{\p}{\p\bn}\phi(t,x)=0,\hspace{3.2cm} (t,x)\in[0,+\infty)\times\p\cK,\\
&(\phi,\p_t\phi)(0,x)=(\ve \phi_0,\ve \phi_1)(x),\qquad x\in\cK,
\end{aligned}
\right.
\end{equation}
where the bounded obstacle $\cO$ is convex and smooth, $\bn$ is the unit outer normal direction of $[0,+\infty)\times \p\cK$.
It follows from direct computation that \eqref{HCC-Y-1} does not satisfy the second null condition and fulfills
the compatible condition as in \eqref{YHHC-0111}. In terms of the proof procedures in Theorem \ref{thm2} and
\cite{HouYinYuan26}, one can obtain

\vskip 0.1 true cm

{\bf Corollary 1.4.} For the equation \eqref{HCC-Y-1} with \eqref{QWE-Y-1}, under the compatible condition \eqref{YHCC-3}
for the initial boundary value \eqref{QWE-Y-1},
then the lifespan $T_{\ve}$ of smooth small data solution $\phi$ satisfies
\begin{equation}\label{YHCC0-2-1}
\begin{split}
&T_\ve\ge e^{\frac{1}{\ve^{2-}}}.\\
\end{split}
\end{equation}

\subsection{Previous results and sketches of proofs}\label{YinA-5}

\vskip 0.2 true cm

A large number of  interesting results on the long time existence of small data smooth solutions to
the quasilinear wave equations in the whole space or in the exterior domains have been obtained.
In what follows, let us systematically recall these well-known conclusions.
Consider the Cauchy problem of the quasilinear wave equations
\begin{equation}\label{QWE:Cauchy}
\left\{
\begin{aligned}
&\Box v+\sum_{\alpha,\beta,\mu=0}^nQ^{\alpha\beta\mu}\p^2_{\alpha\beta}v\p_{\mu}v
+\sum_{\al,\beta,\mu,\nu=0}^nQ^{\alpha\beta\mu\nu}\p^2_{\alpha\beta}v\p_{\mu}v\p_{\nu}v=0 ,\quad t>0,\, x\in\R^n,\\
&(v,\p_tv)(0,x)=(\ve v^C_0,\ve v^C_1)(x),\qquad \qquad \qquad \qquad \qquad \qquad \qquad \qquad x\in\R^n
\end{aligned}
\right.
\end{equation}
and the initial boundary value problem of the quasilinear wave equations in exterior domains
\begin{equation}\label{QWE:exterior}
\left\{
\begin{aligned}
&\Box v+\sum_{\alpha,\beta,\mu=0}^nQ^{\alpha\beta\mu}\p^2_{\al\beta}v\p_{\mu}v
+\sum_{\alpha,\beta,\mu,\nu=0}^nQ^{\alpha\beta\mu\nu}\p^2_{\alpha\beta}v\p_{\mu}v\p_{\nu}v=0 ,\quad t>0,\, x\in\cK,\\
&\cB v=0,\qquad \qquad \qquad \qquad \qquad \qquad \qquad \qquad \qquad \qquad \qquad t>0, x\in\p\cK,\\
&(v,\p_tv)(0,x)=(\ve v^E_0,\ve v^E_1)(x),\qquad \qquad \qquad \qquad \qquad \qquad \qquad x\in\cK,
\end{aligned}
\right.
\end{equation}
where $(v^C_0,v^C_1)\in C_0^{\infty}(\R^n)$, $(v^E_0,v^E_1)\in C_0^{\infty}(\cK)$, $\ve>0$ is small,
$\cK=\R^n\setminus\cO$, the obstacle $\cO\subset\R^n$ is compact and contains the origin, the boundary $\p\cK=\p\cO$ is smooth,
$\cB$ denotes ${\rm Id}$ or $\frac{\p}{\p\bn}$ for the Dirichlet or the Neumann boundary conditions, respectively,
and $\bn=(0, n_1,\cdot\cdot\cdot, n_n)=(0, n_1(x),\cdot\cdot\cdot, n_n(x))$ represents the unit outer normal direction of the boundary $[0,+\infty)\times\p\cK$.
Let $T_{\ve}^{C}, T_{\ve}^{D}$ and $T_{\ve}^{N}$ be the lifespans of smooth solutions to the Cauchy problem \eqref{QWE:Cauchy},
the Dirichlet boundary problem and the Neumann boundary problem in \eqref{QWE:exterior}, respectively.
In addition, as in \eqref{YHCC-3}, the compatible conditions for the initial boundary values in \eqref{QWE:exterior}
hold. In particular, for the Neumann boundary conditions in  \eqref{QWE:exterior}, we impose the corresponding admissible condition \eqref{YHCC-2}.

\vskip 0.3 true cm

$\bullet$  When $n\ge4$, the global existences of small smooth solutions to problem \eqref{QWE:Cauchy} and
 problem \eqref{QWE:exterior} have been established, i.e., $T_{\ve}^{C}, T_{\ve}^{D}, T_{\ve}^{N}=\infty$, see \cite{Hormander97book,Klainerman80,KP83,Li-Chen} for the Cauchy problem \eqref{QWE:Cauchy} and \cite{Chen89,Hayashi94,MetcalfeSogge06,Shib-Tsuti85,Shib-Tsuti86} for the Dirichlet/Neumann boundary problem \eqref{QWE:exterior}, respectively.
 
\vskip 0.3 true cm

$\bullet$ When $n=3$ and the null condition holds (that is, $\ds\sum_{\alpha,\beta,\mu=0}^2Q^{\alpha\beta\mu}\xi_{\alpha}\xi_{\beta}\xi_{\mu}\equiv0$
for any $(\xi_0,\xi_1,\xi_2)\in\{\pm1\}\times\SS$), the small data  smooth solutions to both the Cauchy problem
and the Dirichlet/Neumann boundary problem exist globally (one can be referred to \cite{Christodoulou86,Klainerman} for \eqref{QWE:Cauchy} and \cite{Godin95,KSS00,Li-Yin18,MetcalfeSogge05} for \eqref{QWE:exterior}, respectively).
Without the null condition, the sharp lower bounds $T^C_{\ve}, T_{\ve}^D, T_{\ve}^N\ge e^{C/\ve}$ have also been obtained
for both the Cauchy problem \eqref{QWE:Cauchy} (see \cite{JohnKlainerman84}) and the Dirichlet/Neumann boundary problem \eqref{QWE:exterior} (see \cite{Godin89,Godin08,KSS02jam,KSS04,Yuan22,ZhaWang22}), respectively.

\vskip 0.3 true cm

$\bullet$ When $n=2$ and $Q^{\alpha\beta\mu}\equiv0$ for all $\alpha,\beta,\mu=0,1,2$ in \eqref{QWE:Cauchy} and \eqref{QWE:exterior}, the estimates on the lifespan have been studied and established, see \cite{Alinhac01b,Godin93,Kovalyov87,Hoshiga95} for $T_{\ve}^{C}\ge e^{C/\ve^2}$, \cite{Kubo14} for $T_{\ve}^{D}\ge e^{C/\ve^2}$ and \cite{KKL13} for $T_{\ve}^{N}\ge e^{C/\ve}$, respectively.
In particular, if $Q^{\alpha\beta\mu}\equiv0$ for all $\alpha,\beta,\mu=0,1,2$, and
the second null condition holds for \eqref{QWE:Cauchy} and \eqref{QWE:exterior} (that is, $\ds\sum_{\alpha,\beta,\mu,\nu=0}^2Q^{\alpha\beta\mu\nu}\xi_{\alpha}\xi_{\beta}\xi_{\mu}\xi_{\nu}\equiv0$ for all $(\xi_0,\xi_1,\xi_2)\in\{\pm1\}\times\SS$), then the small data smooth solutions exist globally, see \cite{Katayama95} for \eqref{QWE:Cauchy} and \cite{Kubo13,HouYinYuan26} for \eqref{QWE:exterior}, respectively.

\vskip 0.3 true cm

$\bullet$ When $n=2$ and $Q^{\alpha\beta\mu}\neq0$ for some $\alpha,\beta,\mu=0,1,2$, it is well known
from \cite{Alinhac99,Alinhac01a} that $T^C_{\ve}=\infty$ provided that both the first and second null
conditions hold, otherwise, $T^C_{\ve}\ge C/\ve^2$ if the first null condition fails
and $T^C_{\ve}\ge e^{C/\ve^2}$ if the first null condition holds but the second null condition fails.
In addition, very recently, we have shown that $T^D_{\ve}=\infty$ for the 2-D Cauchy-Dirichlet
problem \eqref{QWE:exterior} in exterior domains with both null conditions (see  \cite{HouYinYuan24}).
If the first null condition fails,
the authors in \cite{LRX25} proved $T^D_\ve\ge\frac{C}{\ve^2|\ln \ve|^3}$, which contains the $|\ln \ve|^{-3}$
loss than the
corresponding Cauchy problem.
For the 2-D Cauchy-Neumann problem \eqref{QWE:exterior} in exterior domains with or without the null conditions,
to the best of the authors' knowledge, so far it is still open.
The aim of this paper is to study the 2-D Cauchy-Neumann problem.

\vskip 0.2 true cm
For convenience, the previous works on $T_{\ve}^{C}, T_{\ve}^{D}$ and $T_{\ve}^{N}$ are listed in the following table.

\newpage

\begin{table}[!h]
\renewcommand\arraystretch{1.5}
\begin{tabular}{|c|c|c|c|c|}
\hline
Dimensions & Quadratic nonlinearity & Cauchy problem & Dirichlet problem& Neumann problem\\
\hline
$n\ge4$ &  & $\ds T_{\ve}^C=\infty$ & $\ds T_{\ve}^D=\infty$ &$\ds T_{\ve}^N=\infty$ \\
\hline
\multirow{2}{*}{$n=3$} & Null condition & $\ds T_{\ve}^C=\infty$ & $\ds T_{\ve}^D=\infty$&$\ds T_{\ve}^N=\infty$ \\
\cline{2-5}
 & No null condition & $\ds T_{\ve}^C\ge e^{\frac{C}{\ve}}$ & $\ds T_{\ve}^D\ge e^{\frac{C}{\ve}}$
 & $\ds T_{\ve}^N\ge e^{\frac{C}{\ve}}$
 \\
\hline
\multirow{3}{*}{$n=2$} & Null conditions & $\ds T_{\ve}^C=\infty$ & $\ds T_{\ve}^D=\infty$ &{\color{red}$\ds T_{\ve}^N=\infty$~Open} \\
\cline{2-5}
 & No first null condition & $\ds T_{\ve}^C\ge C/\ve^2$ & {\color{red}$\ds T_{\ve}^D\ge C/\ve^2$~Open}&
 {\color{red}$\ds T_{\ve}^N\ge C/\ve^2$~Open} \\
\cline{2-5}
 & No second null condition & $\ds T_{\ve}^C\ge e^{C/\ve^2}$ & {\color{red}$\ds T_{\ve}^D\ge e^{C/\ve^2}$~Open}&
 {\color{red}$\ds T_{\ve}^N\ge e^{C/\ve^2}$~Open} \\
\hline
\end{tabular}
\end{table}

\vskip 0.2 true cm

Next we give the comments on the proofs of Theorem \ref{thm1} and \ref{thm2}.
As pointed out in \cite[page 110]{KSS04}, the Lorentz boost fields $t\p_i+x_i\p_t$ $(i=1,2)$ used in the Cauchy
problem are difficult to be applied for the exterior domain problems since these fields are not tangent to the boundary
$(0,+\infty)\times\p\cK$ and contain the large factor $t$ before the space derivatives $\p_i$ ($i=1,2$).
On the other hand, it can be seen from \eqref{thm2:decay:LE} and the decay rates of the local energy discussed in Remark \ref{rmk-decay} that the scaling vector field $t\p_t+r\p_r$ is also hard to be utilized in the initial boundary value problems
(note that $t\p_t+r\p_r$ is not suitable for studying the nonlinear Klein-Gordon equations, one can see Chapter 7 of \cite{Hormander97book}).

\vskip 0.1 true cm

We now state the difficulties in the proof of Theorem \ref{thm1} without the null condition and illustrate
how to overcome these difficulties.
Note that the crucial estimate in Lemma 3.1 of \cite{LRX25} heavily relied on the Dirichlet boundary condition and
fails in the Cauchy-Neumann problem \eqref{QWE}, see page 10 of \cite{LRX25}.
Therefore, the method in \cite{LRX25} can not be used in the Neumann-wave problem \eqref{QWE} directly.
Instead, we will apply the approach in our work \cite{HouYinYuan24}.
Motivated by \cite{HouYinYuan24}, we will establish the following pointwise spacetime estimate
by the continuous argument
\begin{equation}\label{intro-decay1}
\sum_{|a|\le N}|\p Z^au(t,x)|\le C\ve\w{x}^{-1/2}\w{t-|x|}^{-1}\w{t}^{\delta/16},
\end{equation}
where $\delta>0$ is any fixed small constant.
With this pointwise estimate \eqref{intro-decay1}, the required energy estimate can be obtained (see \eqref{thm1-energyB} below).
In this derivation procedure, due to the lack of null conditions for the nonlinearity  $Q(\p u,\p^2u)$ as in \cite{HouYinYuan24}
and the appearance of large factor $\w{t}^{\delta/16}$ on the right hand side of \eqref{intro-decay1} under the analogous
rough bootstrap assumption,
then the decay rate of $Q(\p u,\p^2u)$ in \eqref{QWE}
is too weak to obtain the required estimate \eqref{intro-decay1}.
To overcome this difficulty, we rewrite the quadratic terms in $Q(\p u,\p^2u)$ of \eqref{QWE} into such a divergence form
\begin{equation*}
\begin{split}
\sum_{\alpha,\beta,\mu=0}^2Q^{\alpha\beta\mu}\p^2_{\alpha\beta}u\p_{\mu}u
&=\frac12\sum_{\alpha,\beta,\mu=0}^2[\p_{\alpha}(Q^{\alpha\beta\mu}\p_{\beta}u\p_{\mu}u)+\p_{\beta}(Q^{\alpha\beta\mu}\p_{\alpha}u\p_{\mu}u)
-\p_{\mu}(Q^{\alpha\beta\mu}\p_{\alpha}u\p_{\beta}u)].
\end{split}
\end{equation*}
In this case, some better pointwise decay estimates for the initial boundary value problem of the divergence form wave equations can be
achieved (see \eqref{dpw:ibvp} for details). Based on this, \eqref{intro-decay1} can be
derived and further Theorem \ref{thm1} is proved.

\vskip 0.1 true cm

We next give some comments on the proof of Theorem \ref{thm2}.
For the nonlinear equation with quadratic nonlinearity in \eqref{QWE}, by the standard energy method with the estimate \eqref{thm2:decay:a},
one can obtain that for the higher order energy $\ds E_{high}(t)=\sum_{|a|\le N_0}\|\p Z^au(t)\|_{L^2(\cK)}$
with $N_0\in\Bbb N$ being suitably large constant,
\begin{equation}\label{intro-energy-ineq}
\frac{d}{dt}E_{high}(t)\le C\ve(1+t)^{-1}\ln^2(2+t)E_{high}(t).
\end{equation}
This, together with the Gronwall's lemma,  leads to
\begin{equation}\label{intro-energy}
E_{high}(t)\le E_{high}(0)e^{C\ve\ln^3(2+t)}=E_{high}(0)(2+t)^{C\ve\ln^2(2+t)}.
\end{equation}
Obviously, it is hard to control the energy $E_{high}(t)$ by \eqref{intro-energy} since it grows very rapidly with the development
of time $t$.
To overcome this difficulty, we intend to find some new good unknowns to eliminate the related quadratic nonlinearities
so that the more precise energy estimates with at least slow time growth can be obtained.
Motivated by the following identity
\begin{equation}\label{Q0-identity}
\Box(fh)=f\Box h+h\Box f+2Q_0(f,h),
\end{equation}
we now introduce the perturbed wave operator $\ds\Box_g=\Box+\sum_{\alpha,\beta=0}^2g^{\alpha\beta}\p^2_{\alpha\beta}$
with $\ds g^{\alpha\beta}=\sum_{\mu=0}^2Q^{\alpha\beta\mu}\p_{\mu}u
+\sum_{\mu,\nu=0}^2Q^{\alpha\beta\mu\nu}\p_{\mu}u\p_{\nu}u$ for problem \eqref{QWE}
and using $\Box_g$ instead of $\Box$ in \eqref{Q0-identity} to obtain
\begin{equation}\label{perturb-Q0}
\Box_g(fh)=f\Box_gh+h\Box_gf+2Q_0(f,h)+\sum_{\alpha,\beta=0}^22g^{\alpha\beta}\p_{\alpha}f\p_{\beta}h.
\end{equation}
From this, one knows that if $f$ and $h$ are the solutions of the quasilinear wave operator $\Box_g$,
then there is no loss of regularities for $\Box_g(fh)$ when $\Box_gf$ and $\Box_gh$ are substituted into the right
hand side of \eqref{perturb-Q0}.
By such a careful observation,  we start to search for a series of new good unknowns
so that  the related  $Q_0$-type null forms can be eliminated and the regularities of obtained nonlinearities
are not lost correspondingly, meanwhile, the remaining quadratic nonlinearities
and the resulting cubic nonlinearities can admit better pointwise spacetime decay estimates. More concretely speaking,
it follows from the equation in \eqref{QWE} with \eqref{nonlinear-Chaplygin}
and direct computation that for $i=1,2\cdots,2N$,
\begin{equation}\label{energy:t3-1}
\begin{split}
\Box_g\p_t^iu&=2\sum_{\alpha=0}^2C^{\alpha}Q_0(\p_{\alpha}u,\p_t^iu)
+2\sum_{\alpha=0}^2\sum_{\substack{j+k=i,\\1\le j,k\le i-1}}C^{\alpha}C^i_{jk}Q_0(\p_{\alpha}\p_t^ju,\p_t^ku)\\
&\quad+\sum_{\alpha,\beta,\mu,\nu=0}^2\sum_{\substack{j+k+l=i,\\ j<i}}Q^{\alpha\beta\mu\nu}C^i_{jkl}\p^2_{\alpha\beta}\p_t^ju\p_{\mu}\p_t^ku\p_{\nu}\p_t^lu,
\end{split}
\end{equation}
where $C^i_{jk}$ and $C^i_{jkl}$ are some suitable positive constants.
Introducing the following good unknowns for $i=1,\cdots,2N$ (see \eqref{energy:t4} below)
\begin{equation}\label{energy:t4-1}
V_i=\p_t^iu-\frac12\sum_{\alpha=0}^2C^{\alpha}\p_{\alpha}u\p_t^iu
-\sum_{\alpha=0}^2\sum_{\substack{j+k=i,\\1\le j,k\le i-1}}C^{\alpha}C^i_{jk}\chi_{[1,2]}(x)\p_{\alpha}\p_t^ju\p_t^ku,
\end{equation}
where cutoff function $\chi_{[1,2]}(s)\in C^\infty(\R)$ with $0\le\chi_{[1,2]}(s)\le1$,
$\chi_{[1,2]}(s)=0$ for $s\le 1$ and $\chi_{[1,2]}(s)=1$ for $s\ge 2$.
Then by \eqref{perturb-Q0}-\eqref{energy:t4-1}, we can obtain
\begin{equation}\label{energy:t5-1-0}
\begin{split}
&\Box_gV_i-\sum_{\alpha=0}^2C^{\alpha}Q_0(\p_{\alpha}u,V_i)
=\text{good terms}.\\
\end{split}
\end{equation}
On the other hand, for any multi-index $a\in\N_0^4$ with $|a|\le2N-1$, one has that
for \eqref{QWE} with \eqref{nonlinear-Chaplygin},
\begin{equation}\label{energyA1-1}
\begin{split}
\Box_gZ^au&=2\sum_{\alpha=0}^2C^{\alpha}Q_0(\p_{\alpha}u,Z^au)
+2\sum_{\alpha=0}^2\sum_{\substack{b+c\le a,\\b,c<a}}C^{\alpha}C^a_{bc}Q_0(\p_{\alpha}Z^bu,Z^cu)\\
&\quad+\sum_{\substack{b+c+d\le a,\\b<a}}Q_{abcd}^{\alpha\beta\mu\nu}\p^2_{\alpha\beta}Z^bu\p_{\mu}Z^cu\p_{\nu}Z^du.
\end{split}
\end{equation}
Analogously, we also define the good unknowns as follows (see \eqref{energyA2})
\begin{equation}\label{energyA2-1}
V^a=Z^au-\frac12\sum_{\alpha=0}^2C^{\alpha}\p_{\alpha}uZ^au
-\sum_{\alpha=0}^2\sum_{\substack{b+c\le a,\\b,c<a}}C^{\alpha}C^a_{bc}\p_{\alpha}Z^buZ^cu,
\end{equation}
which leads to
\begin{equation}\label{energyA3-1}
\begin{split}
&\Box_gV^a-\sum_{\alpha=0}^2C^{\alpha}Q_0(\p_{\alpha}u,V^a)
=\text{good terms}.
\end{split}
\end{equation}
By \eqref{energy:t5-1-0} and \eqref{energyA3-1}, together with the ghost weight technique introduced in \cite{Alinhac01a}, the trace theorem, the elliptic estimate, \eqref{thm2:decay:d} and \eqref{thm2:decay:LEdt}, the desired precise energy inequality can be established (see Lemmas \ref{lem:energy:time}-\ref{lem:energy:dtdx} below)
\begin{equation}\label{YHCC-22}
\frac{d}{dt}E_{high}(t)\le C\ve(1+t)^{-1}E_{high}(t),
\end{equation}
which is certainly an improvement of \eqref{intro-energy-ineq} since the large and troublesome factor $\ln^2(2+t)$
on the right hand side of \eqref{intro-energy-ineq} has been removed.
Applying \eqref{YHCC-22} together with the Gronwall's lemma and further the Klainerman-Sobolev inequality
yield the following pointwise estimate
\begin{equation}\label{intro-decay2}
\sum_{|a|\le2N-9}|\p Z^au|\ls\ve_1\w{x}^{-1/2}(1+t)^{0+}.
\end{equation}
Based on \eqref{intro-decay2}, analogously to the proof procedure in \cite{HouYinYuan24} but with slightly different spacetime decay estimates
and different boundary conditions, \eqref{thm2:decay:a}-\eqref{thm2:decay:d} will be achieved by the weighted $L^\infty-L^\infty$
estimates for the Neumann boundary value problem of 2-D linear wave equation
along the following steps (see Section \ref{sect6})
\begin{equation}\label{YHCC-50}
\begin{split}
\sum_{|a|\le2N-15}|Z^au|\le C\ve\w{t+|x|}^{0+} & \Rightarrow \sum_{|a|\le2N-17}|\bar\p Z^au|\le C\ve\w{x}^{-1+},|x|\ge1+t/2 \\
& \Downarrow \\
\sum_{|a|\le2N-24}|Z^au|\le C\ve\w{t+|x|}^{-1/2} & \Rightarrow \sum_{|a|\le2N-27}|\bar\p Z^au|\le C\ve\w{x}^{-3/2},|x|\ge1+t/2 \\
& \Downarrow \\
\sum_{|a|\le2N-28}|\p Z^au|&\le C\ve\w{x}^{-1/2}\w{t-|x|}^{-1}\w{t}^{0+}\\
& \Downarrow \\
\sum_{|a|\le2N-29}\|\p^au\|_{L^2(\cK_R)}\le C\ve(1+t)^{-1}\ln(2+t), &\sum_{|a|\le2N-30}\|\p^a\p_tu\|_{L^2(\cK_R)}\le C\ve(1+t)^{-1} \\
& \Downarrow \\
\sum_{|a|\le2N-35}|Z^au|&\le C\ve\w{t+|x|}^{-1/2}\w{t-|x|}^{-1/2+}\\
& \Downarrow \\
\sum_{|a|\le2N-39}|\p Z^au|\le C\ve\w{x}^{-1/2}\w{t-|x|}^{-1}\ln^2(2&+t),
\sum_{|a|\le2N-39}|\bar\p Z^au|\le C\ve\w{x}^{-1/2}\w{t+|x|}^{-1+}\\
& \Downarrow \\
|\p_t\p u|&\le C\ve\w{x}^{-1/2}\w{t-|x|}^{-1}.
\end{split}
\end{equation}
In this process, similarly to \eqref{energy:t4-1} and \eqref{energyA2-1}, some other good unknowns are also introduced.
For examples, the good unknowns for $i=0,1,2\cdots2N-1$ (see \eqref{good1})
\begin{equation*}\label{good1-1}
\tilde V_i:=\p_t^iu-\sum_{\alpha=0}^2\sum_{j+k=i}C^{\alpha}C^i_{jk}\chi_{[1,2]}(x)\p_{\alpha}\p_t^ju\p_t^ku
\end{equation*}
and such good unknowns for $|a|\le2N-1$ (see \eqref{good2}),
\begin{equation*}\label{good2-1}
\tilde V^a:=\tilde Z^au-\sum_{\alpha=0}^2\sum_{b+c\le a}C^{\alpha}C^a_{bc}\chi_{[1/2,1]}(x)\p_{\alpha}Z^buZ^cu\quad
\text{with $\tilde Z=\chi_{[1/2,1]}(x)Z$}.
\end{equation*}
Based on \eqref{YHCC-50}, together with the continuous argument for problem \eqref{QWE} and \eqref{nonlinear-Chaplygin},
Theorem \ref{thm2} can be eventually proved.

\vskip 0.1 true cm

The paper is organized as follows.
In Section \ref{sect2}, some basic lemmas and several important results on the pointwise spacetime estimates of solutions to the 2-D
linear wave equations are stated. In addition,
some related pointwise estimates of the initial boundary value problem with Neumann condition are given in
Section \ref{sect3}.
In Section \ref{sect4}, the estimate on the lifespan of the solutions to the
quasilinear wave equations without the null condition in exterior domains is investigated and subsequently Theorem \ref{thm1}
is shown.
In Section \ref{sect5}, a series of good unknowns are introduced and the required precise energy estimates are further established.
In Section \ref{sect6}, the crucial pointwise estimates are improved step by step and then the proof of Theorem \ref{thm2}
is completed.

\vskip 0.2 true cm

\noindent \textbf{Notations:}
\begin{itemize}
  \item $\cK:=\R^2\setminus\cO$, the obstacle $\cO\subset\R^2$ is compact and contains the origin,
  $\cK_R:=\cK\cap\{x: |x|\le R\}$, $R>1$ is a fixed constant which may be changed in different places.
  \item Without loss of generality, assume $\p\cK\subset\{x:c_0<|x|<1/2\}$ with $c_0>0$ being a constant.
  \item The cutoff function $\chi_{[a,b]}(s)\in C^\infty(\R)$ with $a,b\in\R$, $0<a<b$, $0\le\chi_{[a,b]}(s)\le1$ and
  \begin{equation*}
    \chi_{[a,b]}(s)=\left\{
    \begin{aligned}
    0,\qquad s\le a,\\
    1,\qquad s\ge b.
    \end{aligned}
    \right.
  \end{equation*}
  \item $\w{x}:=\sqrt{1+|x|^2}$.
  \item $\N_0:=\{0,1,2,\cdots\}$ and $\N:=\{1,2,\cdots\}$.
  \item $\p_0:=\p_t$, $\p_1:=\p_{x_1}$, $\p_2:=\p_{x_2}$, $\p_x:=\nabla_x=(\p_1,\p_2)$, $\p:=(\p_t,\p_1,\p_2)$.
  \item For $|x|>0$, define $\bar\p_i:=\p_i+\frac{x_i}{|x|}\p_t$ ($i=1,2$) and $\bar\p=(\bar\p_1,\bar\p_2)$.
  \item $\Omega_{ij}:=x_i\p_j-x_j\p_i$ with $i,j=1,2$, $\Omega:=\Omega_{12}$,  $Z=\{Z_1,Z_2,Z_3,Z_4\}:=\{\p_t,\p_1,\p_2,\Omega\}$,
  $\tilde Z=\{\tilde Z_1,\tilde Z_2,\tilde Z_3,\tilde Z_4\}$, $\tilde Z_1:=\p_t$, $\tilde Z_i:=\chi_{[1/2,1]}(x)Z_i,i=2,3,4$.
  \item $\p_x^a:=\p_1^{a_1}\p_2^{a_2}$ for $a\in\N_0^2$, $\p^a:=\p_t^{a_1}\p_1^{a_2}\p_2^{a_3}$ for $a\in\N_0^3$
      and $Z^a:=Z_1^{a_1}Z_2^{a_2}Z_2^{a_3}Z_4^{a_4}$ for $a\in\N_0^4$.
  \item The commutator $[X,Y]:=XY-YX$.
  \item For $f\ge0$ and $g\ge0$, $f\ls g$ or $g\gt f$ denotes $f\le Cg$ for a generic constant $C>0$ independent of $\ve>0$,
  and $f\approx g$ means $f\ls g$ and $g\ls f$.
  \item Denote $\|u(t)\|=\|u(t,\cdot)\|$ with norm $\|\cdot\|=\|\cdot\|_{L^2(\cK)},\|\cdot\|_{L^2(\cK_R)},\|\cdot\|_{H^k(\cK)}$.
  \item $|\bar\p f|:=(|\bar\p_1f|^2+|\bar\p_2f|^2)^{1/2}$.
  \item $\cW_{\mu,\nu}(t,x)=\w{t+|x|}^\mu(\min\{\w{x},\w{t-|x|}\})^\nu$ for $\mu, \nu\in\Bbb R$.
  \item $\ds|Z^{\le j}f|:=\big(\sum_{0\le|a|\le j}|Z^af|^2\big)^\frac12$ and $Z^{\le1}f=(f,Zf)$.
\end{itemize}

\section{Preliminaries}\label{sect2}

\subsection{Several lemmas}

In this subsection, some lemmas on the elliptic estimate, the local energy decay estimate, Gronwall's lemma and
the Sobolev embedding will be listed or derived.

\begin{lemma}[Lemma 3.2 of \cite{Kubo13}]
Let $w\in H^j(\cK)$ and $\ds\frac{\p}{\p\bn}w|_{\cK}=0$ with integer $j\ge2$.
Then for any fixed constant $R>1$ and multi-index $a\in\N_0^2$ with $2\le|a|\le j$, one has
\begin{equation}\label{ellip}
\|\p_x^aw\|_{L^2(\cK)}\ls\|\Delta w\|_{H^{|a|-2}(\cK)}+\|w\|_{H^{|a|-1}(\cK_{R+1})}.
\end{equation}
\end{lemma}

\begin{lemma}
Let $w$ be the solution of the IBVP
\begin{equation}\label{loc:decay:eqn}
\left\{
\begin{aligned}
&\Box w=G(t,x),\qquad (t,x)\in[0,\infty)\times\cK,\\
&\frac{\p}{\p\bn}w=0,\qquad\qquad (t,x)\in[0,\infty)\times\p\cK,\\
&(w,\p_tw)(0,x)=(w_0,w_1)(x),\quad x\in\cK,
\end{aligned}
\right.
\end{equation}
where $\supp(w_0(x),w_1(x),G(t,x))\subset\{x: |x|\le R_1\}$ for $R_1>1$.
Then for fixed $R>1$, $m\in\N$, $\eta\in(0,1]$ and $\mu>0$, there is a positive constant $C$ such that
\begin{equation}\label{loc:decay}
\begin{split}
\sum_{|a|\le m}\w{t}^{\eta}\|\p^aw\|_{L^2(\cK_R)}
\le C(\|w_0\|_{H^{m}(\cK)}+\|w_1\|_{H^{m-1}(\cK)}\\
+\ln(2+t)\sum_{|a|\le m-1}\sup_{0\le s\le t}\w{s}^{\eta}\|\p^aG(s)\|_{L^2(\cK)})
\end{split}
\end{equation}
and
\begin{equation}\label{loc:decay+}
\begin{split}
\sum_{|a|\le m}\w{t}^{\eta}\|\p^aw\|_{L^2(\cK_R)}
\le C(\|w_0\|_{H^{m}(\cK)}+\|w_1\|_{H^{m-1}(\cK)}\\
+\sum_{|a|\le m-1}\sup_{0\le s\le t}\w{s}^{\eta+\mu}\|\p^aG(s)\|_{L^2(\cK)}).
\end{split}
\end{equation}
\end{lemma}
\begin{proof}
See Lemma 3.5 of \cite{Kubo13} for \eqref{loc:decay}.
Although the proof of the estimate \eqref{loc:decay+} is similar to that of \eqref{loc:decay} and \cite[Lemma 3.2]{Kubo14},
for the purpose of completeness, we still give the details.
At first, consider the homogeneous problem for \eqref{loc:decay:eqn} with $G=0$
\begin{equation*}
\left\{
\begin{aligned}
&\Box w_{hom}=0,\qquad\qquad (t,x)\in[0,\infty)\times\cK,\\
&\frac{\p}{\p\bn}w_{hom}=0,\qquad\qquad (t,x)\in[0,\infty)\times\p\cK,\\
&(w_{hom},\p_tw_{hom})(0,x)=(w_0,w_1)(x),\quad x\in\cK.
\end{aligned}
\right.
\end{equation*}
Denote by $w_{hom}=\cS[w_0,w_1](t)$, then $\cS[w_0,w_1](0)=w_0$ and $\p_t\cS[w_0,w_1](0)=w_1$.
By \eqref{loc:decay} with $G=0$, $m=\eta=1$ or (17) of \cite{Kubo13}, one has
\begin{equation}\label{loc:decay:pf1}
\sum_{|a|\le 1}\|\p^a\cS[w_0,w_1](t)\|_{L^2(\cK_R)}
\ls(1+t)^{-1}(\|w_0\|_{H^{1}(\cK)}+\|w_1\|_{L^2(\cK)}).
\end{equation}
On the other hand, by Duhamel's principle, the solution to the inhomogeneous problem \eqref{loc:decay:eqn} can be represented as
\begin{equation}\label{loc:decay:pf2}
w=\cS[w_0,w_1](t)+\int_0^t\cS[0,G(s)](t-s)ds.
\end{equation}
Thus, for $|a|\le1$, we have
\begin{equation}\label{loc:decay:pf3}
\p^aw=\p^a\cS[w_0,w_1](t)+\int_0^t\p^a\cS[0,G(s)](t-s)ds.
\end{equation}
Combining \eqref{loc:decay:pf3} with \eqref{loc:decay:pf1} leads to
\begin{equation}\label{loc:decay:pf4}
\begin{split}
\sum_{|a|\le 1}\w{t}^{\eta}\|\p^aw(t)\|_{L^2(\cK_R)}
&\ls\|w_0\|_{H^{1}(\cK)}+\|w_1\|_{L^2(\cK)}+\w{t}^{\eta}\int_0^t(1+t-s)^{-1}\|G(s)\|_{L^2(\cK)}ds\\
&\ls\|w_0\|_{H^{1}(\cK)}+\|w_1\|_{L^2(\cK)}+\sup_{0\le s\le t}\w{s}^{\eta+\mu}\|G(s)\|_{L^2(\cK)},
\end{split}
\end{equation}
where we have used the fact of
\begin{equation*}
\int_0^t(1+t-s)^{-1}(1+s)^{-\eta-\mu}ds\ls(1+t)^{-\eta}.
\end{equation*}
Thus, \eqref{loc:decay+} for $m=1$ is established.
Next, we turn to prove \eqref{loc:decay+} for $m\ge2$.
Set $w_j(x)=\Delta w_{j-2}(x)+\p_t^{j-2}G(0,x)$ for $j\ge2$.
By \eqref{loc:decay:pf2}, one can find that for $j\ge2$,
\begin{equation}\label{loc:decay:pf5}
\p_t^jw=\cS[w_j,w_{j+1}](t)+\int_0^t\cS[0,\p_s^jG(s)](t-s)ds.
\end{equation}
Analogously to \eqref{loc:decay:pf4}, we arrive at
\begin{equation}\label{loc:decay:pf6}
\sum_{|a|\le 1}\w{t}^{\eta}\|\p^a\p_t^jw(t)\|_{L^2(\cK_R)}
\ls\|w_j\|_{H^{1}(\cK)}+\|w_{j+1}\|_{L^2(\cK)}+\sup_{0\le s\le t}\w{s}^{\eta+\mu}\|\p_t^jG(s)\|_{L^2(\cK)}.
\end{equation}
For $\p_x^a\p_t^jw$ with $|a|\ge2$, applying the elliptic estimate \eqref{ellip} to $\p_t^jw$ with cutoff technique yields
\begin{equation}\label{loc:decay:pf7}
\begin{split}
\|\p_x^a\p_t^jw\|_{L^2(\cK_R)}
&\ls\|\Delta\p_t^jw\|_{H^{|a|-2}(\cK_{R+1})}+\|\p_t^jw\|_{H^{|a|-1}(\cK_{R+1})}\\
&\ls\|\p_t^{j+2}w\|_{H^{|a|-2}(\cK_{R+1})}+\|\p_t^jG\|_{H^{|a|-2}(\cK_{R+1})}+\|\p_t^jw\|_{H^{|a|-1}(\cK_{R+1})}.
\end{split}
\end{equation}
Therefore, \eqref{loc:decay+} is achieved by \eqref{loc:decay:pf6} and \eqref{loc:decay:pf7} together with the induction method.
\end{proof}

\begin{lemma}\label{lem-Gronwall-1}
For any non-negative constants $\mu\in[0,1/2)$, $A,B,C,D$ with $B>0$ and function $f(t)\ge0$ satisfying
\begin{equation*}
f(t)\le A+B\int_0^t\frac{f(s)ds}{(1+s)^{1/2-\mu}}+C(1+t)^D,
\end{equation*}
then it holds
\begin{equation}\label{YHCC-20}
f(t)\le e^{\frac{B}{\mu+1/2}(1+t)^{\mu+1/2}}[A+C(1+t)^D].
\end{equation}
\end{lemma}
\begin{proof}
Set $\ds F(t)=\int_0^t\frac{f(s)ds}{(1+s)^{1/2-\mu}}$. Then we have $F(0)=0$ and
\begin{equation}\label{YHCC-20-1}
F'(t)\le A(1+t)^{\mu-1/2}+B(1+t)^{\mu-1/2}F(t)+C(1+t)^{D+\mu-1/2}.
\end{equation}
Hence, it holds 
$$\big(e^{-\frac{B}{\mu+1/2}(1+t)^{\mu+1/2}}F(t)\big)'\le e^{-\frac{B}{\mu+1/2}(1+t)^{\mu+1/2}}[A(1+t)^{\mu-1/2}+C(1+t)^{D+\mu-1/2}]$$
and
\begin{equation}\label{YHCC-20-2}
e^{-\frac{B}{\mu+1/2}(1+t)^{\mu+1/2}}F(t)\le\int_0^t e^{-\frac{B}{\mu+1/2}(1+s)^{\mu+1/2}}[A(1+s)^{\mu-1/2}+C(1+s)^{D+\mu-1/2}]ds.
\end{equation}
In addition, it follows from direct computation that
\begin{equation}\label{YHCC-20-3}
\begin{split}
&\int_0^t e^{-\frac{B}{\mu+1/2}(1+s)^{\mu+1/2}}C(1+s)^{D+\mu-1/2}ds\\
&\le C(1+t)^D\int_0^t e^{-\frac{B}{\mu+1/2}(1+s)^{\mu+1/2}}(1+s)^{\mu-1/2}ds\\
&=\frac{C}{B}(1+t)^D(e^{-\frac{B}{\mu+1/2}}-e^{-\frac{B}{\mu+1/2}(1+t)^{\mu+1/2}})\\
&\le\frac{C}{B}(1+t)^D(1-e^{-\frac{B}{\mu+1/2}(1+t)^{\mu+1/2}})
\end{split}
\end{equation}
and
\begin{equation}\label{YHCC-20-4}
\int_0^t e^{-\frac{B}{\mu+1/2}(1+s)^{\mu+1/2}}A(1+s)^{\mu-1/2}ds\le\frac{A}{B}(1-e^{-\frac{B}{\mu+1/2}(1+t)^{\mu+1/2}}).
\end{equation}
Substituting \eqref{YHCC-20-3}-\eqref{YHCC-20-4} into \eqref{YHCC-20-2} and further \eqref{YHCC-20-1} yields \eqref{YHCC-20}.
\end{proof}

\begin{lemma}\label{lem-Gronwall-2}
For any constants $A,B,C,D$ with $0<B<D$ and function $f(t)\ge0$ satisfying
\begin{equation*}
f(t)\ls A+B\int_0^t\frac{f(s)ds}{(1+s)}+C(1+t)^D,
\end{equation*}
then we have
\begin{equation*}
f(t)\ls A(1+t)^B+C(1+t)^D+\frac{BC}{D-B}(1+t)^D.
\end{equation*}
\end{lemma}
\begin{proof}
Since the proof is similar to that for Lemma \ref{lem-Gronwall-1}, we omit the details here.
\end{proof}

\begin{lemma}[Lemma 3.1 of \cite{Kubo13}]
Given any function $f(x)\in C_0^2(\overline{\cK})$, one has that for all $x\in\cK$,
\begin{equation}\label{Sobo:ineq}
\w{x}^{1/2}|f(x)|\ls\sum_{|a|\le2}\|Z^af\|_{L^2(\cK)}.
\end{equation}
\end{lemma}

\subsection{Some pointwise estimates of Cauchy problem}

\begin{lemma}
Let $v$ be the solution of the Cauchy problem
\begin{equation*}
\left\{
\begin{aligned}
&\Box v=\sum_{\alpha=0}^2\p_\alpha H^\alpha(t,x),\qquad(t,x)\in[0,\infty)\times\R^2,\\
&(v,\p_tv)(0,x)=(0,0),\qquad x\in\R^2.
\end{aligned}
\right.
\end{equation*}
Then for $\eta\in(0,1)$, one has
\begin{align}\label{pw:ivp:div}
&\w{x}^{1/2}\w{t-|x|}^{1-\eta}|v|\ls\ln(2+t+|x|)\sum_{\alpha=0}^2\sum_{|a|\le1}\sup_{(s,y)\in[0,t]\times\R^2}\w{y}^{1/2}
\cW_{1,1}(s,y)|Z^aH^\alpha(s,y)|,
\end{align}
where $\ds\cW_{\mu,\nu}(t,x)=\w{t+|x|}^\mu(\min\{\w{x},\w{t-|x|}\})^\nu$.
\end{lemma}
\begin{proof}
Choosing $\mu=\eta>0$, $\nu=1-\eta$ and $\kappa=1$ in (B.6) of \cite{Kubo14}, we then have
\begin{equation}\label{pw:ivp:div1}
\w{x}^{1/2}\w{t-|x|}^{1-\eta}|v|\ls\ln(2+t+|x|)\sum_{\alpha=0}^2
\sum_{|a|\le1}\sup_{(s,y)\in[0,t]\times\R^2}\w{y}^{3/2}\w{s+|y|}|Z^aH^\alpha(s,y)|.
\end{equation}
Recall (B.7) of \cite{Kubo14}
\begin{equation}\label{pw:ivp:div2}
\w{x}^{1/2}\w{t-|x|}^{1-\eta}|v|\ls\ln(2+t+|x|)\sum_{\alpha=0}^2
\sum_{|a|\le1}\sup_{(s,y)\in[0,t]\times\R^2}\w{y}^{1/2}\w{s+|y|}\w{s-|y|}|Z^aH^\alpha(s,y)|.
\end{equation}
Combining \eqref{pw:ivp:div1} with \eqref{pw:ivp:div2} derives \eqref{pw:ivp:div}.
\end{proof}

\begin{lemma}
Let $v$ be the solution of the Cauchy problem
\begin{equation*}
\left\{
\begin{aligned}
&\Box v=H(t,x),\qquad(t,x)\in[0,\infty)\times\R^2,\\
&(v,\p_tv)(0,x)=(v_0,v_1),\quad x\in\R^2.
\end{aligned}
\right.
\end{equation*}
Then for $\mu,\nu\in(0,1/2)$, it holds that
\begin{align}
&\w{t+|x|}^{1/2}\w{t-|x|}^\mu|v|\ls\cA_{3,0}[v_0,v_1]+\sup_{(s,y)\in[0,t]\times\R^2}\w{y}^{1/2}\cW_{1+\mu,1+\nu}(s,y)|H(s,y)|,\label{pw:ivp1}\\
&\w{t+|x|}^{1/2}\w{t-|x|}^{1/2}|v|\ls\cA_{3,0}[v_0,v_1]+\sup_{(s,y)\in[0,t]\times\R^2}\w{y}^{1/2}\cW_{3/2+\mu,1+\nu}(s,y)|H(s,y)|\label{pw:ivp2}
\end{align}
and
\begin{equation}\label{dpw:ivp}
\begin{split}
&\w{x}^{1/2}\w{t-|x|}^{1+\mu}|\p v|\ls\cA_{4,1}[v_0,v_1]\\
&\qquad+\sum_{|a|\le1}\sup_{(s,y)\in[0,t]\times\R^2}\w{y}^{1/2}\cW_{1+\mu+\nu,1}(s,y)|Z^aH(s,y)|,
\end{split}
\end{equation}
where $\ds\cA_{\kappa,s}[f,g]:=\sum_{\tilde\Gamma\in\{\p_1,\p_2,\Omega\}}
(\sum_{|a|\le s+1}\|\w{z}^\kappa\tilde\Gamma^af(z)\|_{L^\infty}+\sum_{|a|\le s}\|\w{z}^\kappa\tilde\Gamma^ag(z)\|_{L^\infty})$.
\end{lemma}
\begin{proof}
See (3.4), (3.5), (3.7) of \cite{HouYinYuan24} (or (4.2), (4.3), (4.4) and (A.13) of \cite{Kubo19}) for \eqref{pw:ivp1}, \eqref{pw:ivp2}, \eqref{dpw:ivp}, respectively.
\end{proof}

\begin{lemma}\label{lem:pw:compact}
Let $v$ be the solution of the Cauchy problem
\begin{equation*}
\left\{
\begin{aligned}
&\Box v=H(t,x),\qquad\qquad(t,x)\in[0,\infty)\times\R^2,\\
&(v,\p_tv)(0,x)=(0,0),\quad x\in\R^2,
\end{aligned}
\right.
\end{equation*}
where $\supp H(t,\cdot)\subset\{|x|\le R\}$.
Then we have that for $0<\mu<\nu<1/2$,
\begin{equation}\label{pw:compact}
\w{t+|x|}^{1/2}\w{t-|x|}^{\mu}|v|\ls\sum_{|a|\le2}\sup_{(s,y)\in[0,t]\times\R^2}\w{s}^{1/2+\nu}|\p^aH(s,y)|
\end{equation}
and
\begin{equation}\label{dpw:compact}
\w{x}^{1/2}\w{t-|x|}|\p v|
\ls\sup_{(s,y)\in[0,t]\times\R^2}\w{s}|(1-\Delta)^3H(s,y)|.
\end{equation}
\end{lemma}
\begin{proof}
The estimate \eqref{dpw:compact} is just the Lemma 3.8 of \cite{HouYinYuan24}.
On the other hand, \eqref{pw:compact} can be achieved by Lemma 3.10 of \cite{HouYinYuan24} and Theorem 1.1 of \cite{Kubo15},
i.e.,
\begin{equation*}
\w{t+|x|}^{1/2}\w{t-|x|}^{\nu}|v|
\ls\ln^2(2+t+|x|)\sup_{(s,y)\in[0,t]\times\R^2}\w{s}^{1/2+\nu}|(1-\Delta)H(s,y)|
\end{equation*}
and
\begin{equation*}
\w{t+|x|}^{1/2}\min\{\w{t-|x|},\w{x}\}^{\nu}|v|\ls\sup_{(s,y)\in[0,t]\times\R^2}\w{s}^{1/2+\nu}|H(s,y)|,
\end{equation*}
respectively.
\end{proof}

\subsection{Null conditions}

\begin{lemma}\label{lem:null:structure}
Suppose that the constants $N^{\alpha\beta\mu\nu}$ satisfy
\begin{equation*}
\sum_{\alpha,\beta,\mu,\nu=0}^2N^{\alpha\beta\mu\nu}\xi_\alpha\xi_\beta\xi_\mu\xi_\nu\equiv0
\quad\text{for $\xi=(\xi_0,\xi_1,\xi_2)\in\{\pm1\}\times\SS$.}
\end{equation*}
Then for smooth functions $f,g,h$ and $\phi$, it holds that
\begin{equation}\label{null:structure}
\begin{split}
|Q_0(f,g)|&\ls(|\bar\p f||\p g|+|\p f||\bar\p g|),\\
\Big|\sum_{\alpha,\beta,\mu,\nu=0}^2N^{\alpha\beta\mu\nu}
\p^2_{\alpha\beta}f\p_{\mu}g\p_{\nu}h\Big|&\ls|\bar\p\p f||\p g||\p h|+|\p^2f||\bar\p g||\p h|+|\p^2f||\p g||\bar\p h|,\\
\Big|\sum_{\alpha,\beta,\gamma,\delta=0}^2N^{\alpha\beta\mu\nu}
\p_{\alpha}f\p_{\beta}g\p_{\mu}h\p_{\nu}\phi\Big|&\ls|\bar\p f||\p g||\p h||\p\phi|+|\p f||\bar\p g||\p h||\p\phi|\\
&\quad+|\p f||\p g||\bar\p h||\p\phi|+|\p f||\p g||\p h||\bar\p\phi|.
\end{split}
\end{equation}
\end{lemma}
\begin{proof}
Since the proof of \eqref{null:structure} is rather analogous to that in Section 9.1 of \cite{Alinhac:book}
and \cite[Lemma 2.2]{HouYin20jde}, we omit it here.
\end{proof}

\begin{lemma}\label{lem:eqn:high}
Let $u$ be the smooth solution of \eqref{QWE} with \eqref{nonlinear-Chaplygin}-\eqref{null:condition}.
Then for any multi-index $a$, $Z^au$ satisfies
\begin{equation}\label{eqn:high}
\Box Z^au=2\sum_{b+c\le a}C^{\alpha}C^a_{bc}Q_0(\p_{\alpha}Z^bu,Z^cu)
+\sum_{\alpha,\beta,\mu,\nu=0}^2\sum_{b+c+d\le a}Q^{\alpha\beta\mu\nu}_{abcd}\p^2_{\alpha\beta}Z^bu\p_{\mu}Z^cu\p_{\nu}Z^du,
\end{equation}
where $C^a_{bc}$ and $Q_{abcd}^{\alpha\beta\mu\nu}$ are constants, in particular $C^a_{a0}=1$
and $Q_{aa00}^{\alpha\beta\mu\nu}=Q^{\alpha\beta\mu\nu}$ hold.
Furthermore, for any $(\xi_0,\xi_1,\xi_2)\in\{\pm1\}\times\SS$, one has
\begin{equation}\label{null:high}
\sum_{\alpha,\beta,\mu,\nu=0}^2Q_{abcd}^{\alpha\beta\mu\nu}\xi_\alpha\xi_\beta\xi_\mu\xi_\nu\equiv0.
\end{equation}
\end{lemma}
\begin{proof}
See Lemma 6.6.5 of \cite{Hormander97book}.
\end{proof}

\section{Some estimates of the initial boundary value problem}\label{sect3}

At first, we derive the time decay estimate of local energy  near the boundary for the linear wave equation
with the divergence form inhomogeneous source terms.
\begin{lemma}
{\bf (Time decay estimate of $\|\p^{\le1}w\|_{L^2(\cK_R)}$)}
Suppose that the obstacle $\cO\subset\Bbb R^2$ is convex, $\cK=\R^2\setminus\cO$ and let $w$ be the solution of the IBVP
\begin{equation*}
\left\{
\begin{aligned}
&\Box w=\sum_{\alpha=0}^2\p_{\alpha}G^\alpha+G^r,\qquad(t,x)\in[0,\infty)\times\cK,\\
&\frac{\p}{\p\bn}w=0,\qquad \quad \qquad\qquad (t,x)\in[0,\infty)\times\p\cK,\\
&(w,\p_tw)(0,x)=(w_0,w_1)(x),\quad x\in\cK,
\end{aligned}
\right.
\end{equation*}
where $(w_0,w_1)$ has compact support and $\supp_x(G^\alpha,G^r)(t,x)\subset\{x: |x|\le t+M_0\}$.
Then one has that for any $\eta\in(0,1/2)$ and $R>1$,
\begin{equation}\label{loc:ibvp}
\begin{split}
&\w{t}\|\p^{\le1}w\|_{L^2(\cK_R)}
\ls\|(w_0,w_1)\|_{H^1(\cK)}+\ln(2+t)\sum_{|a|\le1}\sup_{s\in[0,t]}\w{s}\|(G^r,\p^aG^\alpha)(s)\|_{L^2(\cK_3)}\\
&\quad+\ln(2+t)[\sup_{y\in\cK}|G^r(0,y)|+\w{t}^\eta\sup_{y\in\cK}|\p^{\le1}G^\alpha(0,y)|]\\
&\quad+\w{t}^\eta\ln^2(2+t)\sum_{|a|\le2}\sup_{(s,y)\in[0,t]\times\overline{\R^2\setminus\cK_2}}\w{y}^{1/2}\cW_{1,1}(s,y)|Z^aG^\alpha(s,y)|\\
&\quad+\ln(2+t)\sum_{|a|\le1}\sup_{(s,y)\in[0,t]\times\overline{\R^2\setminus\cK_2}}\w{y}^{1/2}\cW_{3/2+\eta,1}(s,y)|\p^aG^r(s,y)|.
\end{split}
\end{equation}
\end{lemma}
\begin{proof}
As in \cite{HouYinYuan24}, let $w_1^b$ and $w_2^b$ be the solutions of
\begin{equation}\label{loc:ibvp:pf1}
\begin{split}
&\Box w_1^b:=(1-\chi_{[2,3]}(x))(\sum_{\alpha=0}^2\p_{\alpha}G^\alpha+G^r)-\sum_{i=1}^2\p_i(\chi_{[2,3]}(x))G^i,\\
&\Box w_2^b:=\sum_{\alpha=0}^2\p_{\alpha}(\chi_{[2,3]}(x)G^\alpha)+\chi_{[2,3]}(x)G^r,\\
&\frac{\p}{\p\bn}w_1^b|_{[0,\infty)\times\p\cK}=\frac{\p}{\p\bn}w_2^b|_{[0,\infty)\times\p\cK}=0,\\
&(w_1^b,\p_tw_1^b)(0,x)=(w_0,w_1)(x),\quad (w_2^b,\p_tw_2^b)(0,x)=(0,0),
\end{split}
\end{equation}
respectively.
Then from the uniqueness of smooth solution to the IBVP, one has $w=w_1^b+w_2^b$.
For the estimate of $w_1^b$ in $\cK_R$, it follows from the fact of $\supp_x\Box w_1^b\subset\{x:|x|\le3\}$ and \eqref{loc:decay} with $m=1$ that
\begin{equation}\label{loc:ibvp:pf2}
\begin{split}
\w{t}\|\p^{\le1}w_1^b\|_{L^2(\cK_R)}
&\ls\|(w_0,w_1)\|_{H^1(\cK)}+\ln(2+t)\sup_{s\in[0,t]}\w{s}\|\Box w_1^b(s)\|_{L^2(\cK)}\\
&\ls\|(w_0,w_1)\|_{H^1(\cK)}+\ln(2+t)\sum_{|a|\le1}\sup_{s\in[0,t]}\w{s}\|(G^r,\p^aG^\alpha)(s)\|_{L^2(\cK_3)}.
\end{split}
\end{equation}
Next we estimate $w_2^b$.
Let $w_2^c$ be the solution of the Cauchy problem
\begin{equation}\label{loc:ibvp:pf3}
\Box w_2^c=\left\{
\begin{aligned}
&\sum_{\alpha=0}^2\p_{\alpha}(\chi_{[2,3]}(x)G^\alpha)+\chi_{[2,3]}(x)G^r,\qquad&& x\in\cK=\R^2\setminus\cO,\\
&0,&& x\in\overline{\cO}
\end{aligned}
\right.
\end{equation}
with $(w_2^c, \p_tw_2^c)(0,x)=(0,0)$.
Set $v_2^b=\chi_{[1,2]}(x)w_2^c$ on $\cK$ and thus $v_2^b$ is the solution of the IBVP
\begin{equation}\label{loc:ibvp:pf4}
\begin{split}
&\Box v_2^b=\chi_{[1,2]}\Box w_2^c-[\Delta,\chi_{[1,2]}]w_2^c=\sum_{\alpha=0}^2\p_{\alpha}(\chi_{[2,3]}(x)G^\alpha)+\chi_{[2,3]}(x)G^r-[\Delta,\chi_{[1,2]}]w_2^c,\\
&\frac{\p}{\p\bn}v_2^b|_{[0,\infty)\times\p\cK}=0,\quad (v_2^b,\p_tv_2^b)(0,x)=(0,0).
\end{split}
\end{equation}
In addition, let $\tilde v_2^b$ be the solution of the IBVP
\begin{equation}\label{loc:ibvp:pf5}
\Box\tilde v_2^b=[\Delta,\chi_{[1,2]}]w_2^c,
\quad\frac{\p}{\p\bn}\tilde v_2^b|_{[0,\infty)\times\p\cK}=0,\quad (\tilde v_2^b,\p_t\tilde v_2^b)(0,x)=(0,0).
\end{equation}
Therefore, it follows from \eqref{loc:ibvp:pf1}, \eqref{loc:ibvp:pf4} and \eqref{loc:ibvp:pf5} that
\begin{equation}\label{loc:ibvp:pf6}
w_2^b=\chi_{[1,2]}(x)w_2^c+\tilde v_2^b.
\end{equation}
Due to $\supp_x\Box\tilde v_2^b\subset\{x:|x|\le2\}$, similarly to \eqref{loc:ibvp:pf2} for $w_1^b$, one has that for $\tilde v_2^b$,
\begin{equation}\label{loc:ibvp:pf7}
\begin{split}
\w{t}\|\p^{\le1}\tilde v_2^b\|_{L^2(\cK_R)}
&\ls\ln(2+t)\sup_{s\in[0,t]}\w{s}\|[\Delta,\chi_{[1,2]}]w_2^c(s)\|_{L^2(\cK_2)}\\
&\ls\ln(2+t)\sum_{|a|\le1}\sup_{s\in[0,t]}\w{s}\|\p_x^aw_2^c(s)\|_{L^\infty(|y|\le2)}.
\end{split}
\end{equation}
Thus it requires to treat $w_2^c$ in the right hand side of \eqref{loc:ibvp:pf7}.
By \eqref{loc:ibvp:pf3}, we have $w_2^c=w^{div}+w^r$ with $w^{div}$ and $w^r$ being the solutions to
\begin{equation}\label{loc:ibvp:pf8}
\begin{split}
&\Box w^{div}=\sum_{\alpha=0}^2\p_{\alpha}(\chi_{[2,3]}(x)G^\alpha),\\
&\Box w^r=\chi_{[2,3]}(x)G^r,\\
&(w^{div},\p_tw^{div})(0,x)=(0,0),\quad (w^r,\p_tw^r)(0,x)=(0,0),
\end{split}
\end{equation}
respectively.
Applying \eqref{pw:ivp:div} to $w^{div}$ and \eqref{pw:ivp2} to $w^r$ with $\mu=\nu=\eta/2$, respectively, yields
\begin{equation}\label{loc:ibvp:pf9}
\begin{split}
&\w{x}^{1/2}\w{t-|x|}^{1-\eta}|\p^{\le1}w^{div}|\ls\sum_{\alpha=0}^2\sup_{y\in\cK}|\p^{\le1}G^\alpha(0,y)|\\
&\qquad+\ln(2+t)\sum_{\alpha=0}^2\sum_{|a|\le2}\sup_{(s,y)\in[0,t]\times\overline{\R^2\setminus\cK_2}}\w{y}^{1/2}\cW_{1,1}(s,y)|Z^aG^\alpha(s,y)|,\\
&\w{t+|x|}^{1/2}\w{t-|x|}^{1/2}|\p^{\le1}w^r|\ls\sum_{\alpha=0}^2\sup_{y\in\cK}|G^r(0,y)|\\
&\qquad+\sum_{|a|\le1}\sup_{(s,y)\in[0,t]\times\overline{\R^2\setminus\cK_2}}\w{y}^{1/2}\cW_{3/2+\eta,1}(s,y)|\p^aG^r(s,y)|,
\end{split}
\end{equation}
where $\ds\sup_{y\in\cK}|\p^{\le1}G^\alpha(0,y)|$ and $\ds\sup_{y\in\cK}|G^r(0,y)|$ come from the initial data of $\p_tw^{div},\p_tw^r$ and \eqref{pw:ivp2} with $H=0$.
Collecting \eqref{loc:ibvp:pf2}, \eqref{loc:ibvp:pf6}, \eqref{loc:ibvp:pf7}-\eqref{loc:ibvp:pf9} derives \eqref{loc:ibvp}.
\end{proof}

Secondly, we show the pointwise spacetime decay estimate of the first order derivatives of solution to the linear wave equation
with the divergence form inhomogeneous source terms.
\begin{lemma}
{\bf (Spacetime decay estimate of $\p w$)}
Let $w(t,x)$ be the solution of the IBVP
\begin{equation*}
\left\{
\begin{aligned}
&\Box w=\sum_{\alpha=0}^2\p_{\alpha}G^\alpha+G^r,\qquad(t,x)\in[0,\infty)\times\cK,\\
&\frac{\p}{\p\bn}w=0,\qquad \quad \qquad\qquad (t,x)\in[0,\infty)\times\p\cK,\\
&(w,\p_tw)(0,x)=(w_0,w_1)(x),\quad x\in\cK,
\end{aligned}
\right.
\end{equation*}
where $(w_0,w_1)$ has compact support and $\supp_x(G^\alpha,G^r)(t,x)\subset\{x: |x|\le t+M_0\}$.
Then we have that for any $\eta\in(0,1/2)$,
\begin{equation}\label{dpw:ibvp}
\begin{split}
&\quad\w{x}^{1/2}\w{t-|x|}|\p w|\ls\|(w_0,w_1)\|_{H^9(\cK)}\\
&\quad+\ln(2+t)\sum_{|a|\le8}\sup_{s\in[0,t]}\w{s}\|\p^a(G^r,\p^{\le1}G^\alpha)(s)\|_{L^2(\cK_3)}\\
&\quad+\ln(2+t)[\sum_{|a|\le8}\sup_{y\in\cK}|Z^a(G^r,\p^{\le1}G^\alpha)(0,y)|+\w{t}^\eta\sup_{y\in\cK}|\p^{\le1}G^\alpha(0,y)|]\\
&\quad+\w{t}^\eta\ln^2(2+t)\sum_{|a|\le10}\sup_{(s,y)\in[0,t]\times\overline{\R^2\setminus\cK_2}}\w{y}^{1/2}\cW_{1,1}(s,y)|Z^aG^\alpha(s,y)|\\
&\quad+\ln(2+t)\sum_{|a|\le9}\sup_{(s,y)\in[0,t]\times\overline{\R^2\setminus\cK_2}}\w{y}^{1/2}\cW_{3/2+\eta,1}(s,y)|Z^aG^r(s,y)|.
\end{split}
\end{equation}
\end{lemma}
\begin{remark}
From the last two lines in \eqref{loc:ibvp} and \eqref{dpw:ibvp}, it is obvious to know that the weight
$\cW_{1,1}(s,y)$ before the divergence part $G^\alpha$ is weaker than the weight $\cW_{3/2+\eta,1}(s,y)$ before $G^r$.
\end{remark}
\begin{proof}
Let $w_1^b$, $w_2^b$, $w_2^c$, $\tilde v_2^b$ be defined by \eqref{loc:ibvp:pf1}, \eqref{loc:ibvp:pf3} and \eqref{loc:ibvp:pf5}, respectively.

For the estimate of $w_1^b$ in $\cK_R$ with $R>4$, it follows from the fact of $\supp_x\Box w_1^b\subset\{x:|x|\le3\}$
and \eqref{loc:decay} with $m=3$ that
\begin{equation}\label{dpw:pf1}
\begin{split}
&\w{t}\|\p w_1^b\|_{L^\infty(\cK_R)}\ls\w{t}\|\p w_1^b\|_{H^2(\cK_{R+1})}\\
&\ls\|(w_0,w_1)\|_{H^3(\cK)}+\ln(2+t)\sum_{|a|\le2}\sup_{s\in[0,t]}\w{s}\|\p^a\Box w_1^b(s)\|_{L^2(\cK)}\\
&\ls\|(w_0,w_1)\|_{H^3(\cK)}+\ln(2+t)\sum_{\alpha=0}^2\sum_{|a|\le2}\sup_{s\in[0,t]}\w{s}\|\p^a(G^r,\p^{\le1}G^\alpha)(s)\|_{L^2(\cK_3)}.
\end{split}
\end{equation}
Next, we treat $w_1^b$ in the region $|x|\ge R$.
Let $w_1^c=\chi_{[3,4]}w_1^b$ be a function on $\R^2$.
Then $w_1^c$ solves the Cauchy problem
\begin{equation*}
\left\{
\begin{aligned}
&\Box w_1^c=-[\Delta,\chi_{[3,4]}]w_1^b,\\
&(w_1^c,\p_tw_1^c)(0,x)=(\chi_{[3,4]}w_0,\chi_{[3,4]}w_1).
\end{aligned}
\right.
\end{equation*}
By $\supp_x[\Delta,\chi_{[3,4]}]w_1^b\subset\{x:|x|\le4\}$, \eqref{dpw:ivp} with $H=0$ and \eqref{dpw:compact},
we can arrive at
\begin{equation}\label{dpw:pf2}
\begin{split}
&\quad\; \w{x}^{1/2}\w{t-|x|}|\p w_1^c|\\
&\ls\|(w_0,w_1)\|_{W^{2,\infty}(\cK)}+\sum_{|a|\le6}\sup_{(s,y)\in[0,t]\times\R^2}\w{s}|\p^a[\Delta,\chi_{[3,4]}]w_1^b(s,y)|\\
&\ls\|(w_0,w_1)\|_{H^4(\cK)}+\sum_{|a|\le6}\sup_{s\in[0,t]}\w{s}\|\p^a[\Delta,\chi_{[3,4]}]w_1^b(s,\cdot)\|_{H^2(\R^2)}\\
&\ls\|(w_0,w_1)\|_{H^4(\cK)}+\sum_{|a|\le9}\sup_{s\in[0,t]}\w{s}\|\p^aw_1^b(s,\cdot)\|_{L^2(\cK_4)}\\
&\ls\|(w_0,w_1)\|_{H^9(\cK)}+\ln(2+t)\sum_{|a|\le8}\sup_{s\in[0,t]}\w{s}\|\p^a(G^r,\p^{\le1}G^\alpha)(s)\|_{L^2(\cK_3)},
\end{split}
\end{equation}
where the last three inequalities are derived from the Sobolev embedding and \eqref{loc:decay} with $m=9$.
Note that $\p w_1^b=\p w_1^c$ holds in the region $|x|\ge4$, then \eqref{dpw:pf1} with $R>4$ and \eqref{dpw:pf2} lead to that for $x\in\cK$,
\begin{equation}\label{dpw:pf3}
\w{x}^{1/2}\w{t-|x|}|\p w_1^b|
\ls\|(w_0,w_1)\|_{H^9(\cK)}+\ln(2+t)\sum_{|a|\le8}\sup_{s\in[0,t]}\w{s}\|\p^a(G^r,\p^{\le1}G^\alpha)(s)\|_{L^2(\cK_3)}.
\end{equation}

We now deal with $w_2^b$.
From \eqref{loc:ibvp:pf6}, it holds that $\p w_2^b=\chi_{[1,2]}(x)\p w_2^c+\p(\chi_{[1,2]}(x))w_2^c+\p\tilde v_2^b$.
Due to $\supp_x\Box\tilde v_2^b\subset\{x:|x|\le2\}$, similarly to \eqref{dpw:pf1}-\eqref{dpw:pf3} for $w_1^b$,
one then has that for $\tilde v_2^b$,
\begin{equation}\label{dpw:pf4}
\begin{split}
\w{t}\|\p \tilde v_2^b\|_{L^\infty(\cK_R)}
&\ls\w{t}\|\p \tilde v_2^b\|_{H^2(\cK_{R+1})}\\
&\ls\ln(2+t)\sum_{|a|\le2}\sup_{s\in[0,t]}\w{s}\|\p^a[\Delta,\chi_{[1,2]}]w_2^c(s)\|_{L^2(\cK_2)}\\
&\ls\ln(2+t)\sum_{|a|\le3}\sup_{s\in[0,t]}\w{s}\|\p^aw_2^c(s)\|_{L^\infty(|y|\le2)}.
\end{split}
\end{equation}

To estimate $\p \tilde v_2^b$ in the region $\cK\cap\{|x|\ge3\}$, we now set $v_2^c=\chi_{[2,3]}\tilde v_2^b$.
It follows from $\supp_x\Box v_2^c=\supp_x[\Delta,\chi_{[2,3]}]\tilde v_2^b\subset\{x:|x|\le3\}$ and \eqref{dpw:compact} that
\begin{equation}\label{dpw:pf5}
\begin{split}
\w{x}^{1/2}\w{t-|x|}|\p v_2^c|
&\ls\sum_{|a|\le6}\sup_{s\in[0,t]}\w{s}\|\p^a[\Delta,\chi_{[2,3]}]\tilde v_2^b(s)\|_{H^2(\R^2)}\\
&\ls\sum_{|a|\le9}\sup_{s\in[0,t]}\w{s}\|\p^a\tilde v_2^b(s)\|_{L^2(\cK_3)}\\
&\ls\ln(2+t)\sum_{|a|\le9}\sup_{s\in[0,t]}\w{s}\|\p^aw_2^c(s)\|_{L^\infty(|y|\le2)}.
\end{split}
\end{equation}
Note that $\p \tilde v_2^b=\p v_2^c$ holds in the region $|x|\ge3$.
Combining \eqref{dpw:pf4} with \eqref{dpw:pf5} yields
\begin{equation}\label{dpw:pf6}
\begin{split}
\w{x}^{1/2}\w{t-|x|}|\p \tilde v_2^b|
\ls\ln(2+t)\sum_{|a|\le9}\sup_{s\in[0,t]}\w{s}\|\p^aw_2^c(s)\|_{L^\infty(|y|\le2)}.
\end{split}
\end{equation}
We now treat the term $\p\p^aw_2^c$ with $|a|\le8$ on the right hand side of \eqref{dpw:pf6}.
It can be concluded from \eqref{dpw:ivp} with $\mu=\nu=\eta/2$ that
\begin{equation}\label{dpw:pf7}
\begin{split}
\sum_{|a|\le8}\w{x}^{1/2}\w{t-|x|}|\p\p^aw_2^c|
\ls\sum_{|a|\le8}\sum_{\alpha=0}^2\sup_{y\in\cK}|Z^a(G^r,\p^{\le1}G^\alpha)(0,y)|\\
+\sum_{|a|\le9}\sup_{(s,y)\in[0,t]\times\overline{\R^2\setminus\cK_2}}\w{y}^{1/2}\cW_{1+\eta,1}(s,y)|Z^a(G^r,\p^{\le1}G^\alpha)(s,y)|.
\end{split}
\end{equation}
Collecting \eqref{loc:ibvp:pf9}, \eqref{dpw:pf3}, \eqref{dpw:pf6}, \eqref{dpw:pf7} and $\p w_2^b=\chi_{[1,2]}(x)\p w_2^c+\p(\chi_{[1,2]}(x))w_2^c+\p\tilde v_2^b$ yields \eqref{dpw:ibvp}.
\end{proof}

Thirdly, we establish the pointwise spacetime decay estimate of solution itself and
time decay estimate for the inhomogeneous linear wave equation.

\begin{lemma}
{\bf (Spacetime decay of $w$ and time decay of $\|\p^{\le1}\p_tw\|_{L^2(\cK_R)}$)}
Let $w(t,x)$ be the solution of the IBVP
\begin{equation*}
\left\{
\begin{aligned}
&\Box w=F,\qquad \qquad(t,x)\in[0,\infty)\times\cK,\\
&\frac{\p}{\p\bn}w=0,\qquad\qquad (t,x)\in[0,\infty)\times\p\cK,\\
&(w,\p_tw)(0,x)=(w_0,w_1)(x),\quad x\in\cK,
\end{aligned}
\right.
\end{equation*}
where $(w_0,w_1)$ has compact support and $\supp_x F(t,x)\subset\{x: |x|\le t+M_0\}$.
Then we have that for any $\mu,\nu>0$ with $\mu+\nu<1/2$,
\begin{equation}\label{pw:ibvp}
\begin{split}
&\w{t+|x|}^{1/2}\w{t-|x|}^{\mu}|w|\\
&\ls\|(w_0,w_1)\|_{H^5(\cK)}
+\sum_{|a|\le4}[\sup_{y\in\cK}|\p^aF(0,y)|+\sup_{s\in[0,t]}\w{s}^{1/2+\mu+\nu}\|\p^aF(s)\|_{L^2(\cK_3)}]\\
&\quad +\sum_{|a|\le5}\sup_{(s,y)\in[0,t]\times\overline{\R^2\setminus\cK_2}}\w{y}^{1/2}\cW_{1+\mu+\nu,1}(s,y)|\p^aF(s,y)|
\end{split}
\end{equation}
and
\begin{equation}\label{locdt}
\begin{split}
&\w{t}\|\p^{\le1}\p_tw\|_{L^2(\cK_R)}\\
&\ls\|(w_0,w_1)\|_{H^2(\cK)}
+\sum_{|a|\le1}[\sup_{y\in\cK}|\p^aF(0,y)|+\sup_{s\in[0,t]}\w{s}^{1+\mu}\|\p^aF(s)\|_{L^2(\cK_3)}]\\
&\quad +\sum_{|a|\le2}\sup_{(s,y)\in[0,t]\times\overline{\R^2\setminus\cK_2}}\w{y}^{1/2}\cW_{1+\mu+\nu,1}(s,y)|Z^aF(s,y)|.
\end{split}
\end{equation}
\end{lemma}
\begin{proof}
Although the proof of \eqref{pw:ibvp} is similar to that for the Dirichlet boundary problem \cite{HouYinYuan24},
we still give the details for completeness and due to the different boundary condition and slightly different weight.
Let $w_1^b$ and $w_2^b$ be the solutions of
\begin{equation}\label{pw:ibvp:pf1}
\begin{split}
&\Box w_1^b:=(1-\chi_{[2,3]}(x))F(t,x),\\
&\Box w_2^b:=\chi_{[2,3]}(x)F(t,x),\\
&\frac{\p}{\p\bn}w_1^b|_{[0,\infty)\times\p\cK}=\frac{\p}{\p\bn}w_2^b|_{[0,\infty)\times\p\cK}=0,\\
&(w_1^b,\p_tw_1^b)(0,x)=(w_0,w_1)(x),\quad (w_2^b,\p_tw_2^b)(0,x)=(0,0),
\end{split}
\end{equation}
respectively.
Then one can find that $w=w_1^b+w_2^b$.
For the estimate of $w_1^b$ in $\cK_R$, it follows from the fact of $\supp_x\Box w_1^b\subset\{x:|x|\le3\}$ and \eqref{loc:decay+} with $m=2$ that
\begin{equation}\label{pw:ibvp:pf2}
\begin{split}
&\w{t}^{1/2+\mu}\|w_1^b\|_{L^\infty(\cK_R)}\ls\w{t}^{1/2+\mu}\|w_1^b\|_{H^2(\cK_{R+1})}\\
&\ls\|(w_0,w_1)\|_{H^2(\cK)}+\sum_{|a|\le1}\sup_{s\in[0,t]}\w{s}^{1/2+\mu+\nu}\|\p^a[(1-\chi_{[2,3]}(x))F]\|_{L^2(\cK)}\\
&\ls\|(w_0,w_1)\|_{H^2(\cK)}+\sum_{|a|\le1}\sup_{s\in[0,t]}\w{s}^{1/2+\mu+\nu}\|\p^aF(s)\|_{L^2(\cK_3)}.
\end{split}
\end{equation}
We now estimate $w_1^b$ in the region $|x|\ge4$.
To this end, let $w_1^c=\chi_{[3,4]}(x)w_1^b$, then one has 
\begin{equation*}
\begin{split}
&\Box w_1^c=-[\Delta,\chi_{[3,4]}]w_1^b,\\
&(w_1^c,\p_tw_1^c)(0,x)=(\chi_{[3,4]}(x)w_0,\chi_{[3,4]}(x)w_1).
\end{split}
\end{equation*}
It can be deduced from $\supp_x[\Delta,\chi_{[3,4]}]w_1^b\subset\{x:|x|\le4\}$, \eqref{pw:ivp2} with $H=0$ and \eqref{pw:compact} that
\begin{equation}\label{pw:ibvp:pf3}
\begin{split}
&\w{t+|x|}^{1/2}\w{t-|x|}^{\mu}|w_1^c|\\
&\ls\|(w_0,w_1)\|_{W^{1,\infty}(\cK)}+\sum_{|a|\le2}\sup_{(s,y)\in[0,t]\times\R^2}\w{s}^{1/2+\mu+\nu/2}|\p^a[\Delta,\chi_{[3,4]}]w_1^b|\\
&\ls\|(w_0,w_1)\|_{H^3(\cK)}+\sum_{|a|\le2}\sup_{s\in[0,t]}\w{s}^{1/2+\mu+\nu/2}\|\p^a[\Delta,\chi_{[3,4]}]w_1^b\|_{H^2(\cK)}\\
&\ls\|(w_0,w_1)\|_{H^5(\cK)}+\sum_{|a|\le4}\sup_{s\in[0,t]}\w{s}^{1/2+\mu+\nu}\|\p^aF(s)\|_{L^2(\cK_3)},
\end{split}
\end{equation}
where the last two inequalities are derived from the Sobolev embedding and \eqref{loc:decay+} with $m=5$.
Notice that $w_1^c=w_1^b$ holds in the region $|x|\ge4$.
Then \eqref{pw:ibvp:pf2} with $R>4$ and \eqref{pw:ibvp:pf3} lead to
\begin{equation}\label{pw:ibvp:pf4}
\w{t+|x|}^{1/2}\w{t-|x|}^{\mu}|w_1^b|
\ls\|(w_0,w_1)\|_{H^5(\cK)}+\sum_{|a|\le4}\sup_{s\in[0,t]}\w{s}^{1/2+\mu+\nu}\|\p^aF(s)\|_{L^2(\cK_3)}.
\end{equation}
Next, we turn to the estimate of $w_2^b$.
Denote $w_2^c$ the solution of the Cauchy problem
\begin{equation}\label{pw:ibvp:pf5}
\Box w_2^c=\left\{
\begin{aligned}
&\chi_{[2,3]}(x)F(t,x), &&\qquad x\in\cK=\R^2\setminus\cO,\\
&0, &&\qquad x\in\cO
\end{aligned}
\right.
\end{equation}
with the initial data $(w_2^c,\p_tw_2^c)(0,x)=(0,0)$.
Set $v_2^b=\chi_{[1,2]}(x)w_2^c$ on $\cK$ and then $v_2^b$ is the solution of the IBVP
\begin{equation}\label{pw:ibvp:pf6}
\begin{split}
&\Box v_2^b=\chi_{[1,2]}\Box w_2^c-[\Delta,\chi_{[1,2]}]w_2^c=\chi_{[2,3]}F-[\Delta,\chi_{[1,2]}]w_2^c,\\
&\frac{\p}{\p\bn}v_2^b|_{[0,\infty)\times\p\cK}=0,\quad (v_2^b,\p_tv_2^b)(0,x)=(0,0).
\end{split}
\end{equation}
In addition, let $\tilde v_2^b$ be the solution of the IBVP
\begin{equation}\label{pw:ibvp:pf7}
\Box\tilde v_2^b=[\Delta,\chi_{[1,2]}]w_2^c,\quad \frac{\p}{\p\bn}\tilde v_2^b|_{[0,\infty)\times\p\cK}=0,\quad (\tilde v_2^b,\p_t\tilde v_2^b)(0,x)=(0,0).
\end{equation}
Thereby, \eqref{pw:ibvp:pf1}, \eqref{pw:ibvp:pf6} and \eqref{pw:ibvp:pf7} imply $w_2^b=\chi_{[1,2]}(x)w_2^c+\tilde v_2^b$.
Note that by $\supp_x\Box\tilde v_2^b\subset\{x:|x|\le2\}$, similarly to the estimates \eqref{pw:ibvp:pf2} and \eqref{pw:ibvp:pf3} for $w_1^b$,
then we have
\begin{equation}\label{pw:ibvp:pf8}
\begin{split}
\w{t}^{1/2+\mu}\|\tilde v_2^b\|_{L^\infty(\cK_R)}
&\ls\w{t}^{1/2+\mu}\|\tilde v_2^b\|_{H^2(\cK_{R+1})}\\
&\ls\sum_{|a|\le1}\sup_{s\in[0,t]}\w{s}^{1/2+\mu+\nu/2}\|\p^a[\Delta,\chi_{[1,2]}]w_2^c]\|_{L^2(\cK)}\\
&\ls\sum_{|a|\le2}\sup_{s\in[0,t]}\w{s}^{1/2+\mu+\nu/2}\|\p^aw_2^c\|_{L^\infty(|y|\le2)}.
\end{split}
\end{equation}
Denote $v_2^c=\chi_{[2,3]}\tilde v_2^b$, which solves the Cauchy problem $\Box v_2^c=-[\Delta,\chi_{[2,3]}]\tilde v_2^b$ with zero initial data.
In addition, $\supp_x\Box v_2^c\subset\{x:|x|\le3\}$ and \eqref{pw:compact} yield
\begin{equation}\label{pw:ibvp:pf9}
\begin{split}
\w{t+|x|}^{1/2}\w{t-|x|}^{\mu}|v_2^c|
&\ls\sum_{|a|\le2}\sup_{(s,y)\in[0,t]\times\R^2}\w{s}^{1/2+\mu+\nu/4}|\p^a[\Delta,\chi_{[2,3]}]\tilde v_2^b|\\
&\ls\sum_{|a|\le5}\sup_{s\in[0,t]}\w{s}^{1/2+\mu+\nu/4}\|\p^a\tilde v_2^b\|_{L^2(\cK_3)}\\
&\ls\sum_{|a|\le5}\sup_{s\in[0,t]}\w{s}^{1/2+\mu+\nu/2}\|\p^aw_2^c\|_{L^\infty(|y|\le2)}.
\end{split}
\end{equation}
By virtue of $v_2^c=\tilde v_2^b$ in the region $|x|\ge3$, collecting \eqref{pw:ibvp:pf8} with $R>3$ and \eqref{pw:ibvp:pf9} shows
\begin{equation}\label{pw:ibvp:pf10}
\w{t+|x|}^{1/2}\w{t-|x|}^{\mu}|\tilde v_2^b|
\ls\sum_{|a|\le5}\sup_{s\in[0,t]}\w{s}^{1/2+\mu+\nu/2}\|\p^aw_2^c\|_{L^\infty(|y|\le2)}.
\end{equation}
Note that $w_2^b=\chi_{[1,2]}(x)w_2^c+\tilde v_2^b$, it suffices to deal with the estimate of $w_2^c$, which solves \eqref{pw:ibvp:pf5}.
Applying \eqref{pw:ivp1} to $\Box\p^aw_2^c=\p^a(\chi_{[2,3]}(x)F(t,x))$ yields
\begin{equation}\label{pw:ibvp:pf11}
\begin{split}
\w{t+|x|}^{1/2}\w{t-|x|}^{\mu+\nu/2}\sum_{|a|\le5}|\p^aw_2^c|\ls\sum_{|a|\le4}\sup_{y\in\cK}|\p^aF(0,y)|\\
+\sum_{|a|\le5}\sup_{(s,y)\in[0,t]\times\R^2}\w{y}^{1/2}\cW_{1+\mu+\nu,1}(s,y)|\p^a(\chi_{[2,3]}F)(s,y)|,
\end{split}
\end{equation}
where the term $\ds\sup_{y\in\cK}|\p^aF(0,y)|$ comes from the initial data of $\p^aw_2^c$.
Thus, \eqref{pw:ibvp} can be achieved by \eqref{pw:ibvp:pf4}, \eqref{pw:ibvp:pf10} and \eqref{pw:ibvp:pf11}.

Finally, we start to prove \eqref{locdt}.
Applying \eqref{loc:decay+} to $w_1^b$ in \eqref{pw:ibvp:pf1} with $\eta=1$ yields
\begin{equation}\label{locdt:pf1}
\w{t}\|\p^{\le1}\p_tw_1^b\|_{L^2(\cK_R)}\ls\|(w_0,w_1)\|_{H^2(\cK)}+\sum_{|a|\le1}\sup_{s\in[0,t]}\w{s}^{1+\mu}\|\p^aF(s)\|_{L^2(\cK_3)}.
\end{equation}
Note that $\p_tw_2^b=\chi_{[1,2]}(x)\p_tw_2^c+\p_t\tilde v_2^b$.
Set $\phi=\p_t\tilde v_2^b$ and then
\begin{equation}\label{locdt:pf2}
\Box\phi=[\Delta,\chi_{[1,2]}]\p_tw_2^c,\quad \frac{\p}{\p\bn}\phi|_{[0,\infty)\times\p\cK}=0,\quad (\phi,\p_t\phi)(0,x)=(0,0).
\end{equation}
Since the estimate of $\p_t\tilde v_2^b=\phi$ is similar to \eqref{locdt:pf1} in terms of \eqref{locdt:pf2},
we have
\begin{equation}\label{locdt:pf3}
\w{t}\|\p^{\le1}\p_t\tilde v_2^b\|_{L^2(\cK_R)}=\w{t}\|\p^{\le1}\phi\|_{L^2(\cK_R)}
\ls\sum_{|a|\le1}\sup_{s\in[0,t]}\w{s}^{1+\mu}\|\p^a\p_tw_2^c(s,y)\|_{L^\infty(|y|\le2)}.
\end{equation}
By using \eqref{dpw:ivp} with \eqref{pw:ibvp:pf5}, one can get
\begin{equation*}
\begin{split}
\w{x}^{1/2}\w{t-|x|}^{1+\mu}\sum_{|a|\le1}|\p_t\p^aw_2^c|\ls\sum_{|a|\le1}\sup_{y\in\cK}|\p^aF(0,y)|\\
+\sum_{|a|\le2}\sup_{(s,y)\in[0,t]\times\overline{\R^2\setminus\cK_2}}\w{y}^{1/2}\cW_{1+\mu+\nu,1}(s,y)|Z^aF(s,y)|
\end{split}
\end{equation*}
This, together with \eqref{locdt:pf1} and \eqref{locdt:pf3}, derives \eqref{locdt}.
\end{proof}

Finally, we prove the pointwise spacetime decay estimate of the first order derivatives of solution to the linear wave equation
with a special divergence form source term.
\begin{lemma}
{\bf (Spacetime decay estimate of $\p w$)}
Let $w(t,x)$ be the solution of the IBVP
\begin{equation*}
\left\{
\begin{aligned}
&\Box w=\p_tG^0,\qquad \qquad(t,x)\in(0,\infty)\times\cK,\\
&\frac{\p}{\p\bn}w=0,\qquad\qquad\quad (t,x)\in[0,\infty)\times\p\cK,\\
&(w,\p_tw)(0,x)=(w_0,w_1)(x),\quad x\in\cK,
\end{aligned}
\right.
\end{equation*}
where $(w_0,w_1)$ has compact support and $\supp_x G^0(t,x)\subset\{x: |x|\le t+M_0\}$.
Then it holds that for any $\eta\in(0,1/2)$,
\begin{equation}\label{dtpw:ibvp}
\begin{split}
&\w{x}^{1/2}\w{t-|x|}|\p w|\ls\|(w_0,w_1,G^0(0))\|_{H^9(\cK)}+\sum_{|a|\le9}\sup_{s\in[0,t]}\w{s}^{1+\eta}\|\p^aG^0(s)\|_{L^2(\cK_3)}\\
&\quad+\sum_{|a|\le9}\sup_{y\in\cK}|\p^aG^0(0,y)|
+\sum_{|a|\le10}\sup_{(s,y)\in[0,t]\times\overline{\R^2\setminus\cK_2}}\w{y}^{1/2}\cW_{1+\eta,1}(s,y)|Z^aG^0(s,y)|.
\end{split}
\end{equation}
\end{lemma}
\begin{proof}
Although the proof of \eqref{dtpw:ibvp} is much easier than that of \eqref{dpw:ibvp},
we still give the details for reader's convenience.
Let $w_1^b$ and $w_2^b$ be the solutions of
\begin{equation*}
\begin{split}
&\Box w_1^b:=(1-\chi_{[2,3]}(x))\p_tG^0,\\
&\Box w_2^b:=\p_t(\chi_{[2,3]}(x)G^0),\\
&\frac{\p}{\p\bn}w_1^b|_{[0,\infty)\times\p\cK}=\frac{\p}{\p\bn}w_2^b|_{[0,\infty)\times\p\cK}=0,\\
&(w_1^b,\p_tw_1^b)(0,x)=(w_0,w_1)(x),\quad (w_2^b,\p_tw_2^b)(0,x)=(0,0),
\end{split}
\end{equation*}
respectively.
Then we have $w=w_1^b+w_2^b$.
The estimate of $w_1^b$ in $\cK_R$ can be obtained by \eqref{loc:decay+} with $m=3$ and the fact of $\supp_x\Box w_1^b\subset\{x:|x|\le3\}$
as follows
\begin{equation}\label{dtpw:pf1}
\begin{split}
&\w{t}\|\p w_1^b\|_{L^\infty(\cK_R)}\ls\w{t}\|\p w_1^b\|_{H^2(\cK_{R+1})}\\
&\ls\|(w_0,w_1)\|_{H^3(\cK)}+\sum_{|a|\le2}\sup_{s\in[0,t]}\w{s}^{1+\eta}\|\p^a\Box w_1^b(s)\|_{L^2(\cK)}\\
&\ls\|(w_0,w_1)\|_{H^3(\cK)}+\sum_{\alpha=0}^2\sum_{|a|\le3}\sup_{s\in[0,t]}\w{s}^{1+\eta}\|\p^aG^0(s)\|_{L^2(\cK_3)}.
\end{split}
\end{equation}
Next we treat $w_1^b$ in the region $|x|\ge4$. Let $w_1^c=\chi_{[3,4]}(x)w_1^b$. Then
\begin{equation*}
\Box w_1^c=-[\Delta,\chi_{[3,4]}]w_1^b,\quad (w_1^c,\p_tw_1^c)(0,x)=(\chi_{[3,4]}(x)w_0,\chi_{[3,4]}(x)w_1).
\end{equation*}
From $\supp_x[\Delta,\chi_{[3,4]}]w_1^b\subset\{x:|x|\le4\}$, \eqref{dpw:ivp} with $H=0$ and \eqref{dpw:compact},
one can see that
\begin{equation}\label{dtpw:pf2}
\begin{split}
&\quad\; \w{x}^{1/2}\w{t-|x|}|\p w_1^c|\\
&\ls\|(w_0,w_1)\|_{W^{2,\infty}(\cK)}+\sum_{|a|\le6}\sup_{(s,y)\in[0,t]\times\R^2}\w{s}|\p^a[\Delta,\chi_{[3,4]}]w_1^b(s,y)|\\
&\ls\|(w_0,w_1)\|_{H^4(\cK)}+\sum_{|a|\le6}\sup_{s\in[0,t]}\w{s}\|\p^a[\Delta,\chi_{[3,4]}]w_1^b(s,\cdot)\|_{H^2(\R^2)}\\
&\ls\|(w_0,w_1)\|_{H^4(\cK)}+\sum_{|a|\le9}\sup_{s\in[0,t]}\w{s}\|\p^aw_1^b(s,\cdot)\|_{L^2(\cK_4)}\\
&\ls\|(w_0,w_1)\|_{H^9(\cK)}+\sum_{|a|\le9}\sup_{s\in[0,t]}\w{s}^{1+\eta}\|\p^aG^0(s)\|_{L^2(\cK_3)},
\end{split}
\end{equation}
where the last three inequalities have used the Sobolev embedding and \eqref{loc:decay+} with $m=9$.
By the fact of $w_1^c=w_1^b$ in the region $|x|\ge4$, \eqref{dtpw:pf1} with $R>4$ and \eqref{dtpw:pf2},
we obtain that for $x\in\cK$,
\begin{equation}\label{dtpw:pf3}
\w{x}^{1/2}\w{t-|x|}|\p w_1^b|
\ls\|(w_0,w_1)\|_{H^9(\cK)}+\sum_{|a|\le9}\sup_{s\in[0,t]}\w{s}^{1+\eta}\|\p^aG^0(s)\|_{L^2(\cK_3)}.
\end{equation}
We now treat $w_2^b$.  Let $w_2^c$ solve the Cauchy problem
\begin{equation}\label{dtpw:pf4}
\Box w_2^c=\left\{
\begin{aligned}
&\p_t(\chi_{[2,3]}(x)G^0),\qquad&& x\in\cK=\R^2\setminus\cO,\\
&0,&& x\in\overline{\cO}
\end{aligned}
\right.
\end{equation}
with $(w_2^c, \p_tw_2^c)(0,x)=(0,\chi_{[2,3]}(x)G^0(0,x))$.
Define $v_2^b=\chi_{[1,2]}(x)w_2^c$ on $\cK$ and then $v_2^b$ satisfies
\begin{equation*}
\Box v_2^b=\chi_{[1,2]}\Box w_2^c-[\Delta,\chi_{[1,2]}]w_2^c,\quad \frac{\p}{\p\bn}v_2^b|_{[0,\infty)\times\p\cK}=0,\quad (v_2^b,\p_tv_2^b)(0,x)=(0,\chi_{[2,3]}(x)G^0(0,x)).
\end{equation*}
Thus, $w_2^b=\chi_{[1,2]}(x)w_2^c+\tilde v_2^b$ holds, where $\tilde v_2^b$ is the solution of the IBVP
\begin{equation*}
\Box\tilde v_2^b=[\Delta,\chi_{[1,2]}]w_2^c,
\quad\frac{\p}{\p\bn}\tilde v_2^b|_{[0,\infty)\times\p\cK}=0,\quad (\tilde v_2^b,\p_t\tilde v_2^b)(0,x)=(0,-\chi_{[2,3]}(x)G^0(0,x)).
\end{equation*}
Due to $\supp_x\Box\tilde v_2^b\subset\{x:|x|\le2\}$, similarly to \eqref{dtpw:pf1}-\eqref{dtpw:pf3} for $w_1^b$,
one can get the estimate of $\tilde v_2^b$ as follows
\begin{equation}\label{dtpw:pf5}
\begin{split}
\w{t}\|\p \tilde v_2^b\|_{L^\infty(\cK_R)}
&\ls\w{t}\|\p \tilde v_2^b\|_{H^2(\cK_{R+1})}\\
&\ls\|G^0(0,x)\|_{H^2(\cK)}+\sum_{|a|\le2}\sup_{s\in[0,t]}\w{s}^{1+\eta/2}\|\p^a[\Delta,\chi_{[1,2]}]w_2^c(s)\|_{L^2(\cK_2)}\\
&\ls\|G^0(0,x)\|_{H^2(\cK)}+\sum_{|a|\le3}\sup_{s\in[0,t]}\w{s}^{1+\eta/2}\|\p^aw_2^c(s)\|_{L^\infty(|y|\le2)},
\end{split}
\end{equation}
which yields the estimate of $\p \tilde v_2^b$ in the region $\cK_R$ with $R>3$.
To estimate $\p \tilde v_2^b$ in the region $\cK\cap\{x:|x|\ge3\}$, set $v_2^c=\chi_{[2,3]}\tilde v_2^b$.
It follows from $\supp_x\Box v_2^c=\supp_x[\Delta,\chi_{[2,3]}]\tilde v_2^b\subset\{x:|x|\le3\}$ and \eqref{dpw:compact} that
\begin{equation}\label{dtpw:pf6}
\begin{split}
\w{x}^{1/2}\w{t-|x|}|\p v_2^c|
&\ls\sum_{|a|\le6}\sup_{s\in[0,t]}\w{s}\|\p^a[\Delta,\chi_{[2,3]}]\tilde v_2^b(s)\|_{H^2(\R^2)}\\
&\ls\sum_{|a|\le9}\sup_{s\in[0,t]}\w{s}\|\p^a\tilde v_2^b(s)\|_{L^2(\cK_3)}\\
&\ls\|G^0(0,x)\|_{H^8(\cK)}+\sum_{|a|\le9}\sup_{s\in[0,t]}\w{s}^{1+\eta/2}\|\p^aw_2^c(s)\|_{L^\infty(|y|\le2)}.
\end{split}
\end{equation}
By $\p \tilde v_2^b=\p v_2^c$ in the region $|x|\ge3$, \eqref{dtpw:pf5} and \eqref{dtpw:pf6}, we can arrive at
\begin{equation}\label{dtpw:pf7}
\begin{split}
\w{x}^{1/2}\w{t-|x|}|\p \tilde v_2^b|
\ls\|G^0(0,x)\|_{H^8(\cK)}+\sum_{|a|\le9}\sup_{s\in[0,t]}\w{s}^{1+\eta/2}\|\p^aw_2^c(s)\|_{L^\infty(|y|\le2)}.
\end{split}
\end{equation}
Next we treat the term $\p^aw_2^c$ with $|a|\le9$ on the right hand side of \eqref{dtpw:pf7}.
By the definition \eqref{dtpw:pf4} for $w_2^c$, one has $w_2^c=\p_tw_{div}^c$, where $w_{div}^c$ is the solution of
$\Box w_{div}^c=\chi_{[2,3]}(x)G^0$ with $(w_{div}^c,\p_tw_{div}^c)(0,x)=(0,0)$.
Applying \eqref{dpw:ivp} with $\mu=\nu=\eta/2$ to $\Box w_{div}^c$ yields
\begin{equation}\label{dtpw:pf8}
\begin{split}
&\sum_{|a|\le9}\w{x}^{1/2}\w{t-|x|}^{1+\eta/2}|\p^aw_2^c|\\
&\ls\sum_{|a|\le9}\w{x}^{1/2}\w{t-|x|}^{1+\eta/2}|\p^a\p_tw_{div}^c|\\
&\ls\sum_{|a|\le9}\sup_{y\in\cK}|\p^aG^0(0,y)|
+\sum_{|a|\le10}\sup_{(s,y)\in[0,t]\times\overline{\R^2\setminus\cK_2}}\w{y}^{1/2}\cW_{1+\eta,1}(s,y)|Z^aG^0(s,y)|.
\end{split}
\end{equation}
Collecting \eqref{dtpw:pf3}, \eqref{dtpw:pf5}-\eqref{dtpw:pf8} with $\p w_2^b=\chi_{[1,2]}(x)\p w_2^c+\p(\chi_{[1,2]}(x))w_2^c+\p\tilde v_2^b$
derives \eqref{dtpw:ibvp}.
\end{proof}

\section{Proof of Theorem \ref{thm1}}\label{sect4}

We make the following bootstrap assumption for $t\in[0,T_\ve]$,
\begin{equation}\label{thm1-BA}
\sum_{|a|\le N}|\p Z^au|\le\ve_1\w{x}^{-1/2}\w{t-|x|}^{-1}\w{t}^{\delta_1},
\end{equation}
where $T_\ve=\ve^{\delta-2}$, $\ve_1=C_0\ve\in(\ve,1)$ will be determined later and $\delta_1=\frac{\delta}{16}\in(0,\frac{1}{16})$
with $\delta>0$ being given in Theorem \ref{thm1}.

On the other hand, $t\le T_\ve$ and the smallness of $\ve_0$ ensure that for $0<\ve\le\ve_0$,
\begin{equation}\label{thm1-time:BA}
t^{4\delta_1+1/2}\ve_1\le t^{\frac14\delta+1/2}C_0\ve\le C_0\ve_0^{\frac14\delta^2}\le1.
\end{equation}
Note that due to $\supp(u_0,u_1)\subset\{x: |x|\le M_0\}$,
the solution $u$ of problem \eqref{QWE} is supported in $\{x\in\cK:|x|\le t+M_0\}$.

In addition, in this section, we denote
\begin{equation}\label{YHCC-1-Y}
\begin{split}
Q_p^{\alpha\beta}(\p u)&=\sum_{\mu=0}^2Q^{\alpha\beta\mu}\p_{\mu}u
+\sum_{\mu,\nu=0}^2Q^{\alpha\beta\mu\nu}\p_{\mu}u\p_{\nu}u.
\end{split}
\end{equation}
Then $Q(\p u,\p^2u)$ in \eqref{YHCC-1} can be written as
\begin{equation}\label{YHCC-1-S}
\begin{split}
Q(\p u,\p^2u)&=\sum_{\alpha,\beta=0}^2Q_p^{\alpha\beta}(\p u)\p^2_{\alpha\beta}u.
\end{split}
\end{equation}

\subsection{Energy estimates}\label{Sonn-1}

At first, we establish the energy inequality for $\ds\sum_{j\le2N}\|\p\p_t^ju(t)\|^2_{L^2(\cK)}$
with $u$ being the solution of \eqref{QWE}.

\begin{lemma}\label{Son-4.1}
{\bf (Estimate for $\ds\sum_{j\le2N}\|\p\p_t^ju(t)\|^2_{L^2(\cK)}$)}
Under the assumptions of Theorem \ref{thm1}, let $u$ be the solution of \eqref{QWE} and suppose that \eqref{thm1-BA} holds.
Then we have
\begin{equation}\label{thm1-energy:time}
\sum_{j\le2N}\|\p\p_t^ju(t)\|^2_{L^2(\cK)}
\ls\ve^2+\ve_1\int_0^t\sum_{j\le2N-1}\frac{\|\p^{\le1}\p\p_t^ju(s)\|^2_{L^2(\cK)}}{(1+s)^{1/2-\delta_1}}ds.
\end{equation}
\end{lemma}
\begin{proof}
At first, applying $Z^a$ to \eqref{QWE} with \eqref{YHCC-1-S} yields
\begin{equation}\label{energy:time1}
\Box Z^au=\sum_{\alpha,\beta=0}^2Q_p^{\alpha\beta}(\p u)\p^2_{\alpha\beta}Z^au
+\sum_{\substack{b+c\le a,\\b<a}}\sum_{\alpha,\beta=0}^2C^a_{bc}\p^2_{\alpha\beta}Z^bu Z^c[Q_p^{\alpha\beta}(\p u)],
\end{equation}
where $C^a_{bc}$ are constants.
Multiplying \eqref{energy:time1} by $\p_tZ^au$ leads to
\begin{equation}\label{energy:time2}
\begin{split}
&\quad\;\frac12\p_t(|\p_tZ^au|^2+|\nabla Z^au|^2)-\dive(\p_tZ^au\nabla Z^au)\\
&=\sum_{\alpha,\beta=0}^2Q_p^{\alpha\beta}(\p u)\p^2_{\alpha\beta}Z^au\p_tZ^au
+\sum_{\substack{b+c\le a,\\b<a}}C^a_{bc}I_1^{abc},
\end{split}
\end{equation}
where
\begin{equation}\label{energy:time3}
I_1^{abc}:=\sum_{\alpha,\beta=0}^2\p_tZ^au\p^2_{\alpha\beta}Z^bu Z^c[Q_p^{\alpha\beta}(\p u)].
\end{equation}
For the first term in the second line of \eqref{energy:time2}, it follows from direct computation that
\begin{equation}\label{energy:time4}
\begin{split}
Q_p^{\alpha\beta}(\p u)\p^2_{\alpha\beta}Z^au\p_tZ^au&=\p_{\alpha}[Q_p^{\alpha\beta}(\p u)\p_{\beta}Z^au\p_tZ^au]
-\p_{\alpha}[Q_p^{\alpha\beta}(\p u)]\p_{\beta}Z^au\p_tZ^au\\
&\quad-\frac12\p_t[Q_p^{\alpha\beta}(\p u)\p_{\beta}Z^au\p_{\alpha}Z^au]
+\frac12\p_t[Q_p^{\alpha\beta}(\p u)]\p_{\beta}Z^au\p_{\alpha}Z^au,
\end{split}
\end{equation}
where the summation $\ds\sum_{\alpha,\beta=0}^2$ in \eqref{energy:time4} is omitted.
Let $Z^a=\p_t^j$ with $j=|a|$ and the notation $Z^a$ for other vector fields $Z$ is still used in the remaining part of this section.
Integrating \eqref{energy:time2} and \eqref{energy:time4} over $[0,t]\times\cK$
with the boundary conditions $\frac{\p}{\p\bn}\p_t^lu|_{[0,\infty)\times\p\cK}=0$ for any integer $l\ge0$ and the admissible
condition \eqref{YHCC-2} derives
\begin{equation}\label{energy:time5}
\begin{split}
&\|\p Z^au(t)\|^2_{L^2(\cK)}
\ls\|\p Z^au(0)\|^2_{L^2(\cK)}+\|\p u(0)\|_{L_x^\infty}\|\p Z^au(0)\|^2_{L^2(\cK)}\\
&\quad+\|\p u(t)\|_{L_x^\infty}\|\p Z^au(t)\|^2_{L^2(\cK)}
+\int_0^t\sum_{\substack{b+c\le a,\\b<a}}\|I_1^{abc}\|_{L^1(\cK)}+\|I_2^a\|_{L^1(\cK)}dxds,
\end{split}
\end{equation}
where
\begin{equation}\label{energy:time6}
I_2^a:=\frac12\p_t[Q_p^{\alpha\beta}(\p u)]\p_{\beta}Z^au\p_{\alpha}Z^au
-\p_{\alpha}[Q_p^{\alpha\beta}(\p u)]\p_{\beta}Z^au\p_tZ^au.
\end{equation}
It is pointed out that although $Z^a=\p_t^j$ has been taken here, $I_1^{abc}$ and $I_2^a$ can be still
treated for $Z=\{\p_t,\p_1,\p_2,\Omega\}$.

We next focus on the estimate of $I_1^{abc}$.
If $|b|\le N-1$, it follows from \eqref{thm1-BA} that
\begin{equation}\label{energy:time7}
\sum_{\substack{b+c\le a,\\|b|\le N-1}}\|I_1^{abc}\|_{L^1(\cK)}
\ls\ve_1(1+t)^{\delta_1-1/2}\sum_{j\le2N}\|\p\p_t^ju(t)\|^2_{L^2(\cK)}.
\end{equation}
When $|b|\ge N$, one can see that $|c|\le N$ and
\begin{equation}\label{energy:time8}
\sum_{\substack{b+c\le a,\\|b|\ge N}}\|I_1^{abc}\|_{L^1(\cK)}
\ls\ve_1(1+t)^{\delta_1-1/2}\sum_{j\le2N-1}\|\p^2\p_t^ju(t)\|^2_{L^2(\cK)}.
\end{equation}
Since the estimate of $I_2^a$ is much easier than that of $I_1^{abc}$, we omit the details here.
Substituting \eqref{energy:time7}-\eqref{energy:time8} into \eqref{energy:time5} yields \eqref{thm1-energy:time}.
\end{proof}

Next, we show the energy estimate for the local energy of  $\p^{\le1}u$ near the boundary
with $u$ being the solution of \eqref{QWE}. This will be applied to treat the higher order
energy estimates $\ds\sum_{|a|\le2N}\|\p\p^au\|_{L^2(\cK)}$ and $\ds\sum_{|a|\le2N-1}\|\p Z^au\|_{L^2(\cK)}$ later.

\begin{lemma}\label{Son-4.2}
{\bf (Energy estimate of $\p^{\le1}u$ near the boundary)}
Under the assumptions of Theorem \ref{thm1}, let $u$ be the solution of \eqref{QWE} and suppose that \eqref{thm1-BA} holds.
Then it holds that for $t\in[0,T_\ve]$,
\begin{equation}\label{thm1-loc:energyA}
\|\p^{\le1}u\|_{L^2(\cK_R)}\ls\ve\w{t}^{-1}+\ve_1^2\w{t}^{4\delta_1-1/2}.
\end{equation}
\end{lemma}
\begin{proof}
Note that  the equation in \eqref{QWE} with \eqref{YHCC-1-S} can be reformulated as
\begin{equation}\label{eqn:div}
\begin{split}
\Box u&=\frac12\sum_{\alpha,\beta,\mu=0}^2[\p_{\alpha}(Q^{\alpha\beta\mu}\p_{\beta}u\p_{\mu}u)
+\p_{\beta}(Q^{\alpha\beta\mu}\p_{\alpha}u\p_{\mu}u)
-\p_{\mu}(Q^{\alpha\beta\mu}\p_{\alpha}u\p_{\beta}u)]\\
&\quad+\sum_{\alpha,\beta,\mu,\nu=0}^2Q^{\alpha\beta\mu\nu}\p^2_{\alpha\beta}u\p_{\mu}u\p_{\nu}u.
\end{split}
\end{equation}
Applying \eqref{loc:ibvp} to \eqref{eqn:div} with $\eta=\delta_1$ yields
\begin{equation}\label{thm1-loc:energyA1}
\begin{split}
\w{t}\|\p^{\le1}u\|_{L^2(\cK_R)}&\ls\ve+\ve^2\w{t}^{\delta_1}\ln(2+t)\\
&\quad+\ln(2+t)\sup_{s\in[0,t]}\w{s}\|(|\p^2u||\p u|^2,|\p^{\le1}\p u|^2)(s)\|_{L^2(\cK_3)}\\
&\quad+\w{t}^{\delta_1}\ln^2(2+t)\sup_{(s,y)\in[0,t]\times\cK}\w{y}^{1/2}\cW_{1,1}(s,y)|Z^{\le2}\p u|^2\\
&\quad+\ln(2+t)\sup_{(s,y)\in[0,t]\times\cK}\w{y}^{1/2}\cW_{3/2+\delta_1,1}(s,y)|\p^{\le1}\p^2u||\p^{\le1}\p u|^2\\
&\ls\ve+\ve_1^2\w{t}^{3\delta_1+1/2}\ln^2(2+t),
\end{split}
\end{equation}
where we have used \eqref{initial:data-thm1}, \eqref{thm1-BA} and \eqref{thm1-time:BA} with $N\ge19$.
This completes the proof of \eqref{thm1-loc:energyA}.
\end{proof}

Based on Lemma \ref{Son-4.2}, we now treat the higher order energy $\ds\sum_{|a|\le2N}\|\p\p^au\|_{L^2(\cK)}$
for the solution $u$ of \eqref{QWE}.

\begin{lemma}\label{Son-4.3}
{\bf (Estimate of $\ds\sum_{|a|\le2N}\|\p\p^au\|_{L^2(\cK)}$)}
Under the assumptions of Theorem \ref{thm1}, let $u$ be the solution of \eqref{QWE} and suppose that \eqref{thm1-BA} holds.
Then we have that for $t\in[0,T_\ve]$,
\begin{equation}\label{thm1-energy:dtdx}
\sum_{|a|\le2N}\|\p\p^au\|_{L^2(\cK)}\ls\ve+\ve_1^2.
\end{equation}
\end{lemma}
\begin{proof}
Set $\ds E_j(t):=\sum_{k=0}^{2N-j}\|\p\p_t^ku(t)\|_{H^j(\cK)}$ with $0\le j\le2N$.
Then one can find that for $j\ge1$,
\begin{equation}\label{thm1-energy:dtdx1}
\begin{split}
E_j(t)&\ls\sum_{k=0}^{2N-j}[\|\p\p_t^ku(t)\|_{L^2(\cK)}
+\sum_{1\le|a|\le j}(\|\p_t\p_t^k\p_x^au(t)\|_{L^2(\cK)}+\|\p_x\p_t^k\p_x^au(t)\|_{L^2(\cK)})]\\
&\ls E_0(t)+E_{j-1}(t)+\sum_{k=0}^{2N-j}\sum_{2\le|a|\le j+1}\|\p_t^k\p_x^au(t)\|_{L^2(\cK)},
\end{split}
\end{equation}
where we have used the fact that
\begin{equation*}
\sum_{k=0}^{2N-j}\sum_{1\le|a|\le j}\|\p_t\p_t^k\p_x^au(t)\|_{L^2(\cK)}
\ls\sum_{k=0}^{2N-j}\sum_{|a'|\le j-1}\|\p_x\p_t^{k+1}\p_x^{a'}u(t)\|_{L^2(\cK)}
\ls E_{j-1}(t).
\end{equation*}
For the last term in \eqref{thm1-energy:dtdx1}, it is deduced from the elliptic estimate \eqref{ellip} that
\begin{equation}\label{thm1-energy:dtdx2}
\begin{split}
\|\p_t^k\p_x^au\|_{L^2(\cK)}&\ls\|\Delta_x\p_t^ku\|_{H^{|a|-2}(\cK)}+\|\p_t^ku\|_{H^{|a|-1}(\cK_{R+1})}\\
&\ls\|\p_t^{k+2}u\|_{H^{|a|-2}(\cK)}+\|\p_t^kQ(\p u,\p^2u)\|_{H^{|a|-2}(\cK)}\\
&\quad+\|\p_t^ku\|_{L^2(\cK_{R+1})}+\sum_{1\le|b|\le|a|-1}\|\p_t^k\p_x^bu\|_{L^2(\cK_{R+1})},
\end{split}
\end{equation}
where $\Delta_x=\p_t^2-\Box$ and the equation in \eqref{QWE} have been used.
Furthermore, by \eqref{thm1-BA} with $|a|+k\le2N+1$ we have
\begin{equation}\label{thm1-energy:dtdx3}
\begin{split}
\|\p_t^kQ(\p u,\p^2u)\|_{H^{|a|-2}(\cK)}
&\ls\sum_{l\le k}\ve_1(\|\p_t^l\p u\|_{H^{|a|-1}(\cK)}+\|\p_t^{l+2}u\|_{H^{|a|-2}(\cK)})\\
&\ls\ve_1[E_j(t)+E_{j-1}(t)].
\end{split}
\end{equation}
Collecting \eqref{thm1-energy:dtdx1}-\eqref{thm1-energy:dtdx3} shows that for $j\ge1$,
\begin{equation}\label{thm1-energy:dtdx4}
E_j(t)\ls E_0(t)+E_{j-1}(t)+\ve_1E_j(t)+\|u\|_{L^2(\cK_{R+1})}.
\end{equation}
Then combining \eqref{thm1-energy:time}, \eqref{thm1-loc:energyA}, \eqref{thm1-energy:dtdx4} with $j=1$ and the
smallness of $\ve_1$ yields
\begin{equation}\label{thm1-energy:dtdx5}
\begin{split}
[E_0(t)]^2&\ls\ve^2+\ve_1\int_0^t\{\w{s}^{\delta_1-1/2}[E_0(s)]^2+\ve^2\w{s}^{\delta_1-5/2}+\ve_1^4\w{s}^{9\delta_1-3/2}\}ds\\
&\ls\ve^2+\ve_1^5+\ve_1\int_0^t\w{s}^{\delta_1-1/2}[E_0(s)]^2ds.
\end{split}
\end{equation}
By the Gronwall's Lemma \ref{lem-Gronwall-1} and \eqref{thm1-time:BA}, one obtains from \eqref{thm1-energy:dtdx5}
\begin{equation*}
[E_0(t)]^2\ls(\ve^2+\ve_1^5)e^{C\ve_1\w{t}^{\delta_1+1/2}}\ls\ve^2+\ve_1^5.
\end{equation*}
This, together with \eqref{thm1-loc:energyA} and \eqref{thm1-energy:dtdx4}, completes the proof of \eqref{thm1-energy:dtdx}.
\end{proof}

At last, in terms of Lemma \ref{Son-4.3} and Lemma \ref{lem-Gronwall-1},
we derive the higher order energy $\ds\sum_{|a|\le2N-1}\|\p Z^au\|_{L^2(\cK)}$
for the solution $u$ of \eqref{QWE}.
\begin{lemma}\label{Son-4.4}
{\bf (Estimate of $\ds\sum_{|a|\le2N-1}\|\p Z^au\|_{L^2(\cK)}$)}
Under the assumptions of Theorem \ref{thm1}, let $u$ be the solution of \eqref{QWE} and suppose that \eqref{thm1-BA} holds.
Then one has that for $t\in[0,T_\ve]$,
\begin{equation}\label{thm1-energyA}
\sum_{|a|\le2N-1}\|\p Z^au\|_{L^2(\cK)}\ls(\ve+\ve_1^2)\w{t}^{1/2}.
\end{equation}
\end{lemma}
\begin{proof}
Similarly to \eqref{energy:time2}, \eqref{energy:time3}, \eqref{energy:time4} and \eqref{energy:time5}, we have that
for any $|a|\le2N-1$,
\begin{equation}\label{thm1-energyA1}
\begin{split}
&\|\p Z^au\|^2_{L^2(\cK)}\ls\ve^2+
\int_0^t\Big[\sum_{\substack{b+c\le a,\\b<a}}\|I_1^{abc}\|_{L^1(\cK)}+\|I_2^a\|_{L^1(\cK)}\Big]ds+|B^a_1|+|B^a_2|,\\
&B^a_1:=\int_0^t\int_{\p\cK}(n_1(x),n_2(x))\cdot\nabla Z^au(s,x)\p_tZ^au(s,x)d\sigma ds,\\
&B^a_2:=\int_0^t\int_{\p\cK}\sum_{j=1}^2\sum_{\beta=0}^2Q_p^{j\beta}(\p u)n_j(x)\p_tZ^au\p_{\beta}Z^aud\sigma ds,
\end{split}
\end{equation}
where $d\sigma$ is the curve measure on $\p\cK$.
According to $\p\cK\subset\overline{\cK_1}$ and the trace theorem, one obtains
\begin{equation}\label{thm1-energyA2}
\begin{split}
|B^a_1|&\ls\sum_{|b|\le|a|}\int_0^t\|(1-\chi_{[1,2]}(x))\p_t\p^bu\|_{L^2(\p\cK)}\|(1-\chi_{[1,2]}(x))\p_x\p^bu\|_{L^2(\p\cK)}ds\\
&\ls\sum_{|b|\le|a|}\int_0^t\|(1-\chi_{[1,2]}(x))\p \p^bu\|^2_{H^1(\cK)}ds\\
&\ls\sum_{|b|\le|a|+1}\int_0^t\|\p \p^bu\|^2_{L^2(\cK_2)}ds\ls(\ve+\ve_1^2)^2(1+t),
\end{split}
\end{equation}
where \eqref{thm1-energy:dtdx} has been used. Analogously, $B^a_2$ can be also treated.
On the other hand, it is easy to get
\begin{equation}\label{thm1-energyA3}
\sum_{\substack{b+c\le a,\\b<a}}\|I_1^{abc}\|_{L^1(\cK)}+\|I_2^a\|_{L^1(\cK)}
\ls\ve_1\w{t}^{\delta_1-1/2}\sum_{|a|\le2N-1}\|\p Z^au\|^2_{L^2(\cK)}.
\end{equation}
Then, \eqref{thm1-energyA1}, \eqref{thm1-energyA2} and \eqref{thm1-energyA3} ensure that there is a constant $C_1>0$ such that
\begin{equation}\label{thm1-energyA4}
\sum_{|a|\le2N-1}\|\p Z^au\|^2_{L^2(\cK)}\le C_1\ve^2
+C_1\ve_1\int_0^t\sum_{|a|\le2N-1}\frac{\|\p Z^au(s)\|^2_{L^2(\cK)}}{(1+s)^{1/2-\delta_1}}ds+C_1(\ve+\ve_1^2)^2(1+t).
\end{equation}
Applying Lemma \ref{lem-Gronwall-1} to \eqref{thm1-energyA4} together with \eqref{thm1-time:BA} yields
\begin{equation*}
\begin{split}
\sum_{|a|\le2N-1}\|\p Z^au\|^2_{L^2(\cK)}&\le e^{\frac{C_1\ve_1}{1/2+\delta_1}(1+t)^{1/2+\delta_1}}[C_1\ve^2+C_1(\ve+\ve_1^2)^2(1+t)]\\
&\ls(\ve+\ve_1^2)^2(1+t),
\end{split}
\end{equation*}
which derives \eqref{thm1-energyA}.
\end{proof}

\subsection{Local energy decay estimates and improved energy estimates}\label{sect5}

Based on the estimates in Subsection \ref{Sonn-1}, we now focus on the precise time decay rate of
local energy near the boundary.

\begin{lemma}
{\bf (Time decay of local energy $\ds\sum_{|a|\le2N-4}\|\p^au\|_{L^2(\cK_R)}$)}
Under the assumptions of Theorem \ref{thm1}, let $u$ be the solution of \eqref{QWE} and suppose that \eqref{thm1-BA} holds.
Then we have that for $t\in[0,T_\ve]$,
\begin{equation}\label{thm1-loc:energyB}
\sum_{|a|\le2N-4}\|\p^au\|_{L^2(\cK_R)}\ls(\ve+\ve_1^2)\w{t}^{-1/2-\delta_1}.
\end{equation}
\end{lemma}
\begin{proof}
Applying $\p_t^j$ to \eqref{eqn:div} yields $\ds\frac{\p}{\p\bn}\p_t^ju|_{[0,\infty)\times\p\cK}=0$ and
\begin{equation*}
\begin{split}
\Box\p_t^ju&=\frac12\sum_{\alpha,\beta,\mu=0}^2[\p_{\alpha}\p_t^j(Q^{\alpha\beta\mu}\p_{\beta}u\p_{\mu}u)
+\p_{\beta}\p_t^j(Q^{\alpha\beta\mu}\p_{\alpha}u\p_{\mu}u)
-\p_{\mu}\p_t^j(Q^{\alpha\beta\mu}\p_{\alpha}u\p_{\beta}u)]\\
&\quad+\sum_{\alpha,\beta,\mu,\nu=0}^2Q^{\alpha\beta\mu\nu}\p_t^j(\p^2_{\alpha\beta}u\p_{\mu}u\p_{\nu}u).
\end{split}
\end{equation*}
On the other hand, \eqref{Sobo:ineq} and \eqref{thm1-energyA} lead to
\begin{equation}\label{thm1-loc:energyB1}
\sum_{|a|\le2N-3}|\p Z^au|\ls(\ve+\ve_1^2)\w{x}^{-1/2}\w{t}^{1/2}.
\end{equation}
Similarly to \eqref{thm1-loc:energyA1}, one has that for $j\le2N-5$ and $t\in[0,T_\ve]$,
\begin{equation*}
\begin{split}
&\w{t}\|\p^{\le1}\p_t^ju\|_{L^2(\cK_{R_1})}\ls\ve+\ve^2\w{t}^{\delta_1}\ln(2+t)\\
&+\ln(2+t)\sup_{s\in[0,t]}\w{s}\|\p_t^j(|\p^2u||\p u|^2,|\p^{\le1}\p u|^2)(s)\|_{L^2(\cK_3)}\\
&+\w{t}^{\delta_1}\ln^2(2+t)\sum_{\substack{|a|+|b|\le2,\\ i+k\le j}}
\sup_{(s,y)\in[0,t]\times\cK}\w{y}^{1/2}\cW_{1,1}(s,y)|Z^a\p_t^i\p uZ^b\p_t^k\p u|\\
&+\ln(2+t)\sum_{\substack{|a|+|b|+|c|\le1,\\ i+k+l\le j}}\sup_{(s,y)\in[0,t]\times\cK}\w{y}^{1/2}\cW_{3/2+\delta_1,1}(s,y)
|\p^a\p_t^i\p^2u\p^b\p_t^k\p u\p^c\p_t^l\p u|\\
&\ls\ve+\ve_1(\ve+\ve_1^2)\w{t}^{2\delta_1+1}\ln^2(2+t),
\end{split}
\end{equation*}
where we have used \eqref{initial:data-thm1}, \eqref{thm1-BA}, \eqref{thm1-energy:dtdx} and \eqref{thm1-loc:energyB1}.
Thus, we can achieve from \eqref{thm1-time:BA} that
\begin{equation}\label{thm1-loc:energyB2}
\sum_{j\le2N-5}\|\p^{\le1}\p_t^ju\|_{L^2(\cK_{R_1})}\ls(\ve+\ve_1^2)\w{t}^{-1/2-\delta_1}.
\end{equation}
For $j=0,1,2,\cdots,2N-4$, denote $\ds E_j^{loc}(t):=\sum_{k\le j}\|\p_t^k\p_x^{2N-4-j}u\|_{L^2(\cK_{R+j})}$.
Then \eqref{thm1-loc:energyB2} means
\begin{equation}\label{thm1-loc:energyB3}
E_{2N-4}^{loc}(t)+E_{2N-5}^{loc}(t)\ls(\ve+\ve_1^2)\w{t}^{-1/2-\delta_1}.
\end{equation}
When $j\le2N-6$, we have $\p_x^{2N-4-j}=\p_x^2\p_x^{2N-6-j}$.
In addition, one can apply the elliptic estimate \eqref{ellip} to $(1-\chi_{[R+j,R+j+1]})\p_t^{k}u$
to obtain
\begin{equation}\label{thm1-loc:energyB4}
\begin{split}
E_j^{loc}(t)&\ls\sum_{k\le j}\|\p_x^{2N-4-j}(1-\chi_{[R+j,R+j+1]})\p_t^ku\|_{L^2(\cK)}\\
&\ls\sum_{k\le j}[\|\Delta(1-\chi_{[R+j,R+j+1]})\p_t^ku\|_{H^{2N-6-j}(\cK)}\\
&\quad+\sum_{k\le j}\|(1-\chi_{[R+j,R+j+1]})\p_t^ku\|_{H^{2N-5-j}(\cK)}]\\
&\ls\sum_{k\le j}\|(1-\chi_{[R+j,R+j+1]})\Delta\p_t^ku\|_{H^{2N-6-j}(\cK)}+\sum_{j+1\le l\le2N-4}E_l^{loc}(t)\\
&\ls\sum_{k\le j}[\|\Box\p_t^ku\|_{H^{2N-6-j}(\cK_{R+j+1})}+\|\p_t^{k+2}u\|_{H^{2N-6-j}(\cK_{R+j+1})}]\\
&\quad+\sum_{j+1\le l\le2N-4}E_l^{loc}(t),
\end{split}
\end{equation}
where the fact of $\Delta=\p_t^2-\Box$ has been used.
From \eqref{QWE}, \eqref{thm1-BA} and \eqref{thm1-energy:dtdx}, we can get that for $j\le2N-6$,
\begin{equation}\label{thm1-loc:energyB5}
\sum_{k\le j}\|\Box\p_t^ku\|_{H^{2N-6-j}(\cK_{R+j+1})}\ls\ve_1(\ve+\ve_1^2)\w{t}^{\delta_1-1}.
\end{equation}
Thus, \eqref{thm1-loc:energyB3}, \eqref{thm1-loc:energyB4} and \eqref{thm1-loc:energyB5} imply that for $j\le2N-6$,
\begin{equation*}
E_j^{loc}(t)\ls\ve_1(\ve+\ve_1^2)\w{t}^{\delta_1-1}+\sum_{j+1\le l\le2N-4}E_l^{loc}(t).
\end{equation*}
This, together with \eqref{thm1-loc:energyB3}, yields \eqref{thm1-loc:energyB}.
\end{proof}

Next, we derive the uniform smallness of the energy $\ds\sum_{|a|\le2N-6}\|\p Z^au\|_{L^2(\cK)}$,
which is an obvious improvement for Lemma \ref{Son-4.4}.
\begin{lemma}
{\bf (Precise estimate of $\ds\sum_{|a|\le2N-6}\|\p Z^au\|_{L^2(\cK)}$)}
Under the assumptions of Theorem \ref{thm1}, let $u$ be the solution of \eqref{QWE} and suppose that \eqref{thm1-BA} holds.
Then one has that for $t\in[0,T_\ve]$,
\begin{equation}\label{thm1-energyB}
\sum_{|a|\le2N-6}\|\p Z^au\|_{L^2(\cK)}\ls\ve+\ve_1^2.
\end{equation}
\end{lemma}
\begin{proof}
Analogously to \eqref{thm1-energyA1} and \eqref{thm1-energyA2}, we have that for any $|a|\le2N-6$,
\begin{equation}\label{thm1-energyB1}
\|\p Z^au\|^2_{L^2(\cK)}\ls\ve^2+\ve_1\int_0^t\sum_{|a|\le2N-6}\frac{\|\p Z^au(s)\|^2_{L^2(\cK)}}{(1+s)^{1/2-\delta_1}}ds+|B^a_1|+|B^a_2|
\end{equation}
and
\begin{equation}\label{thm1-energyB2}
\begin{split}
|B^a_1|&\ls\sum_{|b|\le|a|+1}\int_0^t\|\p \p^bu\|^2_{L^2(\cK_2)}ds\\
&\ls\int_0^t(\ve^2+\ve_1^4)(1+s)^{-1-2\delta_1}ds
\ls\ve^2+\ve_1^4,
\end{split}
\end{equation}
where \eqref{thm1-loc:energyB} has been used. Similarly, $B^a_2$ can be also treated.
Then collecting \eqref{thm1-energyB1} and \eqref{thm1-energyB2} yields
\begin{equation*}
\sum_{|a|\le2N-6}\|\p Z^au\|^2_{L^2(\cK)}\ls\ve^2+\ve_1^4
+\ve_1\int_0^t\sum_{|a|\le2N-6}\frac{\|\p Z^au(s)\|^2_{L^2(\cK)}}{(1+s)^{1/2-\delta_1}}ds,
\end{equation*}
which implies \eqref{thm1-energyB} by Gronwall's inequality.
\end{proof}

\begin{lemma}
{\bf (Uniform time decay of local energy $\sum_{|a|\le2N-9}\|\p^au\|_{L^2(\cK_R)}$)}
Under the assumptions of Theorem \ref{thm1}, let $u$ be the solution of \eqref{QWE} and suppose that \eqref{thm1-BA} holds.
Then one has that for $t\in[0,T_\ve]$,
\begin{equation}\label{thm1-loc:energyC}
\sum_{|a|\le2N-9}\|\p^au\|_{L^2(\cK_R)}\ls(\ve+\ve_1^2)\w{t}^{-1}.
\end{equation}
\end{lemma}
\begin{proof}
Collecting \eqref{Sobo:ineq} and \eqref{thm1-energyB} leads to
\begin{equation}\label{thm1-loc:energyC1}
\sum_{|a|\le2N-8}|\p Z^au|\ls(\ve+\ve_1^2)\w{x}^{-1/2}.
\end{equation}
With \eqref{thm1-loc:energyC1} instead of \eqref{thm1-loc:energyB1}, \eqref{thm1-loc:energyB2} can be improved as $t\in[0,T_\ve]$,
\begin{equation}\label{thm1-loc:energyC2}
\begin{split}
\w{t}\sum_{j\le2N-10}\|\p^{\le1}\p_t^ju\|_{L^2(\cK_{R_1})}&\ls\ve+\ve_1(\ve+\ve_1^2)\w{t}^{2\delta_1+1/2}\ln^2(2+t),\\
&\ls\ve+\ve_1^2.
\end{split}
\end{equation}
Thus, in terms of \eqref{thm1-loc:energyC2}, \eqref{thm1-loc:energyC} can be achieved as in the proof of \eqref{thm1-loc:energyB}.
\end{proof}

\subsection{Improved pointwise estimates and proof of Theorem \ref{thm1}}

At first, we establish a better pointwise spacetime estimate for
$\ds\sum_{|a|\le2N-19}|\p Z^au|$ so that the bootstrap assumption in \eqref{thm1-BA} can be closed for $t\in[0,T_\ve]$.
\begin{lemma}
Under the assumptions of Theorem \ref{thm1}, let $u$ be the solution of \eqref{QWE} and suppose that \eqref{thm1-BA} holds.
Then we have for $t\in[0,T_\ve]$
\begin{equation}\label{BA:impr}
\w{x}^{1/2}\w{t-|x|}\sum_{|a|\le2N-19}|\p Z^au|\ls(\ve+\ve_1^2)\ln(2+t).
\end{equation}
\end{lemma}
\begin{proof}
Applying $Z^a$ to \eqref{eqn:div} yields
\begin{equation}\label{BA:impr1}
\Box Z^au=\sum_{\alpha,\beta,\mu=0}^2\p_{\alpha}(Q_{abc}^{\alpha\beta\mu}\p_{\beta}Z^bu\p_{\mu}Z^cu)
+\sum_{\alpha,\beta,\mu,\nu=0}^2Q^{\alpha\beta\mu\nu}Z^a(\p^2_{\alpha\beta}u\p_{\mu}u\p_{\nu}u),
\end{equation}
where $Q_{abc}^{\alpha\beta\mu}$ are constants.

Set $\tilde Z=\chi_{[1/2,1]}(x)Z$ and $\ds\frac{\p}{\p\bn}\tilde Z^au|_{[0,\infty)\times\p\cK}=0$ holds.
By utilizing \eqref{dpw:ibvp} to $\Box\tilde Z^au=\Box(\tilde Z^a-Z^a)u+\Box Z^au$ with $\eta=\delta_1$ and $|a|\le2N-19$,
one can achieve
\begin{equation}\label{BA:impr2}
\begin{split}
&\quad\w{x}^{1/2}\w{t-|x|}|\p\tilde Z^au|
\ls\ve+\ve^2\w{t}^{\delta_1}\ln(2+t)\\
&\qquad+\ln(2+t)\sum_{|b|\le8}\sup_{s\in[0,t]}\w{s}\|\p^b(G_0^r,\p^{\le1}G^\alpha)(s)\|_{L^2(\cK_3)}\\
&\qquad+\w{t}^{\delta_1}\ln^2(2+t)\sum_{|b|\le10}\sup_{(s,y)\in[0,t]\times\overline{\R^2\setminus\cK_2}}\w{y}^{1/2}\cW_{1,1}(s,y)|Z^bG^\alpha(s,y)|\\
&\qquad+\ln(2+t)\sum_{|b|\le9}\sup_{(s,y)\in[0,t]\times\overline{\R^2\setminus\cK_2}}\w{y}^{1/2}\cW_{3/2+\delta_1,1}(s,y)|Z^bG^r(s,y)|,
\end{split}
\end{equation}
where we have used \eqref{initial:data-thm1} and the fact of $\supp_x\Box(\tilde Z^a-Z^a)u\subset\{x:|x|\le1\}$,
meanwhile $G^{\alpha}, G^r$ and $G_0^r$ are respectively given by
\begin{equation}\label{BA:impr3}
G^{\alpha}=\sum_{\tilde a+\tilde b\le a}|Z^{\tilde a}\p uZ^{\tilde b}\p u|,\quad
G^r=\sum_{\tilde a+\tilde b+\tilde c\le a}|Z^{\tilde a}\p^2uZ^{\tilde b}\p uZ^{\tilde c}\p u|,\quad G_0^r=|\Box(\tilde Z^a-Z^a)u|+G^r.
\end{equation}
In addition, it follows from \eqref{thm1-loc:energyC} that
\begin{equation}\label{BA:impr4}
\sum_{|b|\le8}\sup_{s\in[0,t]}\w{s}\|\p^b(G_0^r,\p^{\le1}G^\alpha)(s)\|_{L^2(\cK_3)}\ls\ve+\ve_1^2.
\end{equation}
On the other hand, \eqref{thm1-BA} and \eqref{thm1-loc:energyC1} imply that
\begin{equation}\label{BA:impr5}
\sum_{|b|\le10}\sup_{(s,y)\in[0,t]\times\overline{\R^2\setminus\cK_2}}\w{y}^{1/2}\cW_{1,1}(s,y)|Z^bG^\alpha(s,y)|
\ls\ve_1(\ve+\ve_1^2)\w{t}^{\delta_1+1/2}.
\end{equation}
The estimate of the last term in \eqref{BA:impr2} is much easier than the ones of \eqref{BA:impr4} and \eqref{BA:impr5}.
Collecting \eqref{BA:impr2}-\eqref{BA:impr5} leads to
\begin{equation*}
\w{x}^{1/2}\w{t-|x|}|\p\tilde Z^au|\ls(\ve+\ve_1^2)\ln(2+t).
\end{equation*}
This, together with \eqref{thm1-loc:energyC} and the Sobolev embedding, derives \eqref{BA:impr}.
\end{proof}

\begin{proof}[Proof of Theorem \ref{thm1}]
First of all, \eqref{BA:impr} yields that there is $C_2\ge1$ such that  for $t\in[0,T_\ve]$,
\begin{equation*}
\sum_{|a|\le2N-19}|\p Z^au|\le C_2(\ve+\ve_1^2)\w{x}^{-1/2}\w{t-|x|}^{-1}\w{t}^{\delta_1}.
\end{equation*}
Choosing $\ve_1=C_0\ve=4C_2\ve$ and $\ve_0=\frac{1}{(4C_2)^{4/\delta^2}}\le\frac{1}{16C_2^2}$.
For $t\le T_\ve$, one has
\begin{equation*}
t^{4\delta_1+1/2}\ve_1\le4C_2\ve_0^{(\delta/4+1/2)(\delta-2)+1}=4C_2\ve_0^{\delta^2/4}\le1.
\end{equation*}
Then for $N\ge19$, \eqref{thm1-time:BA} holds and \eqref{thm1-BA} can be improved for $t\in[0,T_\ve]$,
\begin{equation*}
\sum_{|a|\le N}|\p Z^au|\le\frac12\ve_1\w{x}^{-1/2}\w{t-|x|}^{-1}\w{t}^{\delta_1}.
\end{equation*}
This, together with the local existence of classical solution to the initial boundary value problem
of the hyperbolic equation under the admissible condition \eqref{YHCC-2} (for example, see
Section 6 of \cite{Li-Yin18}), yields that problem \eqref{QWE} admits a unique solution
$u\in\bigcap\limits_{j=0}^{2N+1}C^{j}([0,\ve^{\delta-2}], H^{2N+1-j}(\cK))$.
\end{proof}

\section{Energy estimates and local energy decay estimates in Theorem \ref{thm2}}\label{sect5}

We make the following bootstrap assumptions for all $t\in [0, \infty)$
\begin{align}
&\sum_{|a|\le N+1}|\p Z^au|\le\ve_1\w{x}^{-1/2}\w{t-|x|}^{-1}\w{t}^{\ve_2},\label{BA1}\\
&\sum_{|a|\le N+1}|\bar\p Z^au|\le\ve_1\w{x}^{-1/2}\w{t+|x|}^{\ve_2-1},\label{BA2}\\
&\sum_{|a|\le N+1}|Z^au|\le\ve_1\w{t+|x|}^{-1/2}\w{t-|x|}^{\ve_2-1/2},\label{BA3}\\
&\sum_{|a|\le N+2}\|\p^a\p_tu\|_{L^2(\cK_R)}\le\ve_1\w{t}^{-1},\label{BA4}\\
&|\p\p_tu|\le\ve_1\w{x}^{-1/2}\w{t-|x|}^{-1},\label{BA5}
\end{align}
where $u$ is the smooth solution of problem \eqref{QWE} with \eqref{nonlinear-Chaplygin},
$R\ge3$, $\ve_1\in(\ve,1)$ will be determined later and $\ve_2=10^{-3}$.

Note that due to $\supp(u_0,u_1)\subset\{x: |x|\le M_0\}$,
then the solution $u$ is supported in $\{x\in\cK:|x|\le t+M_0\}$.

\subsection{Energy estimates on the spacetime  derivatives of solutions}\label{Sonn-2}
At first, we establish the precise energy inequality for $\ds\sum_{i\le2N}\|\p\p_t^iu(t)\|^2_{L^2(\cK)}$
with $u$ being the solution of \eqref{QWE} with \eqref{nonlinear-Chaplygin}.
In this procedure, it is required to introduce some good unknowns so that the remaining
terms of the resulting nonlinearities admit better spacetime or time decay rates.

\begin{lemma}\label{lem:energy:time}
{\bf (Precise energy inequality for $\ds\sum_{i\le2N}\|\p\p_t^iu(t)\|^2_{L^2(\cK)}$)}
Under the assumptions of Theorem \ref{thm2}, let $u$ be the solution of \eqref{QWE} with \eqref{nonlinear-Chaplygin}
and suppose that \eqref{BA1}-\eqref{BA5} hold.
Then we have
\begin{equation}\label{energy:time}
\begin{split}
\sum_{i\le2N}\|\p\p_t^iu(t)\|^2_{L^2(\cK)}
\ls\ve^2+\ve_1^3+\ve_1\sum_{j\le2N-1}\|\p^{\le1}\p\p_t^ju\|^2_{L^2(\cK)}
+\ve_1\sum_{j\le2N-2}\int_0^t\frac{\|\p^{\le2}\p\p_t^ju(s)\|^2_{L^2(\cK)}}{1+s}ds.
\end{split}
\end{equation}
\end{lemma}
\begin{proof}
Denote the perturbed wave operator by
\begin{equation}\label{energy:t1}
\begin{split}
\Box_gf&:=\Box f-2\sum_{\alpha=0}^2C^{\alpha}Q_0(\p_{\alpha}f,u)-\sum_{\alpha,\beta,\mu,\nu=0}^2Q^{\alpha\beta\mu\nu}\p_{\mu}u\p_{\nu}u\p^2_{\alpha\beta}f
=(\p_t^2-\Delta_g)f,\\
\Delta_g&:=\Delta+\sum_{\alpha,\beta,\mu=0}^2Q^{\alpha\beta\mu}\p_{\mu}u\p^2_{\alpha\beta}
+\sum_{\alpha,\beta,\mu,\nu=0}^2Q^{\alpha\beta\mu\nu}\p_{\mu}u\p_{\nu}u\p^2_{\alpha\beta}.
\end{split}
\end{equation}
Then for smooth functions $f$ and $h$, one can check that
\begin{equation}\label{energy:t2}
\begin{split}
\Box_g(fh)&=2Q_0(f,h)+h\Box_gf+f\Box_gh-2\sum_{\alpha=0}^2C^{\alpha}[\p_{\alpha}fQ_0(h,u)+\p_{\alpha}hQ_0(f,u)]\\
&\quad-2\sum_{\alpha,\beta,\mu,\nu=0}^2Q^{\alpha\beta\mu\nu}\p_{\mu}u\p_{\nu}u\p_{\alpha}f\p_{\beta}h.
\end{split}
\end{equation}
Applying $\p_t^i$ with $i=1,2\cdots,2N$ to \eqref{QWE} with \eqref{nonlinear-Chaplygin} yields
\begin{equation}\label{energy:t3}
\begin{split}
\Box_g\p_t^iu&=2\sum_{\alpha=0}^2C^{\alpha}Q_0(\p_{\alpha}u,\p_t^iu)
+2\sum_{\alpha=0}^2\sum_{\substack{j+k=i,\\1\le j,k\le i-1}}C^{\alpha}C^i_{jk}Q_0(\p_{\alpha}\p_t^ju,\p_t^ku)\\
&\quad+\sum_{\alpha,\beta,\mu,\nu=0}^2\sum_{\substack{j+k+l=i,\\ j<i}}Q^{\alpha\beta\mu\nu}C^i_{jkl}\p^2_{\alpha\beta}\p_t^ju\p_{\mu}\p_t^ku\p_{\nu}\p_t^lu,
\end{split}
\end{equation}
where $C^i_{jk}=\frac{i!}{j!k!}$ and $C^i_{jkl}=\frac{i!}{j!k!l!}$.
Introduce the good unknowns for $i=1,\cdots,2N$ as follows
\begin{equation}\label{energy:t4}
V_i=\p_t^iu-\frac12\sum_{\alpha=0}^2C^{\alpha}\p_{\alpha}u\p_t^iu
-\sum_{\alpha=0}^2\sum_{\substack{j+k=i,\\1\le j,k\le i-1}}C^{\alpha}C^i_{jk}\chi_{[1,2]}(x)\p_{\alpha}\p_t^ju\p_t^ku.
\end{equation}
By \eqref{energy:t2} and \eqref{energy:t3}, we have that
\begin{equation*}\label{energy:t5}
\begin{split}
&\Box_gV_i-\sum_{\alpha=0}^2C^{\alpha}Q_0(\p_{\alpha}u,V_i)
=\sum_{\substack{j+k+l=i,\\ j<i}}\sum_{\alpha,\beta,\mu,\nu=0}^2C^i_{jkl}Q^{\alpha\beta\mu\nu}\p^2_{\alpha\beta}\p_t^ju\p_{\mu}\p_t^ku\p_{\nu}\p_t^lu
+N^V_1+N^V_2,\\
&N^V_1:=\sum_{\alpha=0}^2\sum_{\substack{j+k=i,\\1\le j,k\le i-1}}C^{\alpha}C^i_{jk}\Big\{[\Delta_g,\chi_{[1,2]}(x)](\p_{\alpha}\p_t^ju\p_t^ku)
+2(1-\chi_{[1,2]}(x))Q_0(\p_{\alpha}\p_t^ju,\p_t^ku)\Big\}\\
&\qquad -\chi_{[1,2]}(x)\sum_{\substack{j+k=i,\\1\le j,k\le i-1}}C^{\alpha}C^i_{jk}\Big\{\p_{\alpha}\p_t^ju\Box_g\p_t^ku
-2\sum_{\tilde\alpha,\beta,\mu,\nu=0}^2Q^{\tilde\alpha\beta\mu\nu}\p_{\mu}u\p_{\nu}u\p^2_{\alpha\beta}\p_t^ju\p_{\tilde\alpha}\p_t^ku\\
&\qquad+\p_t^ku\Box_g\p_{\alpha}\p_t^ju-2\sum_{\beta=0}^2C^{\beta}(\p_{\alpha\beta}^2\p_t^juQ_0(\p_t^ku,u)
+\p_{\beta}\p_t^kuQ_0(\p_{\alpha}\p_t^ju,u))\Big\},\\
\end{split}
\end{equation*}
\begin{equation}\label{energy:t5}
\begin{split}
&N^V_2:=-\frac12\sum_{\alpha=0}^2C^{\alpha}\Big\{\p_{\alpha}u\Box_g\p_t^iu+\p_t^iu\Box_g\p_{\alpha}u
-2\sum_{\beta=0}^2C^{\beta}(\p^2_{\alpha\beta}uQ_0(\p_t^iu,u)+\p_{\beta}\p_t^iuQ_0(\p_{\alpha}u,u))\\
&\qquad-2\sum_{\tilde\alpha,\beta,\mu,\nu=0}^2Q^{\tilde\alpha\beta\mu\nu}\p_{\mu}u\p_{\nu}u\p^2_{\alpha\beta}u\p_{\tilde\alpha}\p_t^iu\Big\}\\
&\qquad+\sum_{\alpha,\beta=0}^2C^{\alpha}C^{\beta}Q_0\Big(\p_{\alpha}u,\frac12\p_{\beta}u\p_t^iu
+\sum_{\substack{j+k=i,\\1\le j,k\le i-1}}C^i_{jk}\chi_{[1,2]}(x)\p_{\beta}\p_t^ju\p_t^ku\Big).
\end{split}
\end{equation}
Multiplying \eqref{energy:t5} by $e^q\p_tV_i$ with the ghost weight $q=q(|x|-t)=\arctan(|x|-t)$ introduced in \cite{Alinhac01a},
then from \eqref{energy:t1} and the fact of $\ds Q_0(\p_{\alpha}u,V_i)=\sum_{\mu,\nu=0}^2\eta^{\mu\nu}\p^2_{\alpha\nu}u\p_{\mu}V_i$
with the Minkowski metric $(\eta^{\mu\nu})_{\mu,\nu}=diag\{1,-1,-1\}$,
it is easy to find that the left hand side of \eqref{energy:t5} can be computed as
\begin{equation}\label{energy:t6}
\begin{split}
&e^q\p_tV_i\Big(\Box_gV_i-\sum_{\alpha=0}^2C^{\alpha}Q_0(\p_{\alpha}u,V_i)\Big)\\
&=e^q\p_tV_i\Box V_i-e^q\p_tV_i\p^2_{\alpha\mu}V_i\p_{\nu}u(C^{\alpha}\eta^{\mu\nu}+C^{\mu}\eta^{\alpha\nu})
-C^{\alpha}\eta^{\mu\nu}e^q\p_tV_i\p^2_{\alpha\nu}u\p_{\mu}V_i\\
&\quad-Q^{\alpha\beta\mu\nu}e^q\p_tV_i\p^2_{\alpha\beta}V_i\p_{\mu}u\p_{\nu}u\\
&=\frac12\p_t[e^q(|\p_tV_i|^2+|\nabla V_i|^2)]-\dive(e^q\p_tV_i\nabla V_i)+\frac{e^q|\bar\p V_i|^2}{2\w{t-|x|}^2}\\
&\quad-\p_{\alpha}(Q^{\alpha\mu\nu}e^q\p_tV_i\p_{\mu}V_i\p_{\nu}u)
+e^q\p_tV_i\p_{\mu}V_i(Q^{\alpha\mu\nu}\p_{\alpha}q\p_{\nu}u+C^{\mu}\Box u)\\
&\quad+\frac12\p_t(Q^{\alpha\mu\nu}e^q\p_{\alpha}V_i\p_{\mu}V_i\p_{\nu}u)
-\frac12Q^{\alpha\mu\nu}e^q\p_{\alpha}V_i\p_{\mu}V_i(\p_tq\p_{\nu}u+\p^2_{0\nu}u)\\
&\quad-\p_{\alpha}(Q^{\alpha\beta\mu\nu}e^q\p_tV_i\p_{\beta}V_i\p_{\mu}u\p_{\nu}u)
+Q^{\alpha\beta\mu\nu}\p_tV_i\p_{\beta}V_i\p_{\alpha}(e^q\p_{\mu}u\p_{\nu}u)\\
&\quad+\frac12\p_t(Q^{\alpha\beta\mu\nu}e^q\p_{\alpha}V_i\p_{\beta}V_i\p_{\mu}u\p_{\nu}u)
-\frac12Q^{\alpha\beta\mu\nu}\p_{\alpha}V_i\p_{\beta}V_i\p_t(e^q\p_{\mu}u\p_{\nu}u),
\end{split}
\end{equation}
where the summations $\ds\sum_{\alpha,\mu,\nu=0}^2$ and $\ds\sum_{\alpha,\beta,\mu,\nu=0}^2$ in \eqref{energy:t6}
are omitted.

For the term $Q^{\alpha\beta\mu\nu}\p^2_{\alpha\beta}\p_t^ju\p_{\mu}\p_t^ku\p_{\nu}\p_t^lu$ on the right hand side of the first line in \eqref{energy:t5},
\eqref{BA3} leads to
\begin{equation}\label{energy:t7}
\sum_{1\le i\le2N}\sum_{\substack{j+k+l\le i,\\j<i}}\|\p^2_{\alpha\beta}\p_t^ju\p_{\mu}\p_t^ku\p_{\nu}\p_t^lu\|_{L^2(\cK)}
\ls\ve_1^2\w{t}^{-1}\sum_{j\le2N-1}\|\p^{\le1}\p\p_t^ju\|_{L^2(\cK)}.
\end{equation}

Note that the terms in the second line of \eqref{energy:t5} are supported in $\{x\in\cK:|x|\le2\}$ and
can be controlled by \eqref{BA4} with the Sobolev embedding
\begin{equation*}\label{energy:t8}
\begin{split}
&\|[\Delta_g,\chi_{[1,2]}(x)](\p_{\alpha}\p_t^ju\p_t^ku)\|_{L^2(\cK)}+\|(1-\chi_{[1,2]}(x))Q_0(\p_{\alpha}\p_t^ju,\p_t^ku)\|_{L^2(\cK)}\\
&\ls\sum_{|a|\le N}\|\p^a\p_tu(t)\|_{L^\infty(\cK_2)}\sum_{j\le2N-1}\|\p^{\le1}\p\p_t^ju(t)\|_{L^2(\cK)}\\
\end{split}
\end{equation*}
\begin{equation}\label{energy:t8}
\begin{split}
&\ls\sum_{|a|\le N+2}\|\p^a\p_tu(t)\|_{L^2(\cK_3)}\sum_{j\le2N-1}\|\p^{\le1}\p\p_t^ju(t)\|_{L^2(\cK)}\\
&\ls\ve_1\w{t}^{-1}\sum_{j\le2N-1}\|\p^{\le1}\p\p_t^ju(t)\|_{L^2(\cK)}.
\end{split}
\end{equation}
The terms $Q^{\alpha\mu\nu}e^q\p_tV_i\p_{\mu}V_i\p_{\alpha}q\p_{\nu}u$ and
$Q^{\alpha\mu\nu}e^q\p_{\alpha}V_i\p_{\mu}V_i(\p_tq\p_{\nu}u+\p^2_{0\nu}u)$ in the fifth and
sixth lines of \eqref{energy:t6} can be treated as
\begin{equation}\label{energy:t9}
\begin{split}
\Big|\sum_{\alpha,\beta,\mu=0}^2Q^{\alpha\mu\nu}e^q\p_tV_i\p_{\mu}V_i\p_{\alpha}q\p_{\nu}u\Big|
&\ls\w{s-|x|}^{-2}(|\p V_i||\bar\p V_i||\p u|+|\p V_i|^2|\bar\p u|),\\
\Big|\sum_{\alpha,\beta,\mu=0}^2Q^{\alpha\mu\nu}e^q\p_{\alpha}V_i\p_{\mu}V_i\p^2_{0\nu}u\Big|
&\ls|\p V_i||\bar\p V_i||\p\p_tu|+|\p V_i|^2|\bar\p\p u|,
\end{split}
\end{equation}
where we have used \eqref{null:structure} and the fact of $\bar\p q=0$.
The estimates on the other terms in \eqref{energy:t5} and \eqref{energy:t6} are analogous to \eqref{energy:t7}-\eqref{energy:t9} by \eqref{null:structure} and \eqref{BA1}-\eqref{BA5}.
On the other hand, according to \eqref{energy:t4} with \eqref{BA3}, we can easily get
\begin{equation}\label{energy:t10}
\sum_{1\le i\le2N}\|\p(\p_t^iu-V_i)\|_{L^2(\cK)}\ls\ve_1\sum_{j\le2N-1}\|\p^{\le1}\p\p_t^ju\|_{L^2(\cK)}.
\end{equation}
Integrating \eqref{energy:t6} over $[0,t]\times\cK$ with $i=1,\cdots,2N$, \eqref{energy:t7}-\eqref{energy:t10} and the smallness of $\ve_1$ yields
\begin{equation}\label{energy:t11}
\begin{split}
&\sum_{1\le i\le2N}\Big(\|\p\p_t^iu(t)\|^2_{L^2(\cK)}+\int_0^t\int_{\cK}\frac{|\bar\p V_i|^2}{\w{s-|x|}^2}dxds\Big)\\
&\ls\ve^2+\ve_1\sum_{j\le2N-1}\|\p^{\le1}\p\p_t^ju(t)\|^2_{L^2(\cK)}
+\ve_1\sum_{j\le2N-1}\int_0^t\frac{\|\p^{\le1}\p\p_t^ju(s)\|^2_{L^2(\cK)}}{1+s}ds\\
&\quad+\sum_{1\le i\le2N}(|\cB_1^i|+|\cB_2^i|+|\cB_3^i|),
\end{split}
\end{equation}
where the boundary terms are given by
\begin{equation}\label{energy:t12}
\begin{split}
\cB_1^i&:=\int_0^t\int_{\p\cK}e^q(n_1(x),n_2(x))\cdot\nabla V_i\p_tV_id\sigma ds,\\
\cB_2^i&:=\int_0^t\int_{\p\cK}\sum_{j=1}^2\sum_{\beta,\mu,\nu=0}^2
Q^{j\beta\mu\nu}n_j(x)e^q\p_tV_i\p_{\beta}V_i\p_{\mu}u\p_{\nu}ud\sigma ds,\\
\cB_3^i&:=\int_0^t\int_{\p\cK}\sum_{j=1}^2\sum_{\mu,\nu=0}^2
Q^{j\mu\nu}n_j(x)e^q\p_tV_i\p_{\mu}V_i\p_{\nu}ud\sigma ds.
\end{split}
\end{equation}
According to \eqref{energy:t4}, one has that  $\ds V_i=\p_t^iu(1-\frac12\sum_{\alpha=0}^2C^{\alpha}\p_{\alpha}u)$
holds on the boundary $\p\cK$ and $\cB_1^i$ can be computed with the boundary condition $\frac{\p}{\p\bn}\p_t^{\ell}u|_{[0,\infty)\times\p\cK}=0$
as
\begin{equation}\label{energy:t13}
\begin{split}
&\cB_1^i=\int_0^t\int_{\p\cK}e^q\p_tV_i\Big\{\frac{\p}{\p\bn}\p_t^iu(1-\frac12\sum_{\alpha=0}^2C^{\alpha}\p_{\alpha}u)
-\frac12\p_t^iu\sum_{\alpha=0}^2C^{\alpha}\frac{\p}{\p\bn}\p_{\alpha}u\Big\}d\sigma ds\\
&=-\frac12\sum_{\alpha=0}^2C^{\alpha}\int_0^t\int_{\p\cK}e^q\p_t^iu\frac{\p}{\p\bn}\p_{\alpha}u
\Big\{\p_t\p_t^iu(1-\frac12\sum_{\beta=0}^2C^{\beta}\p_{\beta}u)-\frac12\p_t^iu\sum_{\beta=0}^2C^{\beta}\p^2_{0\beta}u\Big\}d\sigma ds\\
&=-\frac14\sum_{\alpha=0}^2C^{\alpha}\int_0^t\int_{\p\cK}\Big\{\p_t[e^q|\p_t^iu|^2\frac{\p}{\p\bn}\p_{\alpha}u(1-\frac12\sum_{\beta=0}^2C^{\beta}\p_{\beta}u)]\\
&\quad-|\p_t^iu|^2\p_t[e^q\frac{\p}{\p\bn}\p_{\alpha}u(1-\frac12\sum_{\beta=0}^2C^{\beta}\p_{\beta}u)]
-e^q|\p_t^iu|^2\frac{\p}{\p\bn}\p_{\alpha}u\sum_{\beta=0}^2C^{\beta}\p^2_{0\beta}u\Big\}d\sigma ds.
\end{split}
\end{equation}
Combining \eqref{energy:t13} with the trace Theorem, the Sobolev embedding, \eqref{initial:data}, \eqref{BA1}  and \eqref{BA4} yields
\begin{equation}\label{energy:t14}
\sum_{1\le i\le2N}|\cB_1^i|\ls\ve^2+\ve_1\sum_{j\le2N}\|\p\p_t^ju(t)\|^2_{L^2(\cK)}+\ve_1\sum_{j\le2N}\int_0^t\frac{\|\p\p_t^ju(s)\|^2_{L^2(\cK)}}{1+s}ds.
\end{equation}
Analogously, $\cB_2^i$ can be treated by the admissible condition \eqref{YHCC-2} as follows
\begin{equation}\label{energy:t15}
\begin{split}
&\cB_2^i=\int_0^t\int_{\p\cK}Q^{j\beta\mu\nu}n_j(x)e^q\p_tV_i\p_{\mu}u\p_{\nu}u
\Big\{\p_{\beta}\p_t^iu(1-\frac12C^{\alpha}\p_{\alpha}u)-\frac12\p_t^iuC^{\alpha}\p^2_{\alpha\beta}u\Big\}d\sigma ds\\
&=-\frac12C^{\alpha}\int_0^t\int_{\p\cK}Q^{j\beta\mu\nu}n_j(x)e^q\p_t^iu\p^2_{\alpha\beta}u\p_{\mu}u\p_{\nu}u
\Big\{\p_t\p_t^iu(1-\frac12C^{\gamma}\p_{\gamma}u)-\frac12\p_t^iuC^{\gamma}\p^2_{0\gamma}u\Big\}d\sigma ds\\
&=-\frac14C^{\alpha}\int_0^t\int_{\p\cK}\Big\{\p_t[Q^{j\beta\mu\nu}n_j(x)e^q|\p_t^iu|^2
\p^2_{\alpha\beta}u\p_{\mu}u\p_{\nu}u(1-\frac12C^{\gamma}\p_{\gamma}u)]\\
&\quad-Q^{j\beta\mu\nu}n_j(x)|\p_t^iu|^2\p_t[e^q\p^2_{\alpha\beta}u\p_{\mu}u\p_{\nu}u(1-\frac12C^{\gamma}\p_{\gamma}u)]\\
&\quad-Q^{j\beta\mu\nu}C^{\gamma}n_j(x)e^q|\p_t^iu|^2\p^2_{\alpha\beta}u\p_{\mu}u\p_{\nu}u\p^2_{0\gamma}u\Big\}d\sigma ds,
\end{split}
\end{equation}
where the summations $\ds\sum_{j=1}^2$, $\ds\sum_{\beta,\mu,\nu=0}^2$, $\ds\sum_{\alpha=0}^2$ and $\ds\sum_{\gamma=0}^2$
in \eqref{energy:t15} are omitted.
The estimate of $\cB_3^i$ in \eqref{energy:t12} can be similarly achieved.
Collecting \eqref{energy:t11}-\eqref{energy:t15} leads to
\begin{equation}\label{energy:t16}
\begin{split}
\sum_{1\le i\le2N}\|\p\p_t^iu(t)\|^2_{L^2(\cK)}
\ls\ve^2+\ve_1\sum_{j\le2N-1}\|\p^{\le1}\p\p_t^ju(t)\|^2_{L^2(\cK)}
+\ve_1\sum_{j\le2N-1}\int_0^t\frac{\|\p^{\le1}\p\p_t^ju(s)\|^2_{L^2(\cK)}}{1+s}ds.
\end{split}
\end{equation}
Finally, we turn to the zero-th order energy estimate.
Instead of \eqref{energy:t4}, define the good unknown
\begin{equation*}
\ds V_0:=u-\sum_{\alpha=0}^2C^{\alpha}u\p_{\alpha}u.
\end{equation*}
Then it is easy to find that $V_0$ satisfies
\begin{equation*}
\Box V_0=Q^{\alpha\beta\mu\nu}\p^2_{\alpha\beta}u\p_{\mu}u\p_{\nu}u-\sum_{\alpha=0}^2C^{\alpha}(u\Box\p_{\alpha}u+\p_{\alpha}u\Box u).
\end{equation*}
It follows from \eqref{QWE}, \eqref{nonlinear-Chaplygin}, \eqref{null:condition}, \eqref{null:structure} and \eqref{BA1}-\eqref{BA3} that
\begin{equation}\label{energy:t17}
|\Box V_0|\ls\ve_1^2\w{t}^{-1.4}|\p^{\le2}\p u|.
\end{equation}
On the other hand, one has $\|\p(u-V_0)\|_{L^2(\cK)}\ls\ve_1\|\p^{\le1}\p u\|_{L^2(\cK)}$ and
\begin{equation}\label{energy:t18}
\p_tV_0\Box V_0=\frac12\p_t(|\p_tV_0|^2+|\nabla V_0|^2)-\dive(\nabla V_0\p_tV_0).
\end{equation}
Integrating \eqref{energy:t18} over $[0,t]\times\cK$ with \eqref{initial:data} and \eqref{energy:t17} derives
\begin{equation}\label{energy:t19}
\|\p u(t)\|^2_{L^2(\cK)}\ls\ve^2+\ve_1\|\p^{\le1}\p u(t)\|^2_{L^2(\cK)}+\ve_1\int_0^t\|\p^{\le2}\p u(s)\|^2_{L^2(\cK)}\frac{ds}{\w{s}^{1.4}}+\cB_0,
\end{equation}
where the boundary term $\cB_0$ can be controlled by use of the boundary condition $\frac{\p}{\p\bn}u|_{[0,\infty)\times\p\cK}=0$, the trace
theorem and \eqref{BA3} as follows
\begin{equation}\label{energy:t20}
\begin{split}
\cB_0&=\Big|\int_0^t\int_{\p\cK}(n_1(x),n_2(x))\cdot\nabla V_0\p_tV_0d\sigma ds\Big|\\
&=\Big|\int_0^t\int_{\p\cK}\p_tV_0(\frac{\p}{\p\bn}u-\sum_{\alpha=0}^2C^{\alpha}\frac{\p}{\p\bn}(u\p_{\alpha}u))d\sigma ds\Big|\\
&\ls\int_0^t\|(1-\chi_{[1,2]}(x))\p_tV_0\frac{\p}{\p\bn}(u\p_{\alpha}u)\|_{H^1(\cK)}ds\ls\ve_1^3\int_0^t\w{s}^{-2}ds\ls\ve_1^3.
\end{split}
\end{equation}
Collecting \eqref{energy:t16}, \eqref{energy:t19} and \eqref{energy:t20} completes the proof of \eqref{energy:time}.
\end{proof}

Based on Lemma \ref{lem:energy:time}, we derive the slow time growth of the energy $\ds\sum_{|a|\le2N}\|\p\p^au\|_{L^2(\cK)}$
with $u$ being the solution of \eqref{QWE} with \eqref{nonlinear-Chaplygin}.

\begin{lemma}\label{lem:energy:dtdx}
{\bf (Estimate for the energy $\ds\sum_{|a|\le2N}\|\p\p^au\|_{L^2(\cK)}$)}
Under the assumptions of Theorem \ref{thm2}, let $u$ be the solution of \eqref{QWE}  with \eqref{nonlinear-Chaplygin}
and suppose that \eqref{BA1}-\eqref{BA5} hold.
Then we have
\begin{equation}\label{energy:dtdx}
\sum_{|a|\le2N}\|\p\p^au\|_{L^2(\cK)}\ls\ve_1\w{t}^{\ve_2},
\end{equation}
where $\ve_2=10^{-3}$.
\end{lemma}
\begin{proof}
Set $\ds E_j(t):=\sum_{k=0}^{2N-j}\|\p\p_t^ku(t)\|_{H^j(\cK)}$ with $0\le j\le2N$.
Then one can find that for $j\ge1$,
\begin{equation}\label{energy:dtdx1}
\begin{split}
E_j(t)&\ls\sum_{k=0}^{2N-j}[\|\p\p_t^ku(t)\|_{L^2(\cK)}
+\sum_{1\le|a|\le j}(\|\p_t\p_t^k\p_x^au(t)\|_{L^2(\cK)}+\|\p_x\p_t^k\p_x^au(t)\|_{L^2(\cK)})]\\
&\ls E_0(t)+E_{j-1}(t)+\sum_{k=0}^{2N-j}\sum_{2\le|a|\le j+1}\|\p_t^k\p_x^au(t)\|_{L^2(\cK)},
\end{split}
\end{equation}
where we have used the fact of
\begin{equation*}
\sum_{k=0}^{2N-j}\sum_{1\le|a|\le j}\|\p_t\p_t^k\p_x^au(t)\|_{L^2(\cK)}
\ls\sum_{k=0}^{2N-j}\sum_{|a'|\le j-1}\|\p_x\p_t^{k+1}\p_x^{a'}u(t)\|_{L^2(\cK)}
\ls E_{j-1}(t).
\end{equation*}
For the last term in the second line of \eqref{energy:dtdx1}, it can be deduced from the elliptic estimate \eqref{ellip} that
\begin{equation}\label{energy:dtdx2}
\begin{split}
\|\p_t^k\p_x^au\|_{L^2(\cK)}&\ls\|\Delta_x\p_t^ku\|_{H^{|a|-2}(\cK)}+\|\p_t^ku\|_{H^{|a|-1}(\cK_{R+1})}\\
&\ls\|\p_t^{k+2}u\|_{H^{|a|-2}(\cK)}+\|\p_t^kQ(\p u,\p^2u)\|_{H^{|a|-2}(\cK)}\\
&\quad+\|\p_t^ku\|_{L^2(\cK_{R+1})}+\sum_{1\le|b|\le|a|-1}\|\p_t^k\p_x^bu\|_{L^2(\cK_{R+1})},
\end{split}
\end{equation}
where $\Delta_x=\p_t^2-\Box$ and the equation in \eqref{QWE} has been used.
In addition, by \eqref{BA1} with $|a|+k\le2N+1$, we have
\begin{equation}\label{energy:dtdx3}
\begin{split}
\|\p_t^kQ(\p u,\p^2u)\|_{H^{|a|-2}(\cK)}
&\ls\sum_{l\le k}\ve_1(\|\p_t^l\p u\|_{H^{|a|-1}(\cK)}+\|\p_t^{l+2}u\|_{H^{|a|-2}(\cK)})\\
&\ls\ve_1[E_j(t)+E_{j-1}(t)].
\end{split}
\end{equation}
Combining \eqref{energy:dtdx1}-\eqref{energy:dtdx3} shows that for $j\ge1$,
\begin{equation}\label{energy:dtdx4}
E_j(t)\ls E_0(t)+E_{j-1}(t)+\ve_1E_j(t)+\|u\|_{L^2(\cK_{R+1})}.
\end{equation}
Then collecting \eqref{BA3}, \eqref{energy:time}, \eqref{energy:dtdx4} with $j=1,2$ and the smallness of $\ve_1$ yields
\begin{equation}\label{YHCC-20-5}
\begin{split}
[E_0(t)]^2&\ls\ve^2+\ve_1^3+\ve_1\int_0^t\big\{[E_0(s)]^2+\ve_1^2\w{s}^{-1}\big\}\frac{ds}{\w{s}}\\
&\ls\ve^2+\ve_1^3+\ve_1\int_0^t[E_0(s)]^2\frac{ds}{\w{s}}\\
&\le\tilde C_0\ve_1^2+\tilde C_0\ve_1\int_0^t[E_0(s)]^2\frac{ds}{\w{s}},
\end{split}
\end{equation}
where $\tilde C_0$ is a positive constant.
Thus, it follows from the Gronwall's lemma and \eqref{YHCC-20-5} that
\begin{equation*}
[E_0(t)]^2\ls\ve_1^2(1+t)^{\tilde C_0\ve_1}.
\end{equation*}
This, together with \eqref{energy:dtdx4} and the smallness of $\ve_1$, completes the proof of \eqref{energy:dtdx}.
\end{proof}

\subsection{Energy estimates on the  vector field  derivatives of solutions}\label{Sonn-3}
By the obtained results in Subsection \ref{Sonn-2}, we start to derive the energy estimates
for the vector field derivatives of solution $u$ to \eqref{QWE} with \eqref{nonlinear-Chaplygin}.
In this process, several good unknowns are also introduced.
\begin{lemma}
{\bf (Estimate for $\ds\sum_{|a|\le2N-1}\|\p Z^au\|_{L^2(\cK)}$)}
Under the assumptions of Theorem \ref{thm2}, let $u$ be the solution of \eqref{QWE}  with \eqref{nonlinear-Chaplygin}
and suppose that \eqref{BA1}-\eqref{BA5} hold.
Then one has
\begin{equation}\label{energyA}
\sum_{|a|\le2N-1}\|\p Z^au\|_{L^2(\cK)}\ls\ve_1(1+t)^{1/2+\ve_2}.
\end{equation}
\end{lemma}
\begin{proof}
From \eqref{eqn:high} and \eqref{energy:t1}, we can obtain that for any multi-index $a\in\N_0^4$ with $|a|\le2N-1$,
\begin{equation}\label{energyA1}
\begin{split}
\Box_gZ^au&=2\sum_{\alpha=0}^2C^{\alpha}Q_0(\p_{\alpha}u,Z^au)
+2\sum_{\alpha=0}^2\sum_{\substack{b+c\le a,\\b,c<a}}C^{\alpha}C^a_{bc}Q_0(\p_{\alpha}Z^bu,Z^cu)\\
&\quad+\sum_{\substack{b+c+d\le a,\\b<a}}Q_{abcd}^{\alpha\beta\mu\nu}\p^2_{\alpha\beta}Z^bu\p_{\mu}Z^cu\p_{\nu}Z^du.
\end{split}
\end{equation}
Define the following good unknowns for $|a|\le2N-1$,
\begin{equation}\label{energyA2}
V^a:=Z^au-\frac12\sum_{\alpha=0}^2C^{\alpha}\p_{\alpha}uZ^au
-\sum_{\alpha=0}^2\sum_{\substack{b+c\le a,\\b,c<a}}C^{\alpha}C^a_{bc}\p_{\alpha}Z^buZ^cu,
\end{equation}
which derives
\begin{equation}\label{energyA3}
\begin{split}
&\Box_gV^a-\sum_{\alpha=0}^2C^{\alpha}Q_0(\p_{\alpha}u,V^a)
=\sum_{\substack{b+c+d\le a,\\b<a}}Q_{abcd}^{\alpha\beta\mu\nu}\p^2_{\alpha\beta}Z^bu\p_{\mu}Z^cu\p_{\nu}Z^du+N^V_3+N^V_4,\\
&N^V_3:=-\sum_{\alpha=0}^2\sum_{\substack{b+c\le a,\\b,c<a}}C^{\alpha}C^a_{bc}\Big\{
-2\sum_{\beta=0}^2C^{\beta}(\p^2_{\alpha\beta}Z^buQ_0(Z^cu,u)+\p_{\beta}Z^cuQ_0(\p_{\alpha}Z^bu,u))\\
&\qquad +Z^cu\Box_g\p_{\alpha}Z^bu+\p_{\alpha}Z^bu\Box_g Z^cu
-2\sum_{\tilde\alpha,\beta,\mu,\nu=0}^2Q^{\tilde\alpha\beta\mu\nu}\p_{\mu}u\p_{\nu}u\p^2_{\alpha\beta}Z^bu\p_{\tilde\alpha}Z^cu\Big\},\\
&N^V_4:=-\frac12\sum_{\alpha=0}^2C^{\alpha}\Big\{\p_{\alpha}u\Box_gZ^au+Z^au\Box_g\p_{\alpha}u
-2\sum_{\tilde\alpha,\beta,\mu,\nu=0}^2Q^{\tilde\alpha\beta\mu\nu}\p_{\mu}u\p_{\nu}u\p^2_{\alpha\beta}u\p_{\tilde\alpha}Z^au\\
&\qquad-2\sum_{\beta=0}^2C^{\beta}(\p^2_{\alpha\beta}uQ_0(Z^au,u)+\p_{\beta}Z^auQ_0(\p_{\alpha}u,u))\Big\}\\
&\qquad +\sum_{\alpha,\beta=0}^2C^{\alpha}C^{\beta}Q_0\Big(\p_{\alpha}u,\frac12\p_{\beta}uZ^au
+\sum_{\substack{b+c\le a,\\b,c<a}}C^a_{bc}\p_{\beta}Z^buZ^cu\Big).
\end{split}
\end{equation}
Analogously to \eqref{energy:t6}, we have
\begin{equation}\label{energyA4}
\begin{split}
&e^q\p_tV^a(\Box_gV^a-\sum_{\alpha=0}^2C^{\alpha}Q_0(\p_{\alpha}u,V^a))\\
&=\frac12\p_t[e^q(|\p_tV^a|^2+|\nabla V^a|^2)]-\dive(e^q\p_tV^a\nabla V^a)+\frac{e^q|\bar\p V^a|^2}{2\w{t-|x|}^2}\\
&\quad-\p_{\alpha}(Q^{\alpha\mu\nu}e^q\p_tV^a\p_{\mu}V^a\p_{\nu}u)
+e^q\p_tV^a\p_{\mu}V^a(Q^{\alpha\mu\nu}\p_{\alpha}q\p_{\nu}u+C^{\mu}\Box u)\\
&\quad+\frac12\p_t(Q^{\alpha\mu\nu}e^q\p_{\alpha}V^a\p_{\mu}V^a\p_{\nu}u)
-\frac12Q^{\alpha\mu\nu}e^q\p_{\alpha}V^a\p_{\mu}V^a(\p_tq\p_{\nu}u+\p^2_{0\nu}u)\\
&\quad-\p_{\alpha}(Q^{\alpha\beta\mu\nu}e^q\p_tV^a\p_{\beta}V^a\p_{\mu}u\p_{\nu}u)+Q^{\alpha\beta\mu\nu}
\p_tV^a\p_{\beta}V^a\p_{\alpha}(e^q\p_{\mu}u\p_{\nu}u)\\
&\quad+\frac12\p_t(Q^{\alpha\beta\mu\nu}e^q\p_{\alpha}V^a\p_{\beta}V^a\p_{\mu}u\p_{\nu}u)
-\frac12Q^{\alpha\beta\mu\nu}\p_{\alpha}V^a\p_{\beta}V^a\p_t(e^q\p_{\mu}u\p_{\nu}u).
\end{split}
\end{equation}
Note that the terms $Q^{\alpha\mu\nu}e^q\p_tV^a\p_{\mu}V^a\p_{\alpha}q\p_{\nu}u$ and
$Q^{\alpha\mu\nu}e^q\p_{\alpha}V^a\p_{\mu}V^a(\p_tq\p_{\nu}u+\p^2_{0\nu}u)$ in the third and forth lines of \eqref{energyA4} can be controlled as in \eqref{energy:t9}.
The other terms in \eqref{energyA3} are at least cubic and can be treated as \eqref{energy:t7} with \eqref{BA3}.
Integrating \eqref{energyA4} over $[0,t]\times\cK$ with \eqref{initial:data} and the smallness of $\ve_1$ derives
\begin{equation}\label{energyA5}
\begin{split}
&\sum_{|a|\le2N-1}\Big\{\|\p V^a(t)\|^2_{L^2(\cK)}+\int_0^t\int_{\cK}\frac{|\bar\p V^a|^2}{\w{s-|y|}^2}dyds\Big\}\\
&\ls\ve^2+\sum_{|a|\le2N-1}(|\cB_1^a|+|\cB_2^a|)+\ve_1\sum_{|b|\le2N-1}\int_0^t\frac{\|\p V^b(s)\|^2_{L^2(\cK)}+\|\p Z^bu(s)\|^2_{L^2(\cK)}}{1+s}ds,\\
&\cB^a_1:=\int_0^t\int_{\p\cK}e^q(n_1(x),n_2(x))\cdot\nabla V^a(s,x)\p_tV^a(s,x)d\sigma ds,\\
&\cB^a_2:=\int_0^t\int_{\p\cK}\sum_{j=1}^2\sum_{\beta,\mu,\nu=0}^2
Q^{j\beta\mu\nu}n_j(x)e^q\p_tV^a\p_{\beta}V^a\p_{\mu}u\p_{\nu}ud\sigma ds.
\end{split}
\end{equation}
According to $\p\cK\subset\overline{\cK_1}$ and the trace theorem, we arrive at
\begin{equation}\label{energyA6}
\begin{split}
|\cB^a_1|&\ls\sum_{|b|\le|a|}\int_0^t\|(1-\chi_{[1,2]}(x))\p V^a\|^2_{H^1(\cK)}ds\\
&\ls\sum_{|b|\le|a|+1}\int_0^t\|\p \p^bu\|^2_{L^2(\cK_2)}ds\ls\ve_1^2(1+t)^{1+2\ve_2},
\end{split}
\end{equation}
where \eqref{BA3}, \eqref{energy:dtdx} and \eqref{energyA2} have been used.
Analogously, $\cB^a_2$ can be also treated.

On the other hand, similarly to \eqref{energy:t10}, one has
\begin{equation}\label{energyA7}
\sum_{|a|\le2N-1}\|\p(Z^au-V^a)\|_{L^2(\cK)}\ls\ve_1\sum_{|a|\le2N-1}\|\p Z^au\|_{L^2(\cK)}.
\end{equation}
Then, \eqref{energyA5}, \eqref{energyA6} and \eqref{energyA7} ensure that there is a constant $C_3>0$ such that
\begin{equation}\label{energyA8}
\sum_{|a|\le2N-1}\|\p Z^au\|^2_{L^2(\cK)}\le C_3\ve^2+C_3\ve_1\int_0^t\sum_{|a|\le2N-1}\frac{\|\p Z^au(s)\|^2_{L^2(\cK)}}{1+s}ds
+C_3\ve_1^2(1+t)^{1+2\ve_2}.
\end{equation}
Applying Lemma \ref{lem-Gronwall-2} to \eqref{energyA8} yields \eqref{energyA}.
\end{proof}

\subsection{Local energy decay estimates}

To improve the energy estimate in \eqref{energyA}, one only needs to give an improvement on the estimate
of the boundary term in \eqref{energyA6}, which can be accomplished by the local energy decay estimates.
Before taking the local energy decay estimates, we introduce another good unknown ($i=0,1,2\cdots2N-1$)
as follows
\begin{equation}\label{good1}
\tilde V_i:=\p_t^iu-\sum_{\alpha=0}^2\sum_{j+k=i}C^{\alpha}C^i_{jk}\chi_{[1,2]}(x)\p_{\alpha}\p_t^ju\p_t^ku.
\end{equation}
Then $\tilde V_i$ satisfies $\ds\frac{\p}{\p\bn}\tilde V_i|_{[0,\infty)\times\p\cK}=0$ and
\begin{equation}\label{eqn:good1}
\begin{split}
\Box\tilde V_i&=\sum_{j+k+l=i}C^i_{jkl}Q^{\alpha\beta\mu\nu}\p^2_{\alpha\beta}\p_t^ju\p_{\mu}\p_t^ku\p_{\nu}\p_t^lu\\
&\quad-\sum_{\alpha=0}^2\sum_{j+k=i}C^{\alpha}C^i_{jk}\chi_{[1,2]}(x)[\p_t^ku\Box\p_{\alpha}\p_t^ju+\p_{\alpha}\p_t^ju\Box\p_t^ku],\\
&\quad+\sum_{\alpha=0}^2\sum_{j+k=i}C^{\alpha}C^i_{jk}\Big\{[\Delta,\chi_{[1,2]}(x)](\p_{\alpha}\p_t^ju\p_t^ku)
+2(1-\chi_{[1,2]}(x))Q_0(\p_{\alpha}\p_t^ju,\p_t^ku)\Big\}.
\end{split}
\end{equation}
Obviously, \eqref{good1} and \eqref{eqn:good1} are different from the forms in \eqref{energy:t4} and \eqref{energy:t5}, respectively.

\begin{lemma}
{\bf (Uniform time decay of $\|\p^{\le1}u\|_{L^2(\cK_R)}$)}
Under the assumptions of Theorem \ref{thm2}, let $u$ be the solution of \eqref{QWE}  with \eqref{nonlinear-Chaplygin}
and suppose that \eqref{BA1}-\eqref{BA5} hold.
Then we have
\begin{equation}\label{loc:energyA}
\|\p^{\le1}u\|_{L^2(\cK_R)}\ls(\ve+\ve_1^2)\w{t}^{-1}\ln(2+t).
\end{equation}
\end{lemma}
\begin{proof}
Applying \eqref{loc:ibvp} to \eqref{eqn:good1} with $i=0$, $G^\alpha=0$, $\eta=\ve_2$ yields
\begin{equation}\label{loc:energyA1}
\begin{split}
\w{t}\|\p^{\le1}\tilde V_0\|_{L^2(\cK_R)}\ls\ve+\ve^2\ln(2+t)+\ln(2+t)\sup_{s\in[0,t]}\w{s}\|\Box\tilde V_0(s)\|_{L^2(\cK_3)}\\
+\ln(2+t)\sum_{|a|\le1}\sup_{(s,y)\in[0,t]\times\overline{\R^2\setminus\cK_2}}\w{y}^{1/2}\cW_{3/2+\ve_2,1}(s,y)|\p^a\Box\tilde V_0(s,y)|,
\end{split}
\end{equation}
where we have used \eqref{initial:data}.
On the other hand, it can be deduced from \eqref{QWE}, \eqref{nonlinear-Chaplygin}, \eqref{null:condition}, \eqref{null:structure}, \eqref{BA1}, \eqref{BA2}, \eqref{BA3} and \eqref{eqn:good1} with $N\ge40$ that for $|y|\ge2$,
\begin{equation}\label{loc:energyA2}
\sum_{|a|\le1}|\p^a\Box\tilde V_0(s,y)|\ls\ve_1^3\w{y}^{-1}\w{s-|y|}^{-3/2}\w{s+|y|}^{3\ve_2-3/2}.
\end{equation}
On the other hand, the last line in \eqref{eqn:good1} can be controlled by \eqref{BA3} as follows
\begin{equation}\label{loc:energyA3}
\|[\Delta,\chi_{[1,2]}(x)](u\p_{\alpha}u)\|_{L^2(\cK_3)}+\|Q_0(\p_{\alpha}u,u)\|_{L^2(\cK_3)}
\ls\ve_1^2\w{t}^{2\ve_2-2}.
\end{equation}
Combining \eqref{loc:energyA1}-\eqref{loc:energyA3} with \eqref{BA3} and \eqref{good1} finishes the proof of \eqref{loc:energyA}.
\end{proof}

\begin{lemma}
{\bf (Uniform time decay of $\ds\sum_{|a|\le2N-5}\|\p^au\|_{L^2(\cK_R)}$)}
Under the assumptions of Theorem \ref{thm2}, let $u$ be the solution of \eqref{QWE} with \eqref{nonlinear-Chaplygin} and suppose that \eqref{BA1}-\eqref{BA5} hold.
Then one has
\begin{equation}\label{loc:energyB}
\sum_{|a|\le2N-5}\|\p^au\|_{L^2(\cK_R)}\ls(\ve+\ve_1^2)\w{t}^{4\ve_2-1/2}.
\end{equation}
\end{lemma}
\begin{proof}
Applying \eqref{locdt} to \eqref{eqn:good1} with $i\le2N-7$, $\mu=\nu=\ve_2/2$ and \eqref{initial:data} yields
\begin{equation}\label{loc:energyB1}
\begin{split}
\w{t}\|\p^{\le1}\p_t\tilde V_i\|_{L^2(\cK_{R_1})}\ls\ve+\sum_{|a|\le1}\sup_{s\in[0,t]}\w{s}^{1+\ve_2}\|\p^a\Box\tilde V_i(s)\|_{L^2(\cK_3)}\\
+\sum_{|a|\le2}\sup_{(s,y)\in[0,t]\times\overline{\R^2\setminus\cK_2}}\w{y}^{1/2}\cW_{1+\ve_2,1}(s,y)|Z^a\Box\tilde V_i(s,y)|.
\end{split}
\end{equation}
Note that \eqref{Sobo:ineq} and \eqref{energyA} lead to
\begin{equation}\label{loc:energyB2}
\sum_{|a|\le2N-3}|\p Z^au|\ls\ve_1\w{x}^{-1/2}(1+t)^{1/2+\ve_2}.
\end{equation}
It follows from \eqref{QWE}, \eqref{nonlinear-Chaplygin}, \eqref{null:condition}, \eqref{BA1}, \eqref{BA3}, \eqref{eqn:good1}
and \eqref{loc:energyB2} that for $|y|\ge2$,
\begin{equation}\label{loc:energyB3}
\sum_{|a|\le2}|Z^a\Box\tilde V_i(s,y)|\ls\ve_1^3\w{y}^{-1}\w{s-|y|}^{-3/2}(1+s)^{1/2+\ve_2}\w{s+|y|}^{2\ve_2-1/2}.
\end{equation}
On the other hand, the last line in \eqref{eqn:good1} can be treated by \eqref{BA3} and \eqref{energy:dtdx} as follows
\begin{equation}\label{loc:energyB4}
\sum_{|a|\le1}\{\|\p^a[\Delta,\chi_{[1,2]}(x)](\p_{\alpha}\p_t^ju\p_t^ku)\|_{L^2(\cK_3)}+\|\p^aQ_0(\p_{\alpha}\p_t^ju,\p_t^ku)\|_{L^2(\cK_3)}\}
\ls\ve_1^2\w{t}^{2\ve_2-1}.
\end{equation}
Therefore, collecting \eqref{loc:energyB1}, \eqref{loc:energyB3} and \eqref{loc:energyB4} with \eqref{BA3}, \eqref{energy:dtdx}, \eqref{good1} yields
\begin{equation}\label{loc:energyB5}
\sum_{i\le2N-7}\|\p^{\le1}\p_t\p_t^iu\|_{L^2(\cK_{R_1})}\ls(\ve+\ve_1^2)\w{t}^{4\ve_2-1/2}.
\end{equation}
For $j=0,1,2,\cdots,2N-5$, denote $\ds E_j^{loc}(t):=\sum_{k\le j}\|\p_t^k\p_x^{2N-5-j}u\|_{L^2(\cK_{R+j})}$.
Then \eqref{loc:energyA} and \eqref{loc:energyB5} imply
\begin{equation}\label{loc:energyB6}
E_{2N-5}^{loc}(t)+E_{2N-6}^{loc}(t)\ls(\ve+\ve_1^2)\w{t}^{4\ve_2-1/2}.
\end{equation}
When $j\le2N-7$, we have $\p_x^{2N-5-j}=\p_x^2\p_x^{2N-7-j}$.
Thus, one can apply the elliptic estimate \eqref{ellip} to $(1-\chi_{[R+j,R+j+1]})\p_t^{k}u$
to obtain
\begin{equation}\label{loc:energyB7}
\begin{split}
E_j^{loc}(t)&\ls\sum_{k\le j}\|\p_x^{2N-5-j}(1-\chi_{[R+j,R+j+1]})\p_t^ku\|_{L^2(\cK)}\\
&\ls\sum_{k\le j}[\|\Delta(1-\chi_{[R+j,R+j+1]})\p_t^ku\|_{H^{2N-7-j}(\cK)}\\
&\quad+\sum_{k\le j}\|(1-\chi_{[R+j,R+j+1]})\p_t^ku\|_{H^{2N-6-j}(\cK)}]\\
&\ls\sum_{k\le j}\|(1-\chi_{[R+j,R+j+1]})\Delta\p_t^ku\|_{H^{2N-7-j}(\cK)}+\sum_{j+1\le l\le2N-5}E_l^{loc}(t)\\
&\ls\sum_{k\le j}[\|\Box\p_t^ku\|_{H^{2N-7-j}(\cK_{R+j+1})}+\|\p_t^{k+2}u\|_{H^{2N-7-j}(\cK_{R+j+1})}]\\
&\quad+\sum_{j+1\le l\le2N-5}E_l^{loc}(t),
\end{split}
\end{equation}
where the fact of $\Delta=\p_t^2-\Box$ has been used.
From \eqref{QWE}, \eqref{nonlinear-Chaplygin}, \eqref{BA3} and \eqref{energy:dtdx}, we can get that for $j\le2N-7$,
\begin{equation}\label{loc:energyB8}
\sum_{k\le j}\|\Box\p_t^ku\|_{H^{2N-7-j}(\cK_{R+j+1})}\ls\ve_1^2(1+t)^{2\ve_2-1}.
\end{equation}
Thus, \eqref{loc:energyB7} and \eqref{loc:energyB8} conclude that for $j\le2N-7$,
\begin{equation*}
E_j^{loc}(t)\ls\ve_1^2(1+t)^{2\ve_2-1}+\sum_{j+1\le l\le2N-5}E_l^{loc}(t).
\end{equation*}
This, together with \eqref{loc:energyB6}, yields \eqref{loc:energyB}.
\end{proof}

\subsection{Improved energy estimates}

With the local energy decay estimate \eqref{loc:energyB} established in the former subsection,
the energy estimates including all the vector fields in \eqref{energyA} can be actually improved.
\begin{lemma}
{\bf (Slow time growth of $\ds\sum_{|a|\le2N-7}\|\p Z^au\|_{L^2(\cK)}$)}
Under the assumptions of Theorem \ref{thm2}, let $u$ be the solution of \eqref{QWE} with \eqref{nonlinear-Chaplygin}
and suppose that \eqref{BA1}-\eqref{BA5} hold.
Then we have
\begin{equation}\label{energyB}
\sum_{|a|\le2N-7}\|\p Z^au\|_{L^2(\cK)}\ls\ve_1(1+t)^{4\ve_2}.
\end{equation}
\end{lemma}
\begin{proof}
Analogously to \eqref{energyA5}, \eqref{energyA6} and \eqref{energyA7}, one has that for any $|a|\le2N-7$,
\begin{equation}\label{energyB1}
\sum_{|a|\le2N-7}\|\p Z^au\|^2_{L^2(\cK)}\ls\ve^2+\ve_1\int_0^t\sum_{|a|\le2N-7}\frac{\|\p Z^au(s)\|^2_{L^2(\cK)}}{1+s}ds
+|\cB^a_1|+|\cB^a_2|
\end{equation}
and
\begin{equation}\label{energyB2}
\begin{split}
|\cB^a_1|&\ls\sum_{|b|\le|a|+1}\int_0^t\|\p \p^bu\|^2_{L^2(\cK_2)}ds\\
&\ls\ve_1^2\int_0^t(1+s)^{8\ve_2-1}ds\ls\ve_1^2(1+t)^{8\ve_2},
\end{split}
\end{equation}
where \eqref{loc:energyB} has been used. Similarly, $\cB^a_2$ can be also treated.
Then \eqref{energyB1} and \eqref{energyB2} yield
\begin{equation*}
\sum_{|a|\le2N-7}\|\p Z^au\|^2_{L^2(\cK)}\ls\ve_1^2(1+t)^{8\ve_2}+\ve_1\int_0^t\sum_{|a|\le2N-7}\frac{\|\p Z^au(s)\|^2_{L^2(\cK)}}{1+s}ds,
\end{equation*}
which implies \eqref{energyB} by Gronwall's lemma.
\end{proof}

In addition, the local energy decay estimate \eqref{loc:energyB} can be further improved as follows.
\begin{lemma}
{\bf (Uniform time decay of $\ds\sum_{|a|\le2N-11}\|\p^au\|_{L^2(\cK_R)}$)}
Under the assumptions of Theorem \ref{thm2}, let $u$ be the solution of \eqref{QWE} with \eqref{nonlinear-Chaplygin} and suppose that \eqref{BA1}-\eqref{BA5} hold.
Then
\begin{equation}\label{loc:energyC}
\sum_{|a|\le2N-11}\|\p^au\|_{L^2(\cK_R)}\ls(\ve+\ve_1^2)\w{t}^{7\ve_2-1}.
\end{equation}
\end{lemma}
\begin{proof}
The proof is analogous to that of \eqref{loc:energyB}.
At first, \eqref{Sobo:ineq} and \eqref{energyB} imply
\begin{equation}\label{loc:energyC1}
\sum_{|a|\le2N-9}|\p Z^au|\ls\ve_1\w{x}^{-1/2}(1+t)^{4\ve_2}.
\end{equation}
Then, with \eqref{loc:energyC1} instead of \eqref{loc:energyB2}, \eqref{loc:energyB3} can be improved to
\begin{equation*}
\sum_{i\le2N-13}\sum_{|a|\le2}|Z^a\Box\tilde V_i(s,y)|\ls\ve_1^3\w{y}^{-1}\w{s-|y|}^{-3/2}\w{s+|y|}^{6\ve_2-1/2},\quad|y|\ge2.
\end{equation*}
This, together with \eqref{loc:energyB4}, yields
\begin{equation}\label{loc:energyC2}
\sum_{i\le2N-13}\|\p^{\le1}\p_t\p_t^iu\|_{L^2(\cK_{R_1})}\ls(\ve+\ve_1^2)\w{t}^{7\ve_2-1}.
\end{equation}
Finally, \eqref{loc:energyC} can be achieved by \eqref{loc:energyC2} with an argument as in the proof of \eqref{loc:energyB}.
\end{proof}

\section{Improved pointwise estimates and proof of Theorem \ref{thm2}}\label{sect6}

In this section, we will improve the pointwise estimates \eqref{BA1}-\eqref{BA5} with \eqref{loc:energyC1}
and the decay estimates of the local energy established in Section \ref{sect5} so that the
bootstrap assumptions in Section \ref{sect5} can be closed.

Instead of \eqref{energyA2}, we introduce the following good unknowns for $|a|\le2N-1$,
\begin{equation}\label{good2}
\tilde V^a:=\tilde Z^au-\sum_{\alpha=0}^2\sum_{b+c\le a}C^{\alpha}C^a_{bc}\chi_{[1/2,1]}(x)\p_{\alpha}Z^buZ^cu,\quad \tilde Z=\chi_{[1/2,1]}(x)Z.
\end{equation}
Then $\tilde V^a$ verifies $\ds\frac{\p}{\p\bn}\tilde V^a|_{[0,\infty)\times\p\cK}=0$ and
\begin{equation}\label{eqn:good2}
\begin{split}
\Box\tilde V^a&=\Box(\tilde Z^au-Z^au)+\sum_{b+c+d\le a}Q_{abcd}^{\alpha\beta\mu\nu}\p^2_{\alpha\beta}Z^bu\p_{\mu}Z^cu\p_{\nu}Z^du\\
&\quad-\sum_{\alpha=0}^2\sum_{b+c\le a}C^{\alpha}C^a_{bc}\chi_{[1/2,1]}(x)[Z^cu\Box\p_{\alpha}Z^bu+\p_{\alpha}Z^bu\Box Z^cu]\\
&\quad+\sum_{\alpha=0}^2\sum_{b+c\le a}C^{\alpha}C^a_{bc}\{[\Delta,\chi_{[1/2,1]}](\p_{\alpha}Z^buZ^cu)
-2(1-\chi_{[1/2,1]})Q_0(\p_{\alpha}Z^bu,Z^cu)\}.
\end{split}
\end{equation}

\subsection{Decay estimates on the good derivatives of solutions}

It is pointed out that although the conclusions are similar to the corresponding ones in Section 7.1
of \cite{HouYinYuan24}, we still give the details for readers' convenience and due to both the slightly
different pointwise estimates and the crucial form \eqref{eqn:good2} together with the Neumann boundary condition.

At first, we derive the pointwise estimates of $Z^au$ for $|a|\le2N-15$ and the solution
$u$ of \eqref{QWE} with \eqref{nonlinear-Chaplygin}.
\begin{lemma}\label{lem:BA3:A}
{\bf (Pointwise estimate of $\ds\sum_{|a|\le2N-15}|Z^au|$)}
Under the assumptions of Theorem \ref{thm2}, let $u$ be the solution of \eqref{QWE} with \eqref{nonlinear-Chaplygin}
and suppose that \eqref{BA1}-\eqref{BA5} hold.
Then we have
\begin{equation}\label{BA3:impr:A}
\sum_{|a|\le2N-15}|Z^au|\ls(\ve+\ve_1^2)\w{t+|x|}^{6\ve_2}.
\end{equation}
\end{lemma}
\begin{proof}
Due to the lack of the estimates of $Z^au$ with higher order derivatives appeared in \eqref{eqn:good2}, we will adopt the original
equation \eqref{eqn:high}.
Applying \eqref{pw:ibvp} to $\Box\tilde Z^au=\Box(\tilde Z^a-Z^a)u+\Box Z^au$ with $\mu=\nu=\ve_2/2$ and $|a|\le2N-15$ yields
\begin{equation}\label{BA3:impr:A1}
\begin{split}
&\w{t+|x|}^{1/2}|\tilde Z^au|\ls\ve+\sum_{|b|\le4}\sup_{s\in[0,t]}\w{s}^{1/2+\ve_2}\|\p^b\Box(\tilde Z^a-Z^a)u(s)\|_{L^2(\cK_4)}\\
&\qquad+\sum_{|b|\le5}\sup_{(s,y)\in[0,t]\times(\overline{\R^2\setminus\cK_2})}\w{y}^{1/2}\cW_{1+\ve_2,1}(s,y)|\p^b\Box Z^au(s,y)|,
\end{split}
\end{equation}
where we have used \eqref{initial:data} and the fact of $\tilde Z=Z$ in the region $|x|\ge1$.
In addition, it can be deduced from \eqref{loc:energyB} that
\begin{equation}\label{BA3:impr:A2}
\sum_{|a|\le2N-15}\sum_{|b|\le4}\sup_{s\in[0,t]}\w{s}^{1/2+\ve_2}\|\p^b\Box(\tilde Z^a-Z^a)u(s)\|_{L^2(\cK_4)}
\ls(\ve+\ve_1^2)(1+t)^{5\ve_2}.
\end{equation}
For the second line of \eqref{BA3:impr:A1}, one can conclude from \eqref{eqn:high}, \eqref{BA1} and \eqref{loc:energyC1} that
\begin{equation}\label{BA3:impr:A3}
\sum_{|a|\le2N-15}\sum_{|b|\le5}|Z^b\Box Z^au(s,y)|
\ls\ve_1^2\w{y}^{-1}\w{s-|y|}^{-1}(1+s)^{5\ve_2}.
\end{equation}
Substituting \eqref{BA3:impr:A2} and \eqref{BA3:impr:A3} into \eqref{BA3:impr:A1} leads to
\begin{equation*}
\sum_{|a|\le2N-15}|\tilde Z^au|\ls(\ve+\ve_1^2)\w{t+|x|}^{6\ve_2}.
\end{equation*}
This, together with \eqref{loc:energyB} and the Sobolev embedding, finishes the proof of \eqref{BA3:impr:A}.
\end{proof}

\begin{lemma}\label{lem:BA2:A}
{\bf (Pointwise estimate of $\ds\sum_{|a|\le2N-17}|\bar\p Z^au|$ near the light conic surface)}
Under the assumptions of Theorem \ref{thm2}, let $u$ be the solution of \eqref{QWE} with \eqref{nonlinear-Chaplygin}
and suppose that \eqref{BA1}-\eqref{BA5} hold.
Then one has that for $|x|\ge1+t/2$,
\begin{equation}\label{BA2:impr:A}
\sum_{|a|\le2N-17}|\bar\p Z^au|\ls(\ve+\ve_1^2)\w{t+|x|}^{6\ve_2-1}.
\end{equation}
\end{lemma}
\begin{proof}
Set $\p_{\pm}:=\p_t\pm\p_r$ and $r=|x|$. Then
\begin{equation}\label{BA2:impr:A1}
\begin{split}
&\bar\p_1=\frac{x_1}{r}\p_+-\frac{x_2}{r^2}\Omega,
\quad\bar\p_2=\frac{x_2}{r}\p_++\frac{x_1}{r^2}\Omega,\quad r=|x|,\\
&\p_+(r^{1/2}w)(t,r\frac{x}{|x|})-\p_+(r^{1/2}w)(0,(r+t)\frac{x}{|x|})\\
&\quad =\int_0^t\{(r+t-s)^{1/2}\Box w+(r+t-s)^{-3/2}(w/4+\Omega^2w)\}(s,(r+t-s)\frac{x}{|x|})ds,
\end{split}
\end{equation}
one can see Section 4.6 in \cite{Kubo13}
or Section 7 in \cite{HouYinYuan24} for detailed computation.
By choosing $w=Z^au$ with $|a|\le2N-17$ in \eqref{BA2:impr:A1}, it follows from \eqref{BA3:impr:A3} that for $|y|\ge1+s/2$,
\begin{equation}\label{BA2:impr:A2}
|\Box Z^au(s,y)|\ls\ve_1^2\w{s+|y|}^{5\ve_2-1}\w{s-|y|}^{-1}.
\end{equation}
Substituting \eqref{initial:data}, \eqref{BA3:impr:A} and \eqref{BA2:impr:A2} into \eqref{BA2:impr:A1} yields that for $|x|\ge1+t/2$,
\begin{equation*}
\begin{split}
|\p_+(r^{1/2}Z^au)(t,x)|&\ls\ve\w{x}^{-1}+\ve_1^2(r+t)^{5\ve_2-1/2}\int_0^t(1+|r+t-2s|)^{-1}ds\\
&\quad+(\ve+\ve_1^2)\int_0^t(r+t-s)^{-3/2}\w{r+t}^{6\ve_2}ds\\
&\ls(\ve+\ve_1^2)\w{t+|x|}^{6\ve_2-1/2},
\end{split}
\end{equation*}
which leads to
\begin{equation*}
\begin{split}
|\p_+Z^au|&\ls r^{-1/2}|\p_+(r^{1/2}Z^au)(t,x)|+r^{-1}|Z^au(t,x)|\\
&\ls(\ve+\ve_1^2)\w{t+|x|}^{6\ve_2-1}.
\end{split}
\end{equation*}
This, together with \eqref{BA3:impr:A} and \eqref{BA2:impr:A1}, completes the proof \eqref{BA2:impr:A}.
\end{proof}

Based on Lemmas \ref{lem:BA3:A}-\ref{lem:BA2:A}, we next establish the uniform spacetime decay estimates
of $Z^au$ for $|a|\le2N-24$ in terms of \eqref{eqn:good2}.
\begin{lemma}\label{YHCC-52}
{\bf (Uniform spacetime decay estimate of $\ds\sum_{|a|\le2N-24}|Z^au|$)}
Under the assumptions of Theorem \ref{thm2}, let $u$ be the solution of \eqref{QWE} with \eqref{nonlinear-Chaplygin}
and suppose that \eqref{BA1}-\eqref{BA5} hold.
Then one has
\begin{equation}\label{BA3:impr:B}
\sum_{|a|\le2N-24}|Z^au|\ls(\ve+\ve_1^2)\w{t+|x|}^{-1/2}.
\end{equation}
\end{lemma}
\begin{proof}
Applying \eqref{pw:ibvp} to \eqref{eqn:good2} for $|a|\le2N-24$, $\mu=\nu=\ve_2/2$ yields
\begin{equation}\label{BA3:impr:B1}
\begin{split}
&\w{t+|x|}^{1/2}|\tilde V^a|\ls\ve+\sum_{|b|\le4}\sup_{s\in[0,t]}\w{s}^{1/2+\ve_2}\|\p^b\Box\tilde V^a(s)\|_{L^2(\cK_4)}\\
&\qquad+\sum_{|b|\le5}\sup_{(s,y)\in[0,t]\times(\overline{\R^2\setminus\cK_2})}\w{y}^{1/2}\cW_{1+\ve_2,1}(s,y)|Z^b\Box\tilde V^a(s,y)|,
\end{split}
\end{equation}
where we have used \eqref{initial:data}.
Analogously to \eqref{BA3:impr:A2}, we can obtain from \eqref{loc:energyC} that
\begin{equation}\label{BA3:impr:B2}
\sum_{|a|\le2N-24}\sum_{|b|\le4}\sup_{s\in[0,t]}\w{s}^{1-7\ve_2}\|\p^b\Box\tilde V^a(s)\|_{L^2(\cK_4)}\ls\ve+\ve_1^2.
\end{equation}
Note that $\Box(\tilde Z^au-Z^au)$ in the first line of \eqref{eqn:good2} vanishes in the region $|y|\ge1$.
Next, we treat the second line of \eqref{BA3:impr:B1} and focus on the terms $Z^cu\Box\p_{\alpha}Z^bu$ with $|b|+|c|\le2N-19$
in the second line of \eqref{eqn:good2}.
When $|c|\ge N+1$, one has $|b|\le N-20$.
By Lemmas \ref{lem:null:structure}-\ref{lem:eqn:high}, \eqref{BA1}, \eqref{BA2} and \eqref{BA3:impr:A}, we arrive at
\begin{equation}\label{BA3:impr:B3}
\sum_{|c|\ge N+1}\sum_{|b|+|c|\le2N-19}|Z^cu\Box\p_{\alpha}Z^bu|\ls\ve_1^3\w{y}^{-1}\w{s-|y|}^{-1}\w{s+|y|}^{8\ve_2-1}.
\end{equation}
When $|c|\le N$, in the region $\cK\cap\{|y|\le3+s/2\}\supseteq{\cK}_3$, it follows from \eqref{eqn:high}, \eqref{BA1}, \eqref{BA3} and \eqref{loc:energyC1} that
\begin{equation}\label{BA3:impr:B4}
\sum_{|b|+|c|\le2N-19}|Z^cu\Box\p_{\alpha}Z^bu|\ls\ve_1^3\w{y}^{-1}\w{s}^{6\ve_2-2}.
\end{equation}
In the region $\cK\cap\{y:|y|\ge3+s/2\}$, by virtue of Lemmas \ref{lem:null:structure}-\ref{lem:eqn:high}, \eqref{BA1},
\eqref{BA2}, \eqref{loc:energyC1} and \eqref{BA2:impr:A}, one obtains
\begin{equation*}
\begin{split}
\sum_{|b|\le2N-19}|\Box\p_{\alpha}Z^bu|&\ls\sum_{|b'|+|c'|\le2N-17}|\bar\p Z^{b'}u||\p Z^{c'}u|\\
&\ls\ve_1^2\w{y}^{7\ve_2-3/2}\w{s-|y|}^{-1}+\ve_1^2\w{y}^{-1}\w{s+|y|}^{5\ve_2-1}\\
&\ls\ve_1^2\w{y}^{7\ve_2-3/2}\w{s-|y|}^{-1/2},
\end{split}
\end{equation*}
which together with \eqref{BA3} yields
\begin{equation}\label{BA3:impr:B5}
\sum_{|b|+|c|\le2N-19}|Z^cu\Box\p_tZ^bu|\ls\ve_1^3\w{y}^{8\ve_2-2}\w{s-|y|}^{-1}.
\end{equation}
Analogously, we can achieve
\begin{equation}\label{BA3:impr:B6}
\sum_{|a|\le2N-24}\sum_{|b|\le5}|Z^b\Box \tilde V^a(s,y)|\ls\ve_1^3\w{y}^{-1}\w{s-|y|}^{-1}\w{s+|y|}^{8\ve_2-1},\quad |y|\ge1.
\end{equation}
Substituting \eqref{BA3:impr:B2} and \eqref{BA3:impr:B6} into \eqref{BA3:impr:B1} leads to
\begin{equation}\label{BA3:impr:B7}
\sum_{|a|\le2N-24}|\tilde V^a|\ls(\ve+\ve_1^2)\w{t+|x|}^{-1/2}.
\end{equation}
Then it follows from \eqref{BA3}, \eqref{loc:energyC}, \eqref{loc:energyC1}, \eqref{good2} and \eqref{BA3:impr:B7} that
\begin{equation*}
\begin{split}
\sum_{|a|\le2N-24}|Z^au|&\ls(\ve+\ve_1^2)\w{t+|x|}^{-1/2}+\ve_1^2\w{t}^{4\ve_2}\w{x}^{-1/2}\w{t+|x|}^{-1/2}\w{t-|x|}^{\ve_2-1/2}\\
&\ls(\ve+\ve_1^2)\w{t+|x|}^{-1/2}.
\end{split}
\end{equation*}
This completes the proof of \eqref{BA3:impr:B}.
\end{proof}

\begin{lemma}
{\bf (Uniform spacetime decay estimate of $\ds\sum_{|a|\le2N-27}|\bar\p Z^au|$  near the light conic surface)}
Under the assumptions of Theorem \ref{thm2}, let $u$ be the solution of \eqref{QWE} with \eqref{nonlinear-Chaplygin}
and suppose that \eqref{BA1}-\eqref{BA5} hold.
Then one has that for $|x|\ge1+t/2$,
\begin{equation}\label{BA2:impr:B}
\sum_{|a|\le2N-27}|\bar\p Z^au|\ls(\ve+\ve_1^2)\w{t+|x|}^{-3/2}.
\end{equation}
\end{lemma}
\begin{proof}
By choosing $w=\tilde V^a$ with $|a|\le2N-26$ in \eqref{BA2:impr:A1}, it follows from \eqref{BA3:impr:B6} that for $|y|\ge1+s/2$,
\begin{equation}\label{BA2:impr:B1}
|\Box\tilde V^a(s,y)|\ls\ve_1^2\w{s+|y|}^{8\ve_2-2}\w{s-|y|}^{-1}.
\end{equation}
Substituting \eqref{initial:data}, \eqref{BA3:impr:B} and \eqref{BA2:impr:B1} into \eqref{BA2:impr:A1} yields that for $|x|\ge1+t/2$,
\begin{equation*}
\begin{split}
|\p_+(r^{1/2}\tilde V^a)(t,x)|&\ls\ve\w{x}^{-1}+\ve_1^2(r+t)^{8\ve_2-3/2}\int_0^t(1+|r+t-2s|)^{-1}ds\\
&\quad+(\ve+\ve_1^2)\int_0^t(r+t-s)^{-3/2}\w{r+t}^{-1/2}ds\\
&\ls(\ve+\ve_1^2)\w{t+|x|}^{-1}.
\end{split}
\end{equation*}
This, together with  \eqref{BA3:impr:B7}, leads to
\begin{equation}\label{YHCC-51}
\begin{split}
|\p_+\tilde V^a|&\ls r^{-1/2}|\p_+(r^{1/2}\tilde V^a)(t,x)|+r^{-1}|\tilde V^a(t,x)|\\
&\ls(\ve+\ve_1^2)\w{t+|x|}^{-3/2}.
\end{split}
\end{equation}
By \eqref{good2}, \eqref{BA2:impr:A1}, \eqref{BA3:impr:B} and \eqref{YHCC-51}, we have
\begin{equation}\label{BA2:impr:B2}
\sum_{|a|\le2N-26}|\bar\p\tilde V^a|\ls(\ve+\ve_1^2)\w{t+|x|}^{-3/2}.
\end{equation}
In addition, from \eqref{loc:energyC1}, \eqref{good2}, \eqref{BA2:impr:A}, \eqref{BA3:impr:B} and \eqref{BA2:impr:B2},
one has
\begin{equation}\label{BA2:impr:B3}
\begin{split}
\sum_{|a|\le2N-26}|\bar\p Z^au|
&\ls(\ve+\ve_1^2)\w{t+|x|}^{-3/2}+\sum_{|b|+|c|\le2N-26}(|\bar\p Z^bu||\p Z^cu|\\
&\quad+|\bar\p\p Z^bu||Z^cu|+|x|^{-1}|\p Z^bu||Z^cu|)\\
&\ls(\ve+\ve_1^2)\w{t+|x|}^{6\ve_2-3/2}.
\end{split}
\end{equation}
Repeating the estimate in \eqref{BA2:impr:B3} yields
\begin{equation*}
\begin{split}
\sum_{|a|\le2N-27}|\bar\p Z^au|
&\ls(\ve+\ve_1^2)\w{t+|x|}^{-3/2}+\sum_{|b|+|c|\le2N-27}(|\bar\p Z^bu||\p Z^cu|\\
&\quad+|\bar\p\p Z^bu||Z^cu|+|x|^{-1}|\p Z^bu||Z^cu|)\\
&\ls(\ve+\ve_1^2)\w{t+|x|}^{-3/2}.
\end{split}
\end{equation*}
This completes the proof of \eqref{BA2:impr:B}.
\end{proof}

\subsection{Improved pointwise estimates}

Under the preparations above, we now start to establish more precise spacetime or time decay for some important quantities
in order to close the bootstrap assumptions \eqref{BA1}-\eqref{BA5}  and prove Theorem \ref{thm2}.
\begin{lemma}
{\bf (Precise uniform spacetime decay estimate of $\ds\sum_{|a|\le2N-28}|\p Z^au|$)}
Under the assumptions of Theorem \ref{thm2}, let $u$ be the solution of \eqref{QWE} with \eqref{nonlinear-Chaplygin}
and suppose that \eqref{BA1}-\eqref{BA5} hold.
Then we have
\begin{equation}\label{BA1:impr:A}
\sum_{|a|\le2N-28}|\p Z^au|\ls(\ve+\ve_1^2)\w{x}^{-1/2}\w{t-|x|}^{-1}\w{t}^{9\ve_2}.
\end{equation}
\end{lemma}
\begin{proof}
By using \eqref{dpw:ibvp} to \eqref{eqn:good2} with $G^{\alpha}=0$, $\eta=\ve_2/2$ and \eqref{initial:data}, one can obtain
\begin{equation}\label{BA1:impr:A1}
\begin{split}
&\w{x}^{1/2}\w{t-|x|}|\p\tilde V^a|\ls\ve\ln(2+t)
+\ln(2+t)\sum_{|b|\le8}\sup_{s\in[0,t]}\w{s}\|\p^b\Box\tilde V^a(s)\|_{L^2(\cK_3)}\\
&\quad+\ln(2+t)\sum_{|b|\le9}\sup_{(s,y)\in[0,t]\times\overline{\R^2\setminus\cK_2}}\w{y}^{1/2}\cW_{3/2+\ve_2/2,1}(s,y)|Z^b\Box\tilde V^a(s,y)|.
\end{split}
\end{equation}
Analogously to \eqref{BA3:impr:B2} and \eqref{BA3:impr:B6}, we have
\begin{equation}\label{BA1:impr:A2}
\begin{split}
&\sum_{|a|\le2N-28}\sum_{|b|\le8}\sup_{s\in[0,t]}\w{s}\|\p^b\Box\tilde V^a(s)\|_{L^2(\cK_4)}\ls(\ve+\ve_1^2)\w{t}^{7\ve_2},\\
&\sum_{|a|\le2N-28}\sum_{|b|\le9}|Z^b\Box \tilde V^a(s,y)|\ls\ve_1^3\w{y}^{-1}\w{s-|y|}^{-1}\w{s+|y|}^{8\ve_2-1}.
\end{split}
\end{equation}
Collecting \eqref{BA1:impr:A1} and \eqref{BA1:impr:A2} yields
\begin{equation}\label{BA1:impr:A3}
|\p\tilde V^a|\ls(\ve+\ve_1^2)\w{x}^{-1/2}\w{t-|x|}^{-1}\w{t}^{9\ve_2}.
\end{equation}
Thus, \eqref{BA1:impr:A} can be achieved by the definition \eqref{good2} together with \eqref{BA1}, \eqref{BA3}, \eqref{loc:energyC}, \eqref{loc:energyC1}, \eqref{BA3:impr:B} and \eqref{BA1:impr:A3}.
\end{proof}

\begin{lemma}
{\bf (Precise time decay estimate of local energy near boundary)}
Under the assumptions of Theorem \ref{thm2}, let $u$ be the solution of \eqref{QWE} with \eqref{nonlinear-Chaplygin}
and suppose that \eqref{BA1}-\eqref{BA5} hold.
Then one has
\begin{equation}\label{loc:energyD}
\begin{split}
\sum_{|a|\le2N-29}\|\p^au\|_{L^2(\cK_R)}&\ls(\ve+\ve_1^2)\w{t}^{-1}\ln(2+t),\\
\sum_{|a|\le2N-30}\|\p^a\p_tu\|_{L^2(\cK_R)}&\ls(\ve+\ve_1^2)\w{t}^{-1}.
\end{split}
\end{equation}
\end{lemma}
\begin{proof}
At first, \eqref{BA2:impr:B} and \eqref{BA1:impr:A} ensure that
\begin{equation}\label{loc:energyD1}
\sum_{|a|\le2N-28}|\bar\p Z^au|\ls(\ve+\ve_1^2)\w{x}^{-1/2}\w{t+|x|}^{9\ve_2-1}.
\end{equation}
Instead of \eqref{eqn:good1}, applying \eqref{locdt} to $\Box\p_t^iu$ with $i\le2N-31$, $\mu=\nu=\ve_2$ and \eqref{initial:data} derives
\begin{equation}\label{loc:energyD2}
\begin{split}
\w{t}\|\p^{\le1}\p_t^{i+1}u\|_{L^2(\cK_{R_1})}\ls\ve+\sum_{|a|\le1}\sup_{s\in[0,t]}\w{s}^{1+\ve_2}\|\p^a\Box\p_t^iu(s)\|_{L^2(\cK_3)}\\
+\sum_{|a|\le2}\sup_{(s,y)\in[0,t]\times\cK}\w{y}^{1/2}\cW_{1+2\ve_2,1}(s,y)|Z^a\Box\p_t^iu(s,y)|.
\end{split}
\end{equation}
It follows from \eqref{QWE}, \eqref{nonlinear-Chaplygin}, \eqref{null:condition}, \eqref{null:structure}, \eqref{BA1:impr:A} and \eqref{loc:energyD1} that
\begin{equation}\label{loc:energyD3}
\sum_{i\le2N-31}\sum_{|a|\le2}|Z^a\Box\p_t^iu(s,y)|\ls\ve_1^2\w{y}^{-1}\w{s-|y|}^{-1}\w{s+|y|}^{18\ve_2-1}.
\end{equation}
Substituting \eqref{loc:energyD3} into \eqref{loc:energyD2} yields
\begin{equation}\label{loc:energyD4}
\sum_{i\le2N-31}\|\p^{\le1}\p_t^{i+1}u\|_{L^2(\cK_{R_1})}\ls(\ve+\ve_1^2)\w{t}^{-1}.
\end{equation}
Then \eqref{loc:energyD} can be achieved by \eqref{loc:energyA}, \eqref{loc:energyD4} with the same method as
in the proof of \eqref{loc:energyB}.
\end{proof}

\begin{lemma}
{\bf (Precise uniform spacetime decay estimate of $\ds\sum_{|a|\le2N-35}|Z^au|$)}
Under the assumptions of Theorem \ref{thm2}, let $u$ be the solution of \eqref{QWE} with \eqref{nonlinear-Chaplygin}
and suppose that \eqref{BA1}-\eqref{BA5} hold.
Then we have
\begin{equation}\label{BA3:impr:C}
\sum_{|a|\le2N-35}|Z^au|\ls(\ve+\ve_1^2)\w{t+|x|}^{-1/2}\w{t-|x|}^{\ve_2-1/2}.
\end{equation}
\end{lemma}
\begin{proof}
Applying \eqref{pw:ibvp} to \eqref{eqn:good2} for $|a|\le2N-35$, $\mu=1/2-\ve_2$, $\nu=\ve_2/2$ derives
\begin{equation}\label{BA3:impr:C1}
\begin{split}
&\w{t+|x|}^{1/2}\w{t-|x|}^{1/2-\ve_2}|\tilde V^a|\ls\ve+\sum_{|b|\le4}\sup_{s\in[0,t]}\w{s}^{1-\ve_2/2}\|\p^b\Box\tilde V^a(s)\|_{L^2(\cK_4)}\\
&\qquad+\sum_{|b|\le5}\sup_{(s,y)\in[0,t]\times(\overline{\R^2\setminus\cK_2})}\w{y}^{1/2}\cW_{3/2-\ve_2/2,1}(s,y)|Z^b\Box\tilde V^a(s,y)|.
\end{split}
\end{equation}
Then, by \eqref{eqn:good2} and \eqref{loc:energyD}, we arrive at
\begin{equation}\label{BA3:impr:C2}
\sum_{|a|\le2N-35}\sum_{|b|\le4}\sup_{s\in[0,t]}\w{s}^{1-\ve_2/2}\|\p^b\Box\tilde V^a(s)\|_{L^2(\cK_4)}\ls\ve+\ve_1^2.
\end{equation}
On the other hand, it can be concluded from Lemmas \ref{lem:null:structure}-\ref{lem:eqn:high}, \eqref{eqn:good2}, \eqref{BA3:impr:B}, \eqref{BA1:impr:A} and \eqref{loc:energyD1} that
\begin{equation}\label{BA3:impr:C3}
\sum_{|a|\le2N-35}\sum_{|b|\le5}|Z^b\Box\tilde V^a(s,y)|\ls\ve_1^3\w{y}^{-1}\w{s-|y|}^{-1}\w{s+|y|}^{18\ve_2-3/2},\quad |y|\ge1.
\end{equation}
Plugging \eqref{BA3:impr:C2} and \eqref{BA3:impr:C3} into \eqref{BA3:impr:C1} leads to
\begin{equation}\label{BA3:impr:C4}
\sum_{|a|\le2N-35}|\tilde V^a|\ls(\ve+\ve_1^2)\w{t+|x|}^{-1/2}\w{t-|x|}^{\ve_2-1/2}.
\end{equation}
Therefore, \eqref{BA3:impr:C} is obtained by \eqref{BA3}, \eqref{good2}, \eqref{BA1:impr:A},
\eqref{loc:energyD} and \eqref{BA3:impr:C4}.
\end{proof}

\begin{lemma}
{\bf (Precise uniform spacetime decay estimates of $\ds\sum_{|a|\le2N-39}|\p Z^au|$ and $\ds\sum_{|a|\le2N-39}|\bar\p Z^au|$)}
Under the assumptions of Theorem \ref{thm2}, let $u$ be the solution of \eqref{QWE} with \eqref{nonlinear-Chaplygin}
and suppose that \eqref{BA1}-\eqref{BA5} hold. Then 
\begin{align}
\sum_{|a|\le2N-39}|\p Z^au|\ls(\ve+\ve_1^2)\w{x}^{-1/2}\w{t-|x|}^{-1}\ln^2(2+t),\label{BA1:impr:B}\\
\sum_{|a|\le2N-39}|\bar\p Z^au|\ls(\ve+\ve_1^2)\w{x}^{-1/2}\w{t+|x|}^{\ve_2-1}.\label{BA2:impr:C}
\end{align}
\end{lemma}
\begin{proof}
Similarly to \eqref{BA3:impr:C2} and \eqref{BA3:impr:C3}, instead of \eqref{BA1:impr:A2}, we can get
\begin{equation}\label{BA1:impr:B1}
\begin{split}
\sum_{|a|\le2N-39}\sum_{|b|\le8}\sup_{s\in[0,t]}\w{s}\|\p^b\Box\tilde V^a(s)\|_{L^2(\cK_4)}\ls(\ve+\ve_1^2)\ln(2+t),\\
\sum_{|a|\le2N-39}\sum_{|b|\le9}|Z^b\Box \tilde V^a(s,y)|\ls\ve_1^3\w{y}^{-1}\w{s-|y|}^{-1}\w{s+|y|}^{18\ve_2-3/2},\quad |y|\ge1.
\end{split}
\end{equation}
Substituting \eqref{BA1:impr:B1} into \eqref{BA1:impr:A1} gives
\begin{equation*}
\sum_{|a|\le2N-39}|\p\tilde V^a|\ls(\ve+\ve_1^2)\w{x}^{-1/2}\w{t-|x|}^{-1}\ln^2(2+t).
\end{equation*}
This, together with \eqref{good2}, \eqref{loc:energyD} and \eqref{BA3:impr:C}, yields \eqref{BA1:impr:B}.
At last, \eqref{BA2:impr:C} can be obtained by \eqref{BA2:impr:B} and \eqref{BA1:impr:B}.
\end{proof}

\begin{lemma}
{\bf (Precise uniform spacetime decay estimate of $\ds\p\p_tu$)}
Under the assumptions of Theorem \ref{thm2}, let $u$ be the solution of \eqref{QWE} with \eqref{nonlinear-Chaplygin}
and suppose that \eqref{BA1}-\eqref{BA5} hold.
Then one has
\begin{equation}\label{BA5:impr}
|\p\p_tu|\ls(\ve+\ve_1^2)\w{x}^{-1/2}\w{t-|x|}^{-1}.
\end{equation}
\end{lemma}
\begin{proof}
Applying $\p_t$ to \eqref{QWE} yields $\Box\p_tu=\p_tQ(\p u,\p^2u)$ with $\frac{\p}{\p\bn}\p_tu|_{[0,\infty)\times\p\cK}=0$.
From \eqref{initial:data} and \eqref{dtpw:ibvp} with $\eta=\ve_2$, we have
\begin{equation*}
\begin{split}
&\w{x}^{1/2}\w{t-|x|}|\p\p_tu|\ls\ve+\sum_{|a|\le9}\sup_{s\in[0,t]}\w{s}^{1+\ve_2}\|\p^aQ(\p u,\p^2u)(s)\|_{L^2(\cK_3)}\\
&\quad+\sum_{|a|\le10}\sup_{(s,y)\in[0,t]\times\overline{\R^2\setminus\cK_2}}\w{y}^{1/2}\cW_{1+\ve_2,1}(s,y)|Z^aQ(\p u,\p^2u)(s,y)|.
\end{split}
\end{equation*}
This, together with \eqref{QWE}, \eqref{nonlinear-Chaplygin}, \eqref{null:condition}, \eqref{null:structure}, \eqref{BA1:impr:A} and \eqref{loc:energyD1}, derives \eqref{BA5:impr}.
\end{proof}

\subsection{Proof of Theorem \ref{thm2}}
\begin{proof}[Proof of Theorem \ref{thm2}]
Collecting \eqref{loc:energyD}, \eqref{BA3:impr:C}, \eqref{BA1:impr:B}, \eqref{BA2:impr:C} and \eqref{BA5:impr} yields that there is $C_4\ge1$ such that
\begin{align*}
\sum_{|a|\le2N-39}|\p Z^au|&\le C_4(\ve+\ve_1^2)\w{x}^{-1/2}\w{t-|x|}^{-1}\w{t}^{\ve_2},\\
\sum_{|a|\le2N-39}|\bar\p Z^au|&\le C_4(\ve+\ve_1^2)\w{x}^{-1/2}\w{t+|x|}^{\ve_2-1},\\
\sum_{|a|\le2N-35}|Z^au|&\le C_4(\ve+\ve_1^2)\w{t+|x|}^{-1/2}\w{t-|x|}^{\ve_2-1/2},\\
\sum_{|a|\le2N-30}\|\p^a\p_tu\|_{L^2(\cK_R)}&\le C_4(\ve+\ve_1^2)\w{t}^{-1},\\
|\p\p_tu|&\le C_4(\ve+\ve_1^2)\w{x}^{-1/2}\w{t-|x|}^{-1}.
\end{align*}
Choosing $\ve_1=4C_4\ve$ and $\ve_0=\frac{1}{16C_4^2}$. Then for $N\ge40$, \eqref{BA1}-\eqref{BA5} can be improved to
\begin{align*}
\sum_{|a|\le N+1}|\p Z^au|&\le\frac{\ve_1}{2}\w{x}^{-1/2}\w{t-|x|}^{-1}\w{t}^{\ve_2},\\
\sum_{|a|\le N+1}|\bar\p Z^au|&\le\frac{\ve_1}{2}\w{x}^{-1/2}\w{t+|x|}^{\ve_2-1},\\
\sum_{|a|\le N+1}|Z^au|&\le\frac{\ve_1}{2}\w{t+|x|}^{-1/2}\w{t-|x|}^{\ve_2-1/2},\\
\sum_{|a|\le N+2}\|\p^a\p_tu\|_{L^2(\cK_R)}&\le\frac{\ve_1}{2}\w{t}^{-1},\\
|\p\p_tu|&\le\frac{\ve_1}{2}\w{x}^{-1/2}\w{t-|x|}^{-1}.
\end{align*}
This, together with the local existence of classical solution to the initial boundary value problem
of the hyperbolic equation  under the admissible condition \eqref{YHCC-2} (see
Section 6 of \cite{Li-Yin18}), yields that problem \eqref{QWE} with \eqref{nonlinear-Chaplygin} has a global solution $u\in\bigcap\limits_{j=0}^{2N+1}C^{j}([0,\infty), H^{2N+1-j}(\cK))$.
In addition, \eqref{thm2:decay:a}, \eqref{thm2:decay:b}, \eqref{thm2:decay:c}, \eqref{thm2:decay:d}, \eqref{thm2:decay:LE} and \eqref{thm2:decay:LEdt} can be obtained by \eqref{BA1:impr:B}, \eqref{BA2:impr:C}, \eqref{BA3:impr:C}, \eqref{BA5:impr} and \eqref{loc:energyD}, respectively.
\end{proof}


\vskip 0.2 true cm

{\bf \color{blue}{Conflict of Interest Statement:}}

\vskip 0.1 true cm

{\bf The authors declare that there is no conflict of interest in relation to this article.}

\vskip 0.2 true cm
{\bf \color{blue}{Data availability statement:}}

\vskip 0.1 true cm

{\bf  Data sharing is not applicable to this article as no data sets are generated
during the current study.}

\end{document}